\documentclass[11pt,a4paper,reqno]{amsart}
\usepackage{amsfonts}
\usepackage{amsmath,amssymb,amsthm,amsxtra}
\usepackage{float}
\usepackage{mathabx}
\usepackage[colorlinks, linkcolor=RoyalBlue,anchorcolor=Periwinkle,
citecolor=Orange,urlcolor=Emerald]{hyperref}
\usepackage[usenames,dvipsnames]{xcolor}
\usepackage{appendix}
\usepackage{graphicx}
\usepackage{subfigure}
\usepackage{bookmark}
\usepackage{mathtools}
\usepackage{extarrows}
\usepackage{enumerate}
\usepackage{xparse}
\usepackage{tikz}\usetikzlibrary{matrix}
\usepackage{url}
\usepackage{dsfont}
\usepackage[capitalise]{cleveref}
\usepackage[colorinlistoftodos]{todonotes}
\usepackage{amscd}
\usepackage[all]{xy}
\usepackage{rotating}
\usepackage{tikz-cd}
\usepackage{leftidx}
\usepackage{bm}
\usepackage{dynkin-diagrams}
\usepackage{mathrsfs}
\usepackage{hyperref}
\usepackage{rotating}
\usetikzlibrary{patterns}
\usepackage{csquotes}
\usepackage{textcomp}
\usepackage{wasysym}
\usepackage{geometry,array}
\usepackage[numbers,sort&compress]{natbib}
\usetikzlibrary{decorations.markings}
\tikzset{->-/.style={decoration={  markings,  mark=at position #1 with
			{\arrow{>}}},postaction={decorate}}}
\tikzset{-<-/.style={decoration={  markings,  mark=at position #1 with
			{\arrow{<}}},postaction={decorate}}}

\theoremstyle{plain}
\newtheorem{theorem}{Theorem}[section]

\newtheorem{lemma}[theorem]{Lemma}
\newtheorem{corollary}[theorem]{Corollary}
\newtheorem{proposition}[theorem]{Proposition}
\newtheorem{conjecture}[theorem]{Conjecture}

\theoremstyle{definition}
\newtheorem{definition}[theorem]{Definition}

\newtheorem{example}[theorem]{Example}

\newtheorem{remark}[theorem]{Remark}

\numberwithin{equation}{section}
\newtheorem{definition-proposition}[theorem]{Definition-Proposition}

\providecommand{\add}{\mathop{\rm add}\nolimits}%
\providecommand{\per}{\mathop{\rm per}\nolimits}%
\providecommand{\thick}{\mathop{\rm thick}\nolimits}%

\numberwithin{equation}{section}%new added
\newcommand{\Hom}{\operatorname{Hom}}

\newcommand{\im}{\operatorname{Im}}     %%%%new add

\newcommand{\pvd}{\operatorname{pvd}}

\newcommand{\surf}{\mathbf{S}}
\newcommand{\Cone}{\operatorname{Cone}}

\providecommand{\add}{\mathop{\rm add}\nolimits}%
\providecommand{\per}{\mathop{\rm per}\nolimits}%
\providecommand{\thick}{\mathop{\rm thick}\nolimits}%
\renewcommand{\k}{\mathbf{k}}

\def\td{\Gamma^{(d)}}

\renewcommand{\c}{\mathscr{C}}
\renewcommand{\d}{\mathscr{D}}
 
\newcommand{\f}{\mathscr{F}}

\renewcommand{\t}{\mathscr{T}}
\newcommand{\rga}{\rightarrow}

\newcommand{\Rmone}{\uppercase\expandafter{\romannumeral1}}
\newcommand{\Rmtwo}{\uppercase\expandafter{\romannumeral2}}

\DeclareMathOperator{\dec}{dec}
\DeclareMathOperator{\red}{red}
\usepackage{tikz}
\def\rn{node[red]{$\bullet$}}

\def\numbers{\begin{enumerate}[label=\arabic*{$^\circ$}.]}
	\def\ends{\end{enumerate}}

\begin{document}
%\special{dvipdfmx:config z 0} %取消pdf压缩，加快速度，最终版时去掉
%\bibliographystyle{unsrt}

\title {The quiver with superpotentials of a $d$-angulation of a marked surface}
\keywords{Generalized higher cluster categories, geometric model, cluster tilting object, $d $-angulation.}
%\date{\today}
\author{Bo Le and Bin Zhu}
\maketitle
\begin{abstract}
In this paper, we associate a quiver with superpotential to each $d$-angulation of a (unpunctured) marked surface. We show that, under quasi-isomorphisms, the flip of a $d$-angulation is compatible with Oppermann's mutation of (the Ginzburg algebra of) the corresponding quiver with superpotential, thereby partially generalizing the result in \cite{labardini}. Applying to the generalized $(d-2)$-cluster categories associated to this quiver with superpotential, we prove that some certain almost complete $(d-2)$-cluster tilting objects in the higher cluster category have exactly $d-1$ complements. 
\end{abstract}
\tableofcontents	
\section*{Introduction}
Cluster algebras were introduced by Fomin and Zelevinsky \cite{fz} around 2000. The geometric aspect of cluster theory was explored and developed by Fomin, Shapiro and Thurston \cite{fst}. They associated to each bordered surface with marked points a cluster algebra, and proved that the seeds of the cluster algebra are related
by a mutation if and only if the triangulations to which the seeds are associated are related
by a flip.
On the other hand, cluster categories of acyclic quivers were introduced by Buan, Marsh, Reineke, Reiten and Todorov \cite{bmrrt} (see also in
\cite{CCS06} for quivers of type $A$) in order to categorify cluster algebras, and have been investigated in many subsequent articles (cf. \cite{Rei,amiot11} for  surveys).

%To give a representation-theoretic interpretation of mutations, Derksen, Weyman, and Zelevinsky \cite{DWZ} introduced the notion of quivers with potentials and defined the mutations of such objects. Such concepts was deeply related with triangulated marked surfaces due to the work of Labardini \cite{labardini}, in which he first attempted to relate the two works (\cite{fst} and \cite{DWZ}) to each ideal triangulation of a bordered surface with marked points. He associated a quiver with potential to each ideal triangulation in such a way that, whenever two ideal triangulations are related by a flip of an arc, the respective quivers with potentials are related by a mutation with respect to the vertex corresponding to the flipped arc.

 Derksen, Weyman, and Zelevinsky \cite{DWZ} introduced the notion of quivers with potentials and defined the mutations of such objects. Then, Labardini \cite{labardini} first attempted to relate the two works (\cite{fst} and \cite{DWZ}) to each ideal triangulation of a bordered surface with marked points. He associated a quiver with potential to each ideal triangulation in such a way that, whenever two ideal triangulations are related by a flip of an arc, the respective quivers with potentials are related by a mutation with respect to the vertex corresponding to the flipped arc.
	
	For a positive integer $m$ and a finite acyclic quiver $Q$, the
	$m$-cluster category ${\mathcal {C}}_Q^{\tiny{(m)}}$ is constructed
	as the orbit category of ${\mathcal {D}}$$^b ({\rm mod} kQ)$ under
	the action $\tau^{-1}[m]$, where $[1]$ is the suspension functor and $\tau$ is the Auslander–Reiten translation. It was first defined in \cite{orbit} and has
	been studied in several articles \cite{KR07,KR08,Tho07}
	and others. Combinatorial descriptions of higher cluster categories
	of type $A_n$ and $D_n$ are studied in \cite{BM01} \cite{BM02}, and \cite{zhou-zhu,rigid} proved that there are exactly $m+1$
	non-isomorphic complements to an almost complete $m$-cluster tilting object. The geometric realizations of some special
	(higher) cluster categories have also been studied, e.g., \cite{TH,Sch,BM01,BM02,BT20,LJM}.
	
	For some special differential graded algebra $A$, Amiot\cite{amiot} and Guo\cite{guo1} defined generalized higher cluster categories. Note that the generalized higher cluster categories arise from a graded
	quiver with superpotential. Oppermann\cite{oppermann} gave a combinatorial mutation rule for non-positive differential graded quiver algebras. He applied his mutation rule to Ginzburg dg-algebras and immediately obtained the result of Keller and Yang \cite{keller-yang}, generalizing this result to arbitrary Calabi–Yau dimensions.
	
	This paper aims to construct a geometric model of the generalized higher cluster categories and, using this, to give a geometric description of the mutations of (the Ginzburg algebra of) the corresponding quiver with superpotential, which partially generalizes Labardini's work \cite{labardini}. Applying to the generalized $(d-2)$-cluster categories associated to this quiver with superpotential, we prove that some certain almost complete $(d-2)$-cluster tilting objects in the higher cluster category have exactly $d-1$ complements. In recent works \cite{CHQ23,CHQ24}, Christ, Haiden, and Qiu have studied the Ginzburg dg algebra associated with a mixed-angulation. They proved that the exchange graph of hearts is isomorphic to the flip graph.
        
    This paper is organized as follows:	In Section~1, we collect some basic concepts on generalized higher cluster categories and quivers with superpotentials. In Section~2 and Section~3, we introduce the quiver with superpotential associated to a $d$-angulation of 
    a (unpunctured) marked surface, and we show that the flip of a $d$-angulation is compatible with Oppermann's  mutation of (the Ginzburg algebra of) the corresponding quiver with superpotential. In Section~4, we build the generalized higher cluster category arising from a quiver with
    superpotential (associated to a $d$-angulation) and then discuss the number of complements for some almost complete $(d-2)$-cluster tilting objects. We note that there is a similar work by Jacquet-Malo \cite{jacquet-malo_construction_2024}. However, our construction is different from hers, we will give examples to explain the differences in Section~6.
\subsection*{Acknowledgements}
The authors are very grateful to Professor Steffen Oppermann and Dr. Merlin Christ  for their helpful discussions. They also thank Professor Yu Zhou for his careful reading of the manuscript. Additionally, they sincerely appreciate Professor Yu Qiu and Dr. Ping He for pointing out minor mistakes and offering valuable suggestions. The work is supported partially by National Natural Science Foundation
of China Grant No. 12031007, 12371034.
\section{Preliminaries}
Throughout this paper, we assume $\k$ to be an algebraically closed field and $d\geq 3$ is a positive integer.	
In this section, we recall some basic notions and results related to generalized higher cluster categories and quivers with superpotentials.
\subsection{Generalized higher cluster categories}
Let $A$ be a differential graded (dg for short) $\k$-algebra and $\d(A)$ the derived category of $A$. There are two triangulated subcategory of $\d(A)$, namely, 

\begin{itemize}
	\item[-] $\per(A)$ the \emph{perfect category} of $A$, i.e., the
	smallest triangulated subcategory of the derived category $\d (A)$ containing $A$ and stable under passage to direct summands;
	
	\item[-] $\pvd(A)$ be the \emph{perfectly valued derived category} of $A$, i.e., the subcategory of $\d(A)$ consists of objects with finite-dimensional total homology.
\end{itemize}
A dg algebra $A$ is called \emph{homologically smooth} if $A$ belongs to $\per(A^e)$ when considered as a bimodule over itself, where $A^e=A^{op}\otimes_\k A$. Homologically smooth is an invariant under derived Morita equivalence (quasi-isomorphic), see \cite[Remark~8.1]{KN13} or \cite[Lemma~2.6]{dgmorita}.
\begin{lemma}[{\cite{guo1,keller08}}]
	Let $A$ be a homologically smooth dg algebra and
	\begin{equation*}
		\Omega = \mathbb{R}\Hom_{A^{e}}(A,A^{e})
	\end{equation*}
    be an object in $\d(A^{e})$. Then for any $L$ in $\d(A)$ and any $M$ in $\d_{fd}(A)$, there is a canonical isomorphism
	\begin{equation*}
		D\Hom_{\d(A)}(M,L) \stackrel{\simeq}{\longrightarrow} \Hom_{\d(A)}(L\stackrel{\mathbb{L}}{\otimes}_{A} \Omega, M).
	\end{equation*}
\end{lemma}
Suppose that $A$ has the following properties $(\star)$:
\begin{enumerate}
	\item $A$ is homologically smooth;
	\item the $p$-th homology  $H^pA=0$ for each positive integer $p\in\mathbb{Z}_{\geq 0}$;
	\item $dim_k H^0A <\infty$;
	\item $A$ is $(m+2)$-Calabi–Yau as a bimodule, i.e. there is an isomorphism in $\d(A^e)$
	\[\mathbb{R}\Hom_{A^{e}}(A,A^{e})\simeq A[-m-2].\]
\end{enumerate}
Then, by \cite[Lemma~2.1]{keller08}, we have $\pvd(A)$ is a thick subcategory of $\per(A)$. 
\begin{definition}[{\cite{amiot,guo1}}]
	The \emph{generalized higher cluster category} of $A$ is defined as the  triangle quotient
	\[\c_A:=\per(A)/\pvd(A).\]
\end{definition}

\begin{theorem}[{\cite{amiot,guo1}}]
	Let $A$ be a dg $k$-algebra with the properties $(\star)$. Then
	\begin{enumerate}
		\item The category $\c_A$ is Hom-finite and $(m+1)$-Calabi–Yau, i.e, we have 
		\[\Hom_{\c_A}(X,Y)\cong D\Hom_{\c_A}(Y,X[m+1]), \forall X,Y\in\c_A.\]
		
		\item The canonical image $T$ of the silting object $A\in\per(A)$ under the quotient is an $m$-cluster tilting object in $\c_A$, i.e., 
		\[\Hom_{\c_A}(T,T[r])=0, \ r=1, \cdots , m,\]
		and for each object $L$ in $\c_A$, $\Hom_{\c_A}(L,T[r])=0,1\leq r\leq m$ only if $L\in\add T$.

		\item The endomorphism algebra of $T$ over $\c_A$ is isomorphic to  $H^0A$.
	\end{enumerate}
\end{theorem}

\subsection{Quivers with superpotentials}
In this subsection, we briefly recall some basic concepts of quivers with superpotentials from \cite{oppermann}.

A (non-positive) graded quiver is a finite quiver $Q$ together with a map $|-|$ from the set of arrows to $\mathbb{Z}_{\leq0}$. By abuse of notion, we denote by $Q$ a graded quiver.

A dg quiver algebra is a pair $(\k Q,d)$, where $Q$ is a graded quiver and $d$ is the degree-one endomorphism of $\k Q$ such that $d^2=0$.

\begin{definition}\label{def:QsP}
	A quiver with superpotential (QsP for short) is a pair $(Q,W)$, where
	\begin{itemize}
		\item $Q$ is a graded quiver with arrows in degrees $\{0,-1,\cdots ,2-d \}$ such that, for each arrow $\varphi$, there is an arrow (up to sign) $\varphi^{op}$ of degree $2-d-|\varphi|$ in the opposite direction with $\varphi$. We moreover requires that $(\varphi^{op})^{op}=-(-1)^{|\varphi||\varphi^{op}|}\varphi$.
		
		\item $W$, called a superpotential, is a linear combination of cycles up to (signed) cyclic permutation, which is homogeneous of degree $3-d$ and $\{W, W\}=0$ (see \cite{oppermann,van} for more details).
	\end{itemize}
\end{definition}

Let $c$ be a cycle of a graded quiver $Q$. For each $\varphi\in Q_1$, the \emph{cyclic derivative} of $c$ is given as $\partial_{\varphi} c = \sum_{c = p \varphi q} (-1)^{|p||\varphi q|} qp$. Such an operation can be extended  to  linear combinations of cycles.
\begin{definition}[{\cite[Definition~6.1]{oppermann}}]\label{def:Ginz}
	Let $(Q,W)$ be a QsP and $\bar{Q}$ the graded quiver obtained from $Q$ by adding a loop $l_i$ of degree $1-d$ at each vertex $i\in Q_0$. The $d$-dimensional Ginzburg algebra $\td(Q,W)$ associated to $(Q,W)$ is the dg algebra $(\k\bar{Q},d)$ with
	\begin{align}\label{d}
		d \varphi & = \partial_{\varphi^{\rm op}} W && \varphi \in Q_1, \\
		d l_i & = \sum_{\varphi} \varphi\varphi^{op}&& i \in Q_0.
	\end{align}
	where the sum runs over all arrows $\varphi$ of $Q$ starting in vertex $i$. 
\end{definition}

Let $i$ be a vertex in $Q_0$ such that there is no loop of degree 0 at $i$. Denote
\[A=\{\alpha \in Q_1 |\alpha:j \rga i ,  |\alpha|=0 \}.\]

Let $p$ be a (linear combination of) path(s) in $Q$. The \emph{reduction} $\red(p)$ of $p$ is obtained from $p$ by removing all paths that end in an arrow $\alpha \in A$. Denote by $p/\alpha$ the linear combination obtained from $p$ by remembering paths that end with $\alpha $ and removing the $\alpha $. Then we have 
\begin{align*}
	p = \red(p) + \sum_{\alpha \in A} p/\alpha \; \alpha.
\end{align*}

Let $\Delta = 1 -(\sum_{\alpha \in A} \alpha^{-1} \alpha)$. The \emph{decorated} version $\dec(p)$ of $p$, is obtained from $p$ by introducing $\Delta$ whenever this is allowed (see \cite[Construction~3.1]{oppermann} for more details).

For each cycle $c$ that does not start with an arrow $\alpha$ in $A$ (this may always be achieved by cyclic permutation), denote by
\begin{align*}
	\dec_{ cyc} c = \begin{cases} \dec c & \text{if $c$ does not start and end in $i$;} \\ (-1)^d \dec c - \sum \alpha (\dec c) \alpha^{-1} & \text{if $c$ starts and ends in $i$.} \end{cases}
\end{align*} 

\begin{definition}[{\cite[Definition~6.3]{oppermann}}]\label{def:mut QsP}
	Let $(Q,W)$ be a QsP and $i$ be a vertex of $Q$ that admits no loops of degree $0$. Then the mutation $\mu_i(Q,W)=(Q_M,W_M)$ is constructed from $(Q,W)$ via the following steps, where $||\alpha||$ denotes the degree of $\alpha$ in $Q_M$.
	\begin{enumerate}
		\item For each $\alpha\in Q_1$ with $s(\alpha)=i$ (resp. $t(\alpha)=i$), we have $||\alpha||=|\alpha|-1$ (resp. $||\alpha||=|\alpha|+1$).
		%Increase by 1 the degree of all arrows ending in vertex $i$,  and decrease the degree of all arrows starting in $i$;
		\item For any arrow $\alpha\in A$,  $\alpha$ and $\alpha^{op}$ are replaced by $\alpha^{*}$ and $(\alpha^{*})^{op_M}$ respectively, where $\alpha^*$ has the same direction as $\alpha^{op}$ and $||\alpha^{*}||=0$.			
		\item For any arrow $\alpha\in A$ and any $\varphi\in Q_1$ with $s(\varphi)=i$ and $|\varphi|\geq 3-d$, add a new arrow which is the
		formal composition $\alpha\varphi$, and its opposite $(\alpha\varphi)^{op_M}=-\varphi^{op}\alpha^{-1}$ and $||\alpha \varphi||=|\varphi|$.
		\item For any loop $\varphi$ at $i$ in $Q$, and any two
		arrows $\alpha,\beta\in A$, add a new formal composite arrow $\alpha \varphi \beta^{-1}$ of degree $||\alpha \varphi \beta^{-1}||=|\varphi|$, and its opposite $(\alpha \varphi \beta^{-1})^{op_M} = \beta\varphi^{op} \alpha^{-1}$.
		\item The superpotential $W_M$ is given by
		\[W_{\rm M} = \dec_{ cyc} W + \sum_{\alpha, \varphi} \alpha \dec( \varphi \varphi^{op_M} ) \alpha^*,\]
		where the sum runs over all $\alpha\in A$ and all arrows $\varphi\in Q_1$ with $s(\varphi)=i$ and $|\varphi|\geq 3-d$.
	\end{enumerate}
\end{definition}
Denote by $\chi(\cdot)$ the characteristic function of $i$, that is
\begin{align*}
	\chi (j) = \begin{cases} 1 &  j = i; \\ 0 &  j \neq i .\end{cases}
\end{align*}
Then the cyclic derivatives of the mutated QsP can be calculated via the cyclic derivatives of the original one. To be more precisely, $\partial W_M$ is given as the following (see \cite[Lamma~6.5]{oppermann}).
\begin{align*}
	\partial_{\varphi^{op_M}} W_M&=
	\begin{cases}
		\dec \red \partial_{\varphi^{op}} W & , s( \varphi)\neq i ,\varphi \in Q \cap Q_M ;\\
		- \dec \partial_{\varphi^{op}} W + \sum\limits_{\alpha} \alpha^* (\alpha \varphi) & ,s( \varphi)=i ,\varphi \in Q \cap Q_M .
	\end{cases}\\
	\partial_{(\alpha^*)^{op_M}} W_M & = 0  ;\\
	\partial_{\alpha^*} W_M & = \sum_{\varphi} \alpha \dec( \varphi \varphi^{op_M} )  ;\\
	\partial_{(\alpha \varphi)^{op_M}} W_M & = \alpha \dec \red (\partial_{\varphi^{op}} W) ; \\
	\partial_{(\varphi \alpha^{-1})^{op_M}} W_M& = (\partial_{\varphi^{op_M}} W_M ) \alpha^{-1} + (-1)^{\|   \varphi \|} \varphi \alpha^* - (-1)^{\chi(s( \varphi))} \dec (\partial_{\varphi^{op}} W / \alpha)  ;\\
	\partial_{(\alpha \varphi \beta^{-1})^{op_M}} W_M & = (\partial_{(\alpha \varphi)^{op_M}} W_M ) \beta^{-1} + (-1)^{ \| \alpha \varphi \|} \alpha \varphi \beta^* - \alpha \dec (\partial_{\varphi^{op}} W / \beta) .
\end{align*}
Then, we have the differential of the mutated Ginzburg dg algebra. 
\begin{proposition}\label{prop:diff of mut-Ginz}
	The differential $d_M$ of the Ginzburg algebra $\td(Q_M,W_M)$ is given by the following.
	\begin{align}
		d_M (\varphi)&=
		\begin{cases}
			\dec (\red (d\varphi )) & , s( \varphi)\neq i ,\varphi \in Q \cap Q_M ;\\
			- \dec(d\varphi) + \sum\limits_{\alpha} \alpha^* (\alpha \varphi) & ,s( \varphi)=i ,\varphi \in Q \cap Q_M .
		\end{cases}\\
		d_M(\alpha^*) & = 0  ;\\
		d_M((\alpha^*)^{op_M}) & = \sum_{\varphi} \alpha \dec( \varphi \varphi^{op_M} )  ;\\
		d_M(\alpha \varphi) & = \alpha \dec (\red (d\varphi)) ; \\
		d_M(\varphi \alpha^{-1})& = d_M(\varphi ) \alpha^{-1} + (-1)^{\|   \varphi \|} \varphi \alpha^* - (-1)^{\chi(s( \varphi))} \dec ((d\varphi) / \alpha)  ;\\
		d_M(\alpha \varphi \beta^{-1}) & = d_M(\alpha \varphi) \beta^{-1} + (-1)^{ \| \alpha \varphi \|} \alpha \varphi \beta^* - \alpha \dec ((d\varphi)/ \beta) .
	\end{align}
\end{proposition}

Recall from \cite{AI} that a basic object $T$ in a triangulated category $\t$ is called a \emph{silting} object if all its positive self-extension vanishes and $\thick T=\t$. One important feature of silting objects is that they admit a notion of mutation.
\begin{theorem}[{\cite{AI}}]\label{thm:silt-mutation}
	Let $T=T_1\oplus T_2$ be a silting object in $\t$ with $T_1$ an indecomposable direct summand. Assume that $T_1$ has a left $\add T_2$-approximation $T_1\xrightarrow{f}\tilde{T_2}$. Extend $f$ to be the following triangle
	\begin{equation}\label{1.9}
		T_1\xrightarrow{f}\tilde{T_2}\rga \Cone(f) \rga T_1[1].
	\end{equation}
	Then $\mu_{T_1}(T):=\Cone(f)\oplus T_2$ is also a silting object in $\t$.
\end{theorem}

The silting mutation and dg quiver mutation are compatible in the following sense.
\begin{theorem}[{\cite{oppermann}}]\label{thm:1.9}
	Let $(Q,W)$ be a QsP. Then the derived endomorphism
	ring of the silting mutation of $(k\bar{Q},d)$ in $i$ is the Ginzburg algebra of $\mu_i(Q,W)$, i.e., we have the following commutative diagram up to quasi-isomorphism
	$$\xymatrix@C=6pc{(Q,W) \ar[r]^{\text{Ginzburg}}\ar[d]^{\mu_i}&(k\bar{Q},d)\ar[d]^{\text{silting mutation}}\\
	\mu_i(Q,W) \ar[r]^{\text{Ginzburg}}&(kM,\partial).}$$
\end{theorem}

\section{$d$-angulation of a marked surface}

In this section, for each $d$-angulation $D$ of a (unpunctured) marked surface $\surf$, we associate it a quiver with superpotential $(Q_D,W_D)$. We show that the flip of $D$ is compatible with the mutation (see Definition~\ref{mutation}) of $(Q_D,W_D)$.
\begin{definition}
	A \emph{(unpunctured) marked surface} $\surf$ is a pair $(S,M)$, where 
	\begin{itemize}
		\item $S$ is a compact connected oriented Riemann surface with non-empty boundary $\partial S$;
		
		\item $M\subset\partial S$ is a finite set of \emph{marked points}  such that each component of $\partial S$ has at least one marked point. 
	\end{itemize}
\end{definition}

An (ordinary) arc in $(S,M)$  is a curve $\gamma:[0,1]\rga S$ such that: 
\begin{enumerate}
	\item $\gamma(0),\gamma(1)\in M$ and $\gamma(t)\in S\setminus\partial S, t \in (0,1)$;
	\item $\gamma$ is neither null homotopic nor homotopic to a boundary segment;
	\item $\gamma(t_1)\neq\gamma(t_2)$ for any $t_1,t_2\in(0,1)$.
\end{enumerate}

All arcs are considered under their isotopy classes. Two arcs are called \emph{compatible} if there are no universal intersections between them.
\begin{proposition}[{\cite{fst}}]
	Given any collection of pairwise compatible arcs, it is always possible to
	find representatives in their isotopy classes whose relative interiors do not intersect each other.
\end{proposition}
\begin{definition}
	Let $(S,M)$ be a marked surface. A \emph{$d$-angulation} $D$ of $(S,M)$ is a set of compatible arcs which divides $(S,M)$ into a collection of (pseudo) $d$-gons. 
\end{definition}

We use $(S,M,D)$ to denote a marked surface $(S,M)$ with a $d$-angulation $D$. For such a $(S,M,D)$, each arc $i$ is contained in exactly two $d$-gons $D_1$ and $D_2$. Moreover, we call $i$ \emph{self-folded} if $D_1=D_2$. 
\begin{example}
	Let $(S,M)$ be an annulus with three marked points on each of its boundary components. Then $(S,M)$ admits a 3-angulation $D'$ as shown by the left picture in Figure~\ref{fig:3/4-angulation}, as well as a 4-angulation $D$ as shown by the right picture, on it. In the 4-angulation $D$, $1$ is a self-folded arc.
	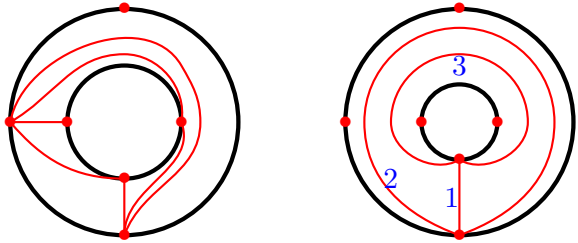
\begin{figure}[htpb]
		\centering
		\begin{tikzpicture}[scale=0.5]
			\draw[ultra thick](0,0) circle (3);
			\draw[ultra thick](0,0) circle (1.5);
			\draw[red,thick](1.5,0)to[out=-70,in=90](0,-3);
			\draw[red,thick](-3,0)to[out=0,in=180](-1.5,0);
			\draw[red,thick](-3,0)to[out=30,in=180](0,1.8)to[out=0,in=80](1.5,0);
			\draw[red,thick](-3,0)to[out=-50,in=180](0,-1.5); 
			\draw[red,thick](0,-1.5)to(0,-3.0); 
			\draw[red,thick](-3,0)to[out=80,in=120](1.5,1.5)to[out=-60,in=90](2,0)to[out=-90,in=70](0,-3);
			\draw(-1.5,0)\rn (1.5,0)\rn (0,-1.5)\rn (-3,0)\rn (0,3)\rn (0,-3)[red]\rn ;
		\end{tikzpicture}\quad \quad \quad
		\begin{tikzpicture}[scale=0.5]
			\draw[ultra thick](0,0) circle (3);
			\draw[ultra thick](0,0) circle (1);
			\draw[red,thick](0,-1)to[out=-90,in=90](0,-3);
			\draw[red,thick](0,-3)to[out=20,in=270](2.5,0)to[out=90,in=0](0,2.5)to[out=180,in=90](-2.5,0)to[out=-90,in=160](0,-3);
			\draw[red,thick](0,-1)to[out=-30,in=270](1.8,0)to[out=90,in=0](0,1.8)to[out=180,in=90](-1.8,0)to[out=-90,in=210](0,-1);
			\draw(-1,0)\rn (1,0)\rn (0,-1)\rn (-3,0)\rn (0,3)\rn (0,-3)[red]\rn ;
			\draw[blue] (-.2,-2)node{$1$} (-1.8,-1.5)node{$2$} (0,1.5)node{$3$};
		\end{tikzpicture}
		\caption{Example of 3-angulation and 4-angulation on the same marked surface $(S,M)$}
		\label{fig:3/4-angulation}
	\end{figure}
\end{example}
Denote by $B$ the set of boundary segments of $(S,M)$.
\begin{definition}[{Flip}]
	Let $D$ be a $d$-angulation on $(S,M)$. For each arc $i\in D$, the \emph{flip} of $i$ is the arc $\mu_D(i)$ obtained from $i$ by moving its two endpoints along $B\cup(D\setminus\{i\})$ anticlockwise to the next marked points (see Figure~\ref{flip} and Figure~\ref{flip self-folded}).
	Denote by $\mu_i(D)=D\setminus\{i\}\cup\{\mu_D(i)\}$ the new $d$-angulation after flipping $i$.
\end{definition}
\begin{remark}
	For a self-folded arc $i$ (see Figure~\ref{flip self-folded}), we know that $i^+,i^-$ are in the $d$-gon $D_1$. Denote by $k_1$ (resp.  $k_2$) the number of edges
	sandwiched between $1^+$ and $1^-$ (resp.  $1^-$ and $1^+$) in clockwise order in $D_1$. We know that
	$k_1,k_2\geq 1$ and $k_1+k_2+2=d\geq4 $ since there are no punctures on $(S,M)$.
\end{remark}
\begin{definition}
	A directed graph 
	$G$ is called a \emph{directed complete graph} if there are exactly two opposite directed edges between any two distinct vertices, and there are no loops in $G$.
	%A directed graph $G$ is a \emph{directed complete graph} if
	%there are  exactly two opposite directed edges between any two different vertices and there are no loops  on $G$.
	
	An undirected graph 
	$G$ is called a \emph{complete graph} (or \emph{regular graph}) if there is exactly one undirected edge between any two distinct vertices, and there are no loops in $G$.	
	%A undirected graph $G$ is a \emph{complete graph (regular graph)} if
	%there is exactly an (undirected) edge between any two different vertices and there are no loops on $G$.
\end{definition}
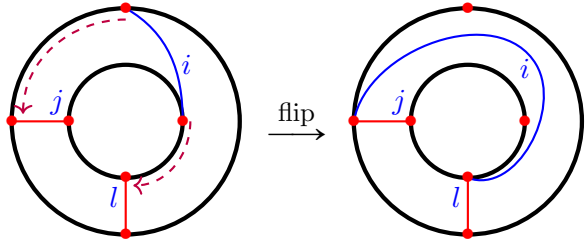
\begin{figure}[htpb]
	\centering
	\begin{tikzpicture}[scale=0.5]
		\draw[ultra thick](0,0) circle (3);
		\draw[ultra thick](0,0) circle (1.5);
		\draw[blue,thick](0,3)to[out=-30,in=90](1.5,0);
		\draw[purple,thick,dashed,->] (0,2.7)to[out=180,in=90](-2.7,0.2);
		\draw[purple,thick,dashed,->] (1.7,0)to[out=-90,in=0](0.2,-1.7);
		\draw[blue] (-1.8,.5)node{$j$};
		\draw[blue] (1.6,1.5)node{$i$};
		\draw[blue] (-0.3,-2.0)node{$l$};
		\draw[red,thick](-3,0)to[out=0,in=180](-1.5,0); 
		\draw[red,thick](0,-1.5)to(0,-3.0); 
		\draw(-1.5,0)\rn (1.5,0)\rn (0,-1.5)\rn (-3,0)\rn (0,3)\rn (0,-3)[red]\rn ;
		\draw[] (4.5,0) node {\Large{$\stackrel{\text{flip}}{\longrightarrow}$}};
		\begin{scope}[shift={(9,0)}]
			\draw[ultra thick](0,0) circle (3);
			\draw[ultra thick](0,0) circle (1.5);
			\draw[blue,thick](-3,0)to[out=80,in=120](1.8,1.5)to[out=-60,in=-20](0,-1.5);
			\draw[blue] (1.5,1.4)node{$i$};
			\draw[blue] (-1.8,.5)node{$j$};
			\draw[blue] (-0.3,-2.0)node{$l$};
			\draw[red,thick](-3,0)to[out=0,in=180](-1.5,0); 
			\draw[red,thick](0,-1.5)to(0,-3.0); 
			\draw(-1.5,0)\rn (1.5,0)\rn (0,-1.5)\rn (-3,0)\rn (0,3)\rn (0,-3)[red]\rn ;
		\end{scope}	
	\end{tikzpicture}			
	\caption{Flip of an arc $i$}\label{flip}
\end{figure} 
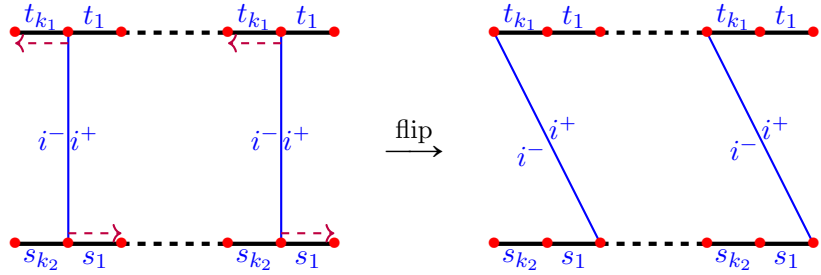
\begin{figure}[htpb]
	\centering
	\begin{tikzpicture}[scale=0.7]
		\draw[ultra thick] (-3,2)to(-1,2);
		\draw[ultra thick] (-3,-2)to(-1,-2);
		\draw[ultra thick] (1,-2)to(3,-2);
		\draw[ultra thick] (1,2)to(3,2);
		\draw[ultra thick,dashed] (-1,2)to(1,2);
		\draw[ultra thick,dashed] (-1,-2)to(1,-2);
		\draw[blue,thick](-2,2)to(-2,-2);
		\draw[blue,thick](2,2)to(2,-2);
		\draw(-3,2)\rn (3,2)\rn(-2,2)\rn (2,2)\rn (-1,2)\rn (1,2)\rn (-3,-2)\rn (3,-2)\rn(-2,-2)\rn (2,-2)\rn (-1,-2)\rn (1,-2)\rn;
		\draw[blue] (-1.7,0)node{$i^+$} (1.7,0)node{$i^-$}  (2.3,0)node{$i^+$} (-2.3,0)node{$i^-$}   (-2.5,2.3)node{$t_{k_1}$} (-1.5,2.3)node{$t_1$}  (1.5,2.3)node{$t_{k_1}$}(2.5,2.3)node{$t_1$}
		(-2.5,-2.3)node{$s_{k_2}$} (-1.5,-2.3)node{$s_1$}  (1.5,-2.3)node{$s_{k_2}$}(2.5,-2.3)node{$s_1$};
		\draw[purple,thick,dashed,->] (-2,1.8)to(-3,1.8)  ; 
		\draw[purple,thick,dashed,->] (2,1.8)to(1,1.8);
		\draw[purple,thick,dashed,->] (-2,-1.8)to(-1,-1.8)  ; 
		\draw[purple,thick,dashed,->] (2,-1.8)to(3,-1.8);
		\draw[] (4.5,0) node {\Large{$\stackrel{\text{flip}}{\longrightarrow}$}};
		\begin{scope}[shift={(9,0)}]
			\draw[ultra thick] (-3,2)to(-1,2);
			\draw[ultra thick] (-3,-2)to(-1,-2);
			\draw[ultra thick] (1,-2)to(3,-2);
			\draw[ultra thick] (1,2)to(3,2);
			\draw[ultra thick,dashed] (-1,2)to(1,2);
			\draw[ultra thick,dashed] (-1,-2)to(1,-2);
			\draw[blue,thick](-3,2)to(-1,-2);
			\draw[blue,thick](1,2)to(3,-2);
			\draw(-3,2)\rn (3,2)\rn(-2,2)\rn (2,2)\rn (-1,2)\rn (1,2)\rn (-3,-2)\rn (3,-2)\rn(-2,-2)\rn (2,-2)\rn (-1,-2)\rn (1,-2)\rn;
			\draw[blue] (-1.7,.2)node{$i^+$} (1.7,-.2)node{$i^-$} (2.3,.2)node{$i^+$} (-2.3,-.3)node{$i^-$}   (-2.5,2.3)node{$t_{k_1}$} (-1.5,2.3)node{$t_1$}  (1.5,2.3)node{$t_{k_1}$}(2.5,2.3)node{$t_1$}
			(-2.5,-2.3)node{$s_{k_2}$} (-1.5,-2.3)node{$s_1$}  (1.5,-2.3)node{$s_{k_2}$}(2.5,-2.3)node{$s_1$};
		\end{scope}	
	\end{tikzpicture}	
	\caption{Flip of a self-folded arc $i$}\label{flip self-folded}
\end{figure}
\begin{definition}\label{def:QW}
	Suppose $D$ divides $S$ into $d$-gons $\{D_1,\cdots, D_s\}$. For each $D_i, 1\leq i\leq s$, let $Q(D_i)$ be the graded directed complete graph constructed as follows.
	\begin{itemize}
		\item Vertices of $Q(D_i)$ are indexed by the non-boundary edges of $D_i$.		
		\item Each arrow $\phi$ in $Q(D_i)$ has degree $|\phi|=-n$ with $n$ the number of other edges sandwiched
		between $s(\phi)$ and $t(\phi)$ in clockwise order.
	\end{itemize}
    A (suitable) superpotential $W(D_i)$ associated to $Q(D_i)$ is defined, up to sign, as the sum of all 3-cycles in $Q(D_i)$ such that $\{W(D_i),W(D_i)\}=0$. Note that $W(D_i)$ is homogeneous of degree $3-d$.
	
    A QsP $(Q_D,W_D)$ associated to $D$ is obtained from  the \emph{pre-quiver} $ _{p}Q=\{Q(D_1),\cdots,Q(D_s)\}$ by gluing vertices indexed by the same arc, and $W_D=\sum^s_{i=1}W(D_i)$. 
    
    For simplicity, we also denote $(Q_D,W_D)$ by $(Q,W)$ when there is no confusion arising.
\end{definition}
\begin{remark}
	Note that in the Definition~\ref{def:QW}, all (suitable) choices of superpotential  are considered (while the quiver is fixed).
	The quiver with superpotential associated to a $d$-angulation has already been studied, see \cite[Section~4.3]{merlin} and \cite[Section~8]{IQ} for example.
\end{remark}
\begin{remark}
	Since each arc $i$ is exactly associated to two vertices in $_{p}Q$, we denote by $i^+,i^-$ the  two vertices. Denote by $a_{sm}$ the arrow in $_{p}Q$ from $s$ to $m$, where $s,m\in (_pQ)_0$.
\end{remark}
\begin{example}
	 Consider the 4-angulation $(S,M,D)$ in Figure~\ref{fig:3/4-angulation}, then we have pre-quiver $_{p}Q$  associated to $(S,M,D)$  as follows:
		
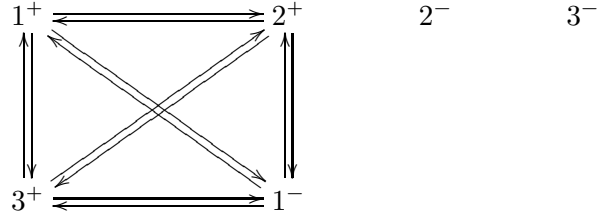
\begin{align*}
			\xymatrix@C=3pc{1^+\ar@<0.3ex>[rr]^{} \ar@<0.3ex>[rrdd]^{} \ar@<0.3ex>[dd]^{}&&  \ar@<0.3ex>[ll]^{} \ar@<0.3ex>[dd]^{} \ar@<0.3ex>[lldd]^{} 2^+  & 2^- & 3^- 
				\\				
				\\
				3^+\ar@<0.3ex>[rr]^{} \ar@<0.3ex>[uu]^{} \ar@<0.3ex>[rruu]^{} & & \ar@<0.3ex>[ll]^{} \ar@<0.3ex>[lluu]^{} \ar@<0.3ex>[uu]^{} 1^-  & &
			}
		\end{align*}
		The quiver $Q$ associated to $(S,M,D)$ is:		
		
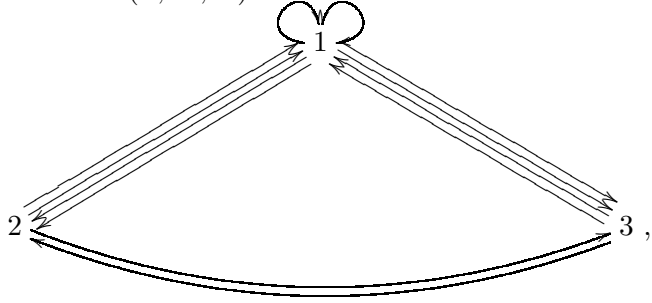
\begin{align*}
			\xymatrix@C=4pc{ & & \ar@(l,u)
				1 \ar@(r,u) \ar@<0.3ex>[lldd]^{}  \ar@<1ex>[lldd]^{}  \ar@(u,r) \ar@<0.3ex>[rrdd]^{}  \ar@<1ex>[rrdd]^{}&& 
				\\
				\\
				\ar@<0.3ex>@/_2pc/[rrrr]	2 \ar@<0.3ex>[rruu]^{}  \ar@<1ex>[rruu]^{} &&&&   \ar@<0.4ex>@/^2pc/[llll] 3 \ , \ar@<0.3ex>[lluu]^{}  \ar@<1ex>[lluu]^{} 
				\\
				&&&&
			}
		\end{align*}
		where 	$|a_{1^+2^+}|=|a_{2^+1^-}|=|a_{1^-3^+}|=|a_{3^+1^+}|=0$,
		$|a_{2^+1^+}|=|a_{1^-2^+}|=|a_{3^+1^-}|=|a_{1^+3^+}|=-2$ and $|a_{1^-1^+}|=|a_{1^+1^-}|=|a_{2^+3^+}|=|a_{3^+2^+}|=-1$.		
\end{example}
\begin{remark}\label{rmk:sign}
	There is no canonical choice for the sign of 3-cycles in $W$. Indeed, when $d=3$, we have $|a_{ij}a_{jl}a_{li}|=0$. So there is no sign-difference up to cyclic equivalence. Hence different choices of the signs result in the same Ginzburg dg algebra. However when $d=4$, an arbitrary choice of the sign may not give a well-defined Ginzburg dg algebra, see Example~\ref{exm:sign}.
\end{remark}
\begin{example}\label{exm:sign}
	Let $\surf$ be a disk with twelve marked points on the boundary. Let $D$ be the 4-angulation as shown by the left picture of Figure~\ref{fig:sign}. Then the associated $Q_D$ is given by the right picture.
	
	Denote by $\triangle_{ijl}=a_{ij}a_{jl}a_{li}$. Then $W=\triangle_{132}+\triangle_{142}+\triangle_{234}-\triangle_{134}$ is a well-defined superpotential. However, $W=\triangle_{132}+\triangle_{142}+\triangle_{234}+\triangle_{134}$ is not a superpotential because $d^2(a_{12})\neq 0$.
\begin{figure}[htpb]\centering
	\begin{tikzpicture}[scale=0.5]
		\draw[ultra thick](0,0) circle (3);
		\draw[red,thick](0,3)to (3,0);
		\draw[red,thick](3,0)to (0,-3);
		\draw[red,thick](0,-3)to (-3,0);
		\draw[red,thick](-3,0)to (0,3);
		\draw[blue] (-1.3,1.3)node{$1$};
		\draw[blue] (1.3,1.3)node{$3$};
		\draw[blue] (1.3,-1.3)node{$4$};
		\draw[blue] (-1.3,-1.3)node{$2$};
		\draw(2.6,1.5)\rn (2.6,-1.5)\rn  (-2.6,1.5)\rn  (-2.6,-1.5)\rn (3,0)\rn (0,-3)\rn (-3,0)\rn (0,3)\rn (0,-3)[red]\rn (0,-3)\rn (1.5,2.6)\rn  (-1.5,2.6)\rn(1.5,-2.6)\rn(-1.5,-2.6)\rn;		
		\draw(8,0)node{$\xymatrix@R=1.5em@C=2em{& & &3 \ar@<0.2ex>[ldd]^{-1} \ar@<0.2ex>^{-2}@/_1pc/[llldd] \ar@<0.2ex>[dddd]^{0}\\
				&&&\\  1 \ar@<0.2ex>[rr]^{-2}  \ar@<0.2ex>^{-1}@/_1pc/[rrrdd] \ar@<0.2ex>^{0}@/^1pc/[rrruu] & & \ar@<0.2ex>[ll]^{0} \ar@<0.2ex>[rdd]^{-2} 2 \ar@<0.2ex>[ruu]^{-1} & \\
				&&& \\ 
				& & &      \ar@<0.2ex>^{-1}@/^1pc/[llluu] 4 \ar@<0.2ex>[luu]^{0}  \ar@<0.2ex>[uuuu]^{-2}  }$  };
	\end{tikzpicture}
	\caption{}
	\label{fig:sign}
\end{figure}
\end{example}

Note that every arc $i$ belongs to  two $d$-gons $D_1,D_2$ (and $D_1=D_2$ when $i$ is a self-fold arc).  Denote by 
\begin{align}
	A=\{\alpha \in Q_1 |\alpha:j \rga i ,  |\alpha|=0 \}.
\end{align}
By Definition~\ref{def:QW}, $A$ consists of all degree-0 arrows end at $i$ and we have $|A|\in\{0, 1 ,2\}$.

Let $(Q_D,W_D)$ be a well-defined QsP of a $d$-angulation $D$. Although the superpotential associated to $Q_{D}$ is not unique (due to the choice of signs, see Remark~\ref{rmk:sign}), we can always construct a well-defined superpotential $W_D$ of $Q_{D}$ (see {\cite[Section4.3]{merlin}}).

\begin{definition}\label{mutation}
	Let $(Q, W)$ be as in Definition~\ref{def:QW}. Then the  mutation $\mu^{'}_{i}(Q,W)=(\mu^{'}_{i}(Q),\mu^{'}_{i}(W))$ is defined by
	the following constructions.
	
	%First, we construct a new graded quiver $\mu^{'}_{i}(Q)$:
	\begin{enumerate}[Step~1]
		\item  Let $\widetilde{_p(Q)}$ be the quiver obtained from  $ _{p}Q$ by gluing vertices $i^+,i^-$ into $i$. Let $\mu_{i}(\widetilde{_p(Q)})$ be the quiver constructed as in Definition~\ref{def:mut QsP}.
		
		\item Remove all \emph{superfluous} arrow pairs defined as follows.
		\begin{enumerate}[(1)]
			\item $|A|=0$, there are no superfluous arrow pairs;
			\item $|A|=1$, there are some new arrow pairs given in Step~1 (formal composition): $\alpha\varphi , \varphi \alpha^{-1},\alpha \varphi \alpha^{-1},\ \alpha \in A$.
			Assume $A=\{a_{ki^+} \},\ k \in (_pQ)_0$. Then double-arrow-sets 
			\[(a_{ki^+}a_{i^+j},a_{kj}),\quad (a_{ki^+}a_{i^+i^-}a_{ki^+}^{-1},a_{ki^-}a_{ki^+}^{-1})\] 
			(and their opposites) are called superfluous if (the two arrows in) it exists in 
			$\mu_{i}(\widetilde{_p(Q)})$.	 	      	
			\item $|A|=2$, there are some new arrow pairs given in step 1 (formal composition): $\alpha\varphi , \varphi \alpha^{-1},\alpha \varphi \beta^{-1},\ \alpha,\beta \in A$.
			Assume $A=\{a_{ki^+} ,a_{k'i^-}\},\ k,k' \in (_pQ)_0$. Then double-arrows-sets 
			\[(a_{ki^+}a_{i^+j},a_{kj}),\quad (a_{ki^+}a_{i^+i^-}a_{ki^+}^{-1},a_{ki^-}a_{ki^+}^{-1}),\quad (a_{ki^+}a_{i^+i^-}a_{k'i^-}^{-1},a_{ki^-}a_{k'i^-}^{-1})\] 
			\[(a_{k'i^-}a_{i^-j},a_{k'j}),\quad (a_{k'i^-}a_{i^-i^+}a_{k'i^-}^{-1}, a_{k'i^+}a_{k'i^-}^{-1}),\quad (a_{k'i^-}a_{i^-i^+}a_{ki^+}^{-1},a_{k'i^+}a_{ki^+}^{-1})\]
			(and their opposites) are superfluous if (the two arrows in) it exists in $\mu_{i}(\widetilde{_p(Q)})$.	
		\end{enumerate}
		
		\item $\mu^{'}_{i}(Q)$ is obtained from $\mu_{i}(\widetilde{_p(Q)})$ by gluing all vertex-pairs $\{k^{\pm},k\neq i\}$ into $k$.
	
	    \item The superpotential $\mu^{'}_{i}(W)$ is obtained from $\mu_i(W)$ (see Definition~\ref{def:mut QsP}) by removing all paths that contain the arrows being removed in Step~2.
\end{enumerate}
\end{definition}
\begin{remark}
	The removing of superfluous arrow pairs and the reducing of superpotential  are necessary (to make sure that there is no 2-cycle in superpotential). When $d=3$, the reducing of $\mu_{i}(W)$ is consistent with right-equivalence in \cite{DWZ}, see the first case in Example~\ref{ex2.17}.
\end{remark}

The main result in this section is the following theorem, the proof will be given in Section~5.
\begin{theorem}\label{thm2.16}
	For a given  quiver with superpotential $(Q,W)$ associated to a $d$-angulation $(S,M,D)$, $\mu^{'}_{i}(Q,W)$ is a  quiver with superpotential associated to $(S,M,\mu_i(D))$.
\end{theorem}

\begin{figure}[htpb]
	\centering
	\begin{tikzpicture}[scale=0.6]
		\draw[ultra thick](0,0) circle (3);
		\draw[red,thick](0,3)to (2.6,-1.5);
		\draw[red,thick](2.6,-1.5)to (-2.6,-1.5);
		\draw[red,thick](-2.6,-1.5)to (0,3);
		\draw[blue] (1.3,1.3)node{$3$};
		\draw[blue] (-1.3,1.3)node{$2$};
		\draw[blue] (0,-1.1)node{$1^{+}$};
		\draw[blue] (0,-1.8)node{$1^{-}$};
		\draw(2.6,1.5)\rn (2.6,-1.5)\rn  (-2.6,1.5)\rn  (-2.6,-1.5)\rn  (0,-3)\rn  (0,3)\rn;
		\draw[] (4.5,0) node {\Large{$\stackrel{\text{flip}}{\longrightarrow}$}};
		\begin{scope}[shift={(9,0)}]
			\draw[ultra thick](0,0) circle (3);
			\draw[red,thick](0,3)to (2.6,-1.5);
			\draw[red,thick](0,3)to (0,-3);
			\draw[red,thick](-2.6,-1.5)to (0,3);
			\draw[blue] (1.3,1.3)node{$3$};
			\draw[blue] (-1.3,1.3)node{$2$};
			\draw[blue] (-0.5,0)node{$1^{+}$};
			\draw[blue] (0.5,0)node{$1^{-}$};
			\draw(2.6,1.5)\rn (2.6,-1.5)\rn  (-2.6,1.5)\rn  (-2.6,-1.5)\rn  (0,-3)\rn  (0,3)\rn;
		\end{scope}	
	\end{tikzpicture}
	\caption{}
	\label{d=3,A3}
\end{figure}
\begin{example}\label{ex2.17}
	\begin{enumerate}[(1)]
		\item Consider $(S,M,D)$ and $(S,M,\mu_1(D))$ in Figure~\ref{d=3,A3} ({\cite[Example 14]{labardini}}).
		$(Q,W)$ associated to $(S,M,D)$ is as follows:
		\begin{align*}
			\xymatrix@C=6pc{&  2 \ar@<0.3ex>^{a_{21^+}}[ld] \ar@<0.3ex>^{a_{23}}[rd] &  \\
				1 \ar@<0.3ex>^{a_{1^{+}2}}[ru] \ar@<0.3ex>^{a_{1^{+}3}}[rr]&   & 3 \ar@<0.3ex>^{a_{32}}[lu] \ar@<0.3ex>^{a_{31^{+}}}[ll]}
		\end{align*}
		\[
		W=a_{31^+}a_{1^{+}2}a_{23}
		\]
		Next consider $\mu_{1}(Q,W)$, we know that $A=\{ a_{31^+}\}$ and the only new arrow pair is%as follows by Definition~\ref{def:mut QsP}:
		\[
		\{
		a_{31^+}a_{1^+2},a_{21^+}a_{31^+}^{-1}
		\}.
		\]
		For convenience, we denote by $\{b_{32},b_{23}\}$ the new arrow pair. Then the superpotential is 
		\begin{align*}
			\mu_{1}(W)&=a_{31^+}a_{1^{+}2}a_{23}+a_{31^+}a_{1^+2}a_{21^+}a_{31^+}^{*}\\
			&=b_{32}a_{23}+b_{32}a_{21^+}a_{31^+}^{*}.
		\end{align*}
		
		The superfluous arrow pairs are:
		\[
		\{a_{23},a_{32},b_{32},b_{23}\}.
		\]
		Then $\mu_{1}^{'}(Q)$ as follows:
		
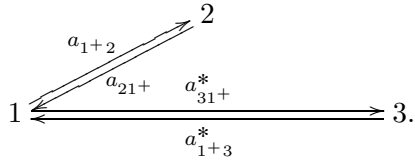
\begin{align*}
			\xymatrix@C=5pc{&  2 \ar@<0.3ex>^{a_{21^+}}[ld]  &  \\
				1 \ar@<0.3ex>^{a_{1^{+}2}}[ru] \ar@<0.3ex>^{a_{31^{+}}^{*}}[rr]&   & 3. \ar@<0.3ex>^{a_{1^{+}3}^{*}}[ll]}
		\end{align*}
		So, we have
		\begin{align*}
			\mu_{1}^{'}(W)&=0.
		\end{align*}
		Notice the arrow pair $a_{31^+},a_{1^+3}$ is replaced by the arrow pair $a_{31^+}^*,a_{1^+3}^*=(a_{31^+}^*)^{op}$  in the opposite direction. We view $a_{31^+}^*,a_{1^+3}^*=(a_{31^+}^*)^{op}$
		as the arrow pair between $3$ and $1^-$.
		\item  	
		\begin{figure}[htpb]
			\centering
			\begin{tikzpicture}[scale=0.9]
				\draw[ultra thick](0,0) circle (3);
				\draw[ultra thick](0,0) circle (1);
				\draw[red,thick](0,-1)to[out=-90,in=90](0,-3);
				\draw[red,thick](0,-3)to[out=20,in=270](2.7,0)to[out=90,in=0](0,2.5)to[out=180,in=90](-2.5,0)to[out=-90,in=160](0,-3);
				\draw[red,thick](0,-1)to[out=-30,in=270](1.6,0)to[out=90,in=0](0,1.6)to[out=180,in=110](-1,0);
				\draw(-1,0)\rn (1,0)\rn  (0,1)\rn (0,-1)\rn (0.72,0.72)\rn  (3,0)\rn  (-3,0)\rn (0,3)\rn (0,-3)[red]\rn ;
				\draw[blue] (-.2,-2)node{$1^-$}(.4,-2)node{$1^+$} (-1.8,-1.5)node{$2$} (1.4,0.3)node{$3$};
				\draw[] (4.5,0) node {\Large{$\stackrel{\text{flip}}{\longrightarrow}$}};
				\begin{scope}[shift={(9,0)}]
					\draw[ultra thick](0,0) circle (3);
					\draw[ultra thick](0,0) circle (1);
					\draw[red,thick](-1,0)to[out=-90,in=180](0,-2)to[out=0,in=-90](2.1,0)to[out=90,in=0](0,2.1)to[out=180,in=160](0,-3);
					\draw[red,thick](0,-3)to[out=20,in=270](2.7,0)to[out=90,in=0](0,2.5)to[out=180,in=90](-2.5,0)to[out=-90,in=160](0,-3);
					\draw[red,thick](0,-1)to[out=-30,in=270](1.6,0)to[out=90,in=0](0,1.6)to[out=180,in=110](-1,0);
					\draw(-1,0)\rn (1,0)\rn  (0,1)\rn (0,-1)\rn (0.72,0.72)\rn  (3,0)\rn  (-3,0)\rn (0,3)\rn (0,-3)[red]\rn ;
					\draw[blue] (0,-2.2)node{$1^-$}(0,-1.7)node{$1^+$}(-1.8,-1.5)node{$2$} (1.4,0.3)node{$3$};
				\end{scope}	
			\end{tikzpicture}
			\caption{}
			\label{d=4,slfe-folded,mutation}
		\end{figure}
		 Let $(S,M,D)$ be a $d$-angulation as shown by the left picture of Figure~\ref{d=4,slfe-folded,mutation}. Let $(S,M,\mu_1(D))$ the flip of $D$ at arc $1$, see the right picture of Figure~\ref{d=4,slfe-folded,mutation}. Then $(Q,W)$ associated to $(S,M,D)$ is as follows (after identifying $1^+$ and $1^-$):
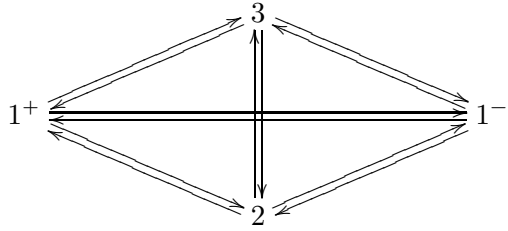
		\begin{align*}
			\xymatrix@C=6pc{& 3\ar@<0.3ex>[ld] \ar@<0.3ex>[dd]  \ar@<0.3ex>[rd]&\\
				1^{+} \ar@<0.3ex>[rr] \ar@<0.3ex>[ru] \ar@<0.3ex>[rd]&  &1^{-} \ar@<0.3ex>[ll]  \ar@<0.3ex>[lu] \ar@<0.3ex>[ld]  \\
				&2\ar@<0.3ex>[lu] \ar@<0.3ex>[uu]  \ar@<0.3ex>[ru] & }
		\end{align*}
		with
		\[
		W=a_{31^-}a_{1^-1^+}a_{1^+3}+a_{31^-}a_{1^-2}a_{23}
		+a_{21^+}a_{1^+1^-}a_{1^-2}+a_{21^+}a_{1^+3}a_{32} .
		\]
		Next consider $\mu_{1}(Q,W)$, we know that $A=\{ a_{21^+}\}$ and the new arrow pairs are %as follows by Definition~\ref{def:mut QsP}:
		\[
		\{a_{21^+}a_{1^+3},a_{31^+}a_{21^+}^{-1} , a_{21^+}a_{1^+1^{-}},a_{1^-1^+}a_{21^+}^{-1} \},
		\]
		\[
		\{ \textcolor{blue}{a_{21^+}a_{1^-3},a_{31^-}a_{21^+}^{-1}} , a_{21^+}a_{1^-2},a_{21^-}a_{21^+}^{-1} , \textcolor{blue}{a_{21^+}a_{1^-1^{+}},a_{1^+1^-}a_{21^+}^{-1}} \},
		\]
		\[
		\{
		a_{21^+}a_{1^+1^-}a_{21^+}^{-1},a_{21^+}a_{1^-1^+}a_{21^+}^{-1}
		\}.
		\]
		The superfluous arrow pairs are:
		\[
		\{a_{21^+}a_{1^+3},a_{31^+}a_{21^+}^{-1} ,a_{23},a_{32}, a_{21^+}a_{1^+1^{-}},a_{1^-1^+}a_{21^+}^{-1},a_{21^-},a_{1^-2} \},
		\]
		\[
		\{
		a_{21^+}a_{1^+1^-}a_{21^+}^{-1},a_{21^+}a_{1^-1^+}a_{21^+}^{-1},a_{21^+}a_{1^-2},a_{21^-}a_{21^+}^{-1} 
		\}.
		\]
		For convenience, we denote         \[c_{23}=a_{21^+}a_{1^-3},c_{32}=a_{31^-}a_{21^+}^{-1},c_{21^+}=a_{21^+}a_{1^-1^{+}},c_{1^+2}=a_{1^+1^-}a_{21^+}^{-1}.\]
		Then we have $\mu_{1}^{'}(Q)$ as follows:
		
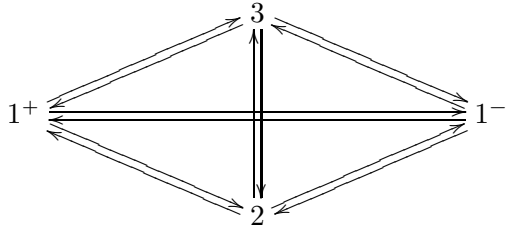
\begin{align*}
			\xymatrix@C=6pc{& 3\ar@<0.3ex>[ld] \ar@<0.3ex>[dd]  \ar@<0.3ex>[rd]&\\
				1^{+} \ar@<0.3ex>[rr] \ar@<0.3ex>[ru] \ar@<0.3ex>[rd]&  &1^{-} \ar@<0.3ex>[ll]  \ar@<0.3ex>[lu] \ar@<0.3ex>[ld]  \\
				&2\ar@<0.3ex>[lu] \ar@<0.3ex>[uu]  \ar@<0.3ex>[ru] & }
		\end{align*}
		where we can view $a_{21^+}^*,a_{1^+2}^*=(a_{21^+}^*)^{op}$
		as the arrow pair between $2$ and $1^-$, $\{c_{23},c_{32}\}$ the arrow pair between $2$ and $3$, $\{c_{21^{+}}c_{1^{+}2}\}$ the arrow pair between $2$ and $1^+$.
		
		Finally, we have 
		\begin{align*}
			\mu_{1}(W)=\dec_{ cyc} W + \sum_{\alpha, \varphi} \alpha \dec( \varphi \varphi^{op_M} ) \alpha^*,
		\end{align*}
		\begin{align*}
			\dec_{ cyc} W =&  a_{31^-}(1-a_{21^+}^{-1}a_{21^+})a_{1^-1^+}(1-a_{21^+}^{-1}a_{21^+})a_{1^+3}\\
			&+a_{31^-}(1-a_{21^+}^{-1}a_{21^+})a_{1^-2}a_{23}\\
			&+a_{21^+}a_{1^+1^-}(1-a_{21^+}^{-1}a_{21^+})a_{1^-2}\\
			&+a_{21^+}a_{1^+3}a_{32}.
		\end{align*}
		For the first part in $\mu_{1}(W)$, we know that	 
		\[
		a_{31^-}(1-a_{21^+}^{-1}a_{21^+})a_{1^-1^+}a_{1^+3}=a_{31^-}a_{1^-1^+}a_{1^+3}-c_{32}c_{21^+}a_{1^+3}
		\] 
		is the only path that containing no superfluous arrows.
		\begin{align*}
			\sum_{\alpha, \varphi} \alpha \dec( \varphi \varphi^{op_M} ) \alpha^* =&  a_{21^+}a_{1^+1^-}(1-a_{21^+}^{-1}a_{21^+})a_{1^-1^+}a_{21^+}^{*} \\
			&+a_{21^+}a_{1^-1^+}(1-a_{21^+}^{-1}a_{21^+})a_{1^+1^-}a_{21^+}^{*} \\
			&+a_{21^+}a_{1^+3}a_{31^+}a_{21^+}^{*}
			+a_{21^+}a_{1^-3}a_{31^-}a_{21^+}^{*}\\
			&+a_{21^+}a_{1^-2}a_{21^-}a_{21^+}^{*} .
		\end{align*}
		
		As for the second part in $\mu_{1}(W)$, we know that	 
		\[
		a_{21^+}a_{1^-1^+}a_{1^+1^-}a_{21^+}^{*}+a_{21^+}a_{1^-3}a_{31^-}a_{21^+}^{*}=c_{21^+}a_{1^+1^-}a_{21^+}^{*}+c_{23}a_{31^{-}}a_{21^+}^{*}
		\] 
		is the only path that containing no superfluous arrows.	 
		Thus, we know that 
		\begin{align*}
			\mu_{1}^{'}(W)=a_{31^-}a_{1^-1^+}a_{1^+3}-c_{32}c_{21^+}a_{1^+3} +c_{21^+}a_{1^+1^-}a_{21^+}^{*}+c_{23}a_{31^{-}}a_{21^+}^{*}. 
		\end{align*}
	\end{enumerate}
\end{example}
At the end of this section, we give a proposition regarding the number of arcs in a given $d$-angulation (\cite{fst,labardini,jacquet-malo_construction_2024}).
\begin{proposition}
	Let $(S,M)$ be a marked surface with d-angulation $D$. Denote by $g,b$ and $c$ the number of genus, boundary components and marked points respectively. Then the number $m$ of $d$-gons   and the number $n$ of arcs  in $(S,M,D)$ are:
	$$m=\frac{4g+2b+c-4}{d-2}, n=\frac{2dg+db+c-2d}{d-2}.$$
\end{proposition}
\begin{proof}
	The number of arcs in a triangulation of $(S,M)$ is given by $n_1=6g+3b+c-6$, and the number of triangles is $\frac{2n_1+c}{3}=4g+2b+c-4$. 
	We can view a $d$-gon as consisting of $d-2$ triangles, so the number of $d$-gons is $m=\frac{4g+2b+c-4}{d-2}$.
	Since each arc contribute 2 and each boundary segment contribute 1, we obtain the following expression for $n$:
	\begin{equation*}
		n=\frac{\frac{d}{d-2} (4g+2b+c-4)-c}{2}=\frac{2dg+db+c-2d}{d-2}.
	\end{equation*}	
\end{proof}

\section{$\Gamma_{1}$ is quasi-isomorphic to $\Gamma_ {2}$}

The following lemma is  useful when dealing with dg quiver algebra. Similar treatments can be found in  \cite{oppermann,keller-yang}.
\begin{lemma}[{\cite{keller-yang}}]\label{lemma3.1}
	Let $A=\k Q$ be a dg quiver algebra such that $d(\alpha)\in    \mathfrak{m},\forall \alpha \in Q_1$, where $\mathfrak{m}$ is the arrow ideal of $A$. Assume that $a,b\in Q_1$ with $d(a)=k_1b+p$, where $k_1\in \k,\ k_1\neq 0$, and $p\in \mathfrak{m}$, and there is no path containing arrow $b$ in $p$. Then the dg  quotient homomorphism
	\begin{equation}
		A\longrightarrow A/(a,d(a))
	\end{equation}
	is a quasi-isomorphism.
	Now, view $A/(a,d(a))$ as $A'=\k Q'$, where $Q'_1=Q_1\verb|\| \{a,b\}$, and the differential $d_{A'}$ of $A'$ is inherited from $d_A$. 
	By replacing the arrow $b$ appearing in the differential $d_{A}$ with $-\frac{1}{k_1}p$ and setting $a=0$, we obtain the following quasi-isomorphism:
	\begin{equation}
		A\longrightarrow A'.
	\end{equation}
\end{lemma}
\begin{proof}
	Firstly, consider $p\in \mathfrak{m}^2$  such that $\forall \alpha\in Q_1\verb|\|\{a,b\} ,d(\alpha)=p_1+p_2$, where $p_1\in \mathfrak{m} \verb|\|\mathfrak{m}^2$ and $p_2\in \mathfrak{m}^2$, and $a,b$ do not appear in $p_1$. Let $S=(a,d(a))$ and $S_i=\mathfrak{m}^i\cap S$. Then, we have $S_1=S$.
	
	It is clear that the complex  $S_i/S_{i+1}$ is isomorphic to
	\begin{equation}
		(U\oplus V)^{\otimes_{R}i}/U^{\otimes_{R}i},
	\end{equation}
	where $R=\oplus_{t\in Q_0}\k t$ and $U=\oplus_{t\in Q_1'}\k t$. The differential $d_U$ is defined by $d_U(\alpha)=p_1$ when $d_A(\alpha)=p_1+p_2$, with $p_1\in \mathfrak{m} \verb|\|\mathfrak{m}^2$ and $p_2\in \mathfrak{m}^2$.  By assumption, $U$ is complex. The complex $V=\k a\oplus \k b$ has the differential $d(a)=k_1 b$, so $V$ is contractible. 
	Thus $S_i/S_{i+1}$ is contractible. By the exact sequence
	\begin{equation}
		0\rga 	S_i/S_{i+1}\rga S_1/S_{i+1} \rga S_1/S_{i} \rga 0
	\end{equation}
	we know that $S/S_{i+1}$ is contractible. Therefore, $S$ is also contractible by Mittag–
	Leffler Lemma, which implies that $A\rga A/S$ is a quasi-isomorphism.
	
	Let $k_1=1$; otherwise replace  $b$ by a new arrow $k_1b$. Since there is no path containing $b$ in $p$, we can view $A$ as a dg quiver algebra generated by  $Q_1\verb|\| \{b \}\cup \{b+p \}$, where $|b+p|=|b|=|a|+1$. Therefore, we can focus on the case where $d(a)=b$.
	Next, collect all arrows $\{x_i\}_{i=1}^k$ such that $d(x_i)=t_ib+m_i$, where $t_i\neq 0 $ and there is no path containing $b$ in $m_i$.  Without loss of generality, set $t_i=1$. Define $y_i=x_i-a$ and view $A$ as the dg quiver algebra generated by $Q_1\verb|\|\{x_i\}_{i=1}^k\cup \{y_i \}_{i=1}^k$. We have now obtained dg quotient algebra $A'$ from the situation discussed above, where we have the following relations: 
	$\overline{x_i}=\overline{x_i-a}=\overline{y_i}$, 
	$d(\overline{x_i})=\overline{m_i}=d(\overline{y_i})$.

	Thus, $A/(a,d(a))$ is quasi-isomorphic to $A'$ the
	dg quotient algebra after the arrows are replaced. We only need to consider  the case where
	$d(a)=b$ such that  there is no path containing $b$ in $d(\alpha), \forall \alpha \in Q_1\verb|\|\{a\}$. 
	If there is an arrow $\alpha\in Q_1\verb|\|\{a,b\}$ such that the path $a$ is contained in $d(\alpha)$, then we have $d^2(\alpha)\neq 0$, because the path $b$ can only appear in $d(a)$. Finally, we arrive at the situation discussed above.
\end{proof}
\begin{corollary}\label{cor3.2}
	Let $A=\k Q$ be a non-positive dg quiver algebra, where $\mathfrak{m}$ is the arrow ideal of $A$ such that $\forall \alpha \in Q_1 ,d(\alpha)\in \mathfrak{m}$,  and there are no cycles of degree $0$ in $\k Q$. If $a,b\in Q_1$, and $d(a)=b+p$ with $p\in \mathfrak{m}^2$, then we have the following quasi-isomorphism
	\begin{equation}
		A\longrightarrow A/(a,d(a)).
	\end{equation}
\end{corollary}

We give some examples to show that the conditions in Corollary~\ref{cor3.2} is necessary.
\begin{example}
	\begin{enumerate}
		\item
		Let $A=\k Q$ be a dg quiver algebra with relation $bc=0$. The quiver $Q$ as follows, where $|\alpha|=-1,|a|=|b|=|c|=0$ and $d(\alpha)=ab,d(a)=d(b)=d(c)=0$:
		\begin{align*}
			\xymatrix@C=4pc{1 \ar@<0.4ex>^{\alpha}@/_2pc/[rr]  \ar[r]^{a} & 2\ar[r]^{b} & 3\ar[r]^{c}&4    }.
		\end{align*}
		
		 The map $A\rga A/(\alpha,d(\alpha))$ is not a quasi-isomorphism, because $H^{-1}A\neq 0$ (because $d(\alpha c)=0$), and the cohomology of the latter only concentrate on degree $0$.
		\item 
		Let $A=\k Q$ be a dg quiver algebra with $Q$ the quiver as follows, where $|\alpha|=-1,|a|=|b|=0$ and  $d(\alpha)=aba,d(a)=d(b)=0$:
		\begin{align*}
			\xymatrix@C=6pc{1 \ar@<0.4ex>_{\alpha}@/_2pc/[r]  \ar@<0.3ex>[r]^{a} & 2\ar@<0.3ex>[l]^{b}    }.     
		\end{align*}
		Let $\overline{A}=A/(\alpha ,d(\alpha))$. We can view $\overline{A}$ as the following dg quiver algebra with relation $aba=0$,
		\begin{align*}
			\xymatrix@C=6pc{1   \ar@<0.3ex>[r]^{a} & 2\ar@<0.3ex>[l]^{b}    }.          
		\end{align*}
		 It is known that $H^{-1}\overline{A}=0$. However, we have  $d(\alpha ba-ab\alpha)=0$ in 
		 $A$, and for any path $p$ with $|p|=-2$ , the length of any path in $d(p)$ is greater than 3. This means that $\alpha ba-ab\alpha \notin \im d$. So $H^{-1}A\neq 0$ and hence the map $A\rga \overline{A}$ is not a quasi-isomorphism.
	\end{enumerate}
\end{example}

If the approximation in (\ref{1.9}) is not minimal, the new silting objects are not basic. In this case, the induced endomorphism  dg algebras become more complicated. In particular, the mutations of two dg quiver algebras which are quasi-isomorphic may no longer be quasi-isomorphic. See the example below.
\begin{example}
	Let $A=\k Q$ be a dg quiver algebra with $Q$ the quiver as follows, where $|\alpha|=-1,|a|=|b|=|c|=0$ and $d(\alpha)=bc-a$.  By Corollary~\ref{cor3.2}, it is quasi-isomorphic to $A'=\k Q'$, where $Q'=Q\verb|\|\{\alpha,a\}$.
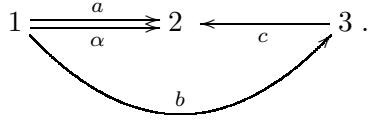
	\begin{align*}
		\xymatrix@C=4pc{1 \ar@<0.4ex>^{b}@/_3pc/[rr] \ar@<-0.3ex>[r]_{\alpha} \ar@<0.3ex>[r]^{a} & 2\ & \ar[l]^{c}3   }. 
	\end{align*}
	A direct calculation shows that $\dim H^0(\mathscr{E}nd (\mu_1 (A)))=8$ and $\dim H^0(\mathscr{E}nd(\mu_1 (A')))=5$.
\end{example}
Let $(Q,W)$ be a QsP given by a $d$-angulation $(S,M,D)$. Let $\Gamma$ be the Ginzburg algebra associated to $(Q,W)$.  Denote by $\Gamma_1$ and $\Gamma_2$ the Ginzburg algebra associated to $\mu_{i}(Q,W)$ and $\mu^{'}_{i}(Q,W)$. Denote by $d,d_1,d_2$ the differential of  $\Gamma$, $\Gamma_1$, $\Gamma_2$, respectively.

The main result in this section is the following theorem which says the two Ginzburg algebras are quasi-isomorphic. The proof will be given in Section~5.
\begin{theorem}\label{mainthm}
	There is a quasi-isomorphism $\Gamma_1 \rga \Gamma_2$.
\end{theorem}

By Theorem~\ref{mainthm} and Theorem~\ref{thm2.16}, we have the following main result in our paper.
\begin{theorem}\label{commutative diagram}
		We have the following
	commutative diagram up to quasi-isomorphism
	$$\xymatrix@C=8pc{(S,M,D) \ar[r]\ar[d]^{\text{flip}}&(Q,W)\ar[d]^{\text{mutation}}\\
		(S,M,\mu_{i}(D)) \ar[r]&\mu_{i}(Q,W).}$$
\end{theorem}
\section{The number of complements for almost complete $(d-2)$-cluster tilting objects}

    Let $A$ be a dg algebra satisfying the properties $(\star)$. We have the following theorem regarding $\per (A)$ and $\c_A:=\per(A)/\pvd(A).$
   \begin{theorem}[{\cite{guo2,iyama-yang}}]\label{thm4.1}
	  Let $ T\in  \per (A)$ be a silting object,  then the image $\pi (T)$ under the projection functor $\pi : \per (A)\rga\c_A$ is a $(d-2)$-cluster tilting object in  $\c_A$.
   \end{theorem}
   \begin{theorem}[{\cite{guo2}}]\label{thm4.2}
   	Let $A=\Gamma_d(Q,W)$ be a Ginzburg  dg algebra such that there are no loops of  $Q$ at vertex $i$. Then the image of mutation sequence (Theorem~\ref{thm:silt-mutation}) \ \ 
   	 $\{\cdots,L^tA_i,\cdots,L^1A_i,A_{i},R^1A_i,\cdots,R^tA_i,\cdots \}$ at $i$ in $\c_A$ is periodic with period  $d-1$. Furthermore,  the almost complete $(d-2)$-cluster tilting object $\pi(A/A_i)$ has
   	exactly $d-1$ complements in $\c_A$, which correspond to all the objects occurring in the sequence.
   \end{theorem}

   For any silting object $T\in \per(A)$,   $\f=T\ast T[1]\ast\cdots\ast T[d-2] $ is a fundamental domain of $\c_A$ in $\per (A)$.
   \begin{theorem}[{\cite{iyama-yang}}]\label{thm4.3}
   	The functor $\pi:\per(A)\rga \c_A$ restricts to an equivalence $\f \rga \c_A$ of additive categories.
   \end{theorem} 

  Let $(Q,W)$ be a QsP associated to a $d$-angulation $(S,M,D)$ and $\Gamma$ the $d$-dimensional Ginzburg algebra associated to $(Q,W)$.
  By  Definition~\ref{def:QW} and Definition~\ref{def:Ginz}, we have
  \begin{equation}
  	H^p\Gamma =0 \ ,  \forall p>0.
  \end{equation}  
  Moreover, the Ginzburg algebra $\Gamma$ is homologically smooth and  $(m + 2)$-Calabi–Yau (see \cite{van,guo2}). 
  \begin{lemma}\label{no 0 loops}
    There are no loops of degree $0$ in $\Gamma$.
    \begin{proof}
      The arrow $\alpha$ of degree $0$ originates from two adjacent
  	arcs, $i,j$ in a single  $d$-gon of $(S,M,D)$, where $i\neq j$, as there are no punctures in our setting (see Figure~\ref{fig7}).
  	\begin{figure}[htpb]
  		\centering
  		\begin{tikzpicture}[scale=0.5]
  			\draw[ultra thick,dashed](0,0) circle (3);
  			\draw[red,thick](0,-3)to (0,0);
  			\draw[blue] (-.3,-2)node{$i$};
  			\draw[blue] (.3,-2)node{$j$};
  			\draw[red] (.4,0.4)node{$O$};
  			\draw(0,0)\rn (0,-3)\rn ;
  		\end{tikzpicture}
  		\quad \quad \quad
  		\begin{tikzpicture}[scale=0.4]
  			\draw[thick,dashed] (0,3)to[out=0,in=90](3,0)to[out=-90,in=0](0,-3);
  			\draw[red,thick](-3,0)to(0,3);
  			\draw[red,thick](-3,0)to(0,-3);
  			\draw[blue] (-1.3,1.2)node{$i$};
  			\draw[blue] (-1.3,-1.2)node{$j$};
  			\draw[red] (-2.4,0)node{$O$};
  			\draw(0,3)\rn (0,-3)\rn (-3,0)\rn;
  		\end{tikzpicture}
  		\caption{$i=j$ and $|\alpha |=0$}\label{fig7}
  	\end{figure}
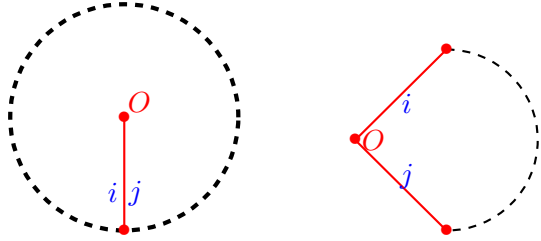
    \end{proof}
  \end{lemma}
  \begin{lemma}
  	$dim_k H^0\Gamma <\infty$.
  \end{lemma}
  \begin{proof}
  	It follows from the definition that $\Gamma$ is a non-positive dg quiver algebra. 	
  	If $dim_k H^0\Gamma =\infty$, there exists a cyclic path $p=\alpha_1\alpha_2 \cdots \alpha_k\neq 0$ with the least number of arrows such that $|p|=0$, since there are only finite arrows. By Lemma~\ref{no 0 loops}, $k\geq 2$ and $ \alpha_{s}\neq \alpha_{t}$ for any  $s\neq t, s,t \in \{1,2,\cdots ,k\}$. If  two adjacent arrows in $p$ come from the same $d$-gon, 
  	then there exists $\beta$ such that $d(\beta)=\alpha_t \alpha_{t+1}$ (see Figure~\ref{fig8}). This contradicts the assumption that $p\neq 0$.
  	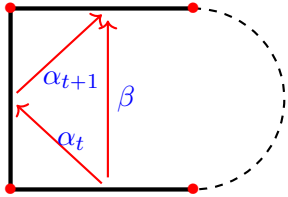
\begin{figure}[htpb]
  		\centering
  		\begin{tikzpicture}[scale=0.8]
  			\draw[ultra thick](0,0)to (3,0)  (0,0)to (0,-3)  (0,-3)to(3,-3);
  			\draw[thick,dashed] (3,0)to[out=0,in=90](4.5,-1.5)to[out=-90,in=0](3,-3);
  			\draw[red,thick,->] (1.5,-2.9)to(.1,-1.6) ;
  			\draw[red,thick,->]  (.1,-1.4)to(1.5,-0.1);
  			\draw[red,thick,->]  (1.6,-2.8)to(1.6,-0.2);
  			\draw[blue] (1,-2.2)node{$\alpha_t$};
  			\draw[blue] (1,-1.2)node{$\alpha_{t+1}$};
  			\draw[blue] (1.9,-1.5)node{$\beta$};
  			\draw(0,0)\rn (0,-3)\rn  (3,0)\rn (3,-3)\rn;
  		\end{tikzpicture}
  		\caption{The $d$-gon contain $\alpha_t,\alpha_{t+1}$} \label{fig8}
  	\end{figure}
  	
  	Therefore, $\alpha_t$ and $\alpha_{t+1}$  lie in adjacent $d$-gons for all $t\in\{1,2,\cdots,k \}$ with $\alpha_{k+1}=\alpha_{1}$. 
  	Then the common endpoints of $\alpha_t$ is a marked point locating in the interior of $S$ (see Figure~\ref{fig9}), a contradiction.
  	Thus, we conclude that
  	$dim_k H^0\Gamma <\infty$. 
  	\begin{figure}[htpb]
  		\centering
  		\begin{tikzpicture}[scale=0.7]
  			\draw[ultra thick](0,0)to (3,0)  (0,0)to (1.3,2.4)  (0,0)to(2.1,1.5) (0,0)to (2.1,-1.5);
  			\draw[thick,dashed] (0,1)to[out=180,in=90](-1,0)to[out=-90,in=240](.6,-.7);
  			\draw[red,thick,->] (1.6,-1)to(1.8,-.2) ;
  			\draw[red,thick,->] (1.8,.2) to(1.6,0.9) ;
  			\draw[red,thick,->] (1.5,1.2) to(1,1.6) ;
  			\draw[blue] (2.1,-.7)node{$\alpha_k$};
  			\draw[blue] (2.3,.6)node{$\alpha_{1}$};
  			\draw[blue] (1.7,1.8)node{$\alpha_{2}$};
  			\draw(0,0)\rn ;
  			\draw[red] (-.4,0.2)node{$O$};
  		\end{tikzpicture}
  		\caption{}\label{fig9}
  	\end{figure}
  \end{proof}
 
  So the $d$-dimensional Ginzburg algebra $\Gamma$ associated with $(Q,W)$, and consequently with $(S,M,D)$, satisfies the properties ($\star$).
  Therefore, we obtain the generalized higher cluster category corresponding to
  $(Q,W)$
  \begin{align}\label{4.2}
  	\c_{\Gamma}:=\per(\Gamma)/\pvd(\Gamma)
  \end{align}  
  and the projection functor $\pi : \per (\Gamma)\rga\c_{\Gamma}$. By Theorem~\ref{thm4.1}, we know that $\pi(\Gamma)$ is a $(d-2)$-cluster tilting object in $\c_{\Gamma}$.
  
  \begin{definition}
    We call the $d$-angulation $(S,M,D)$ is a  geometric
  model for the generalized higher cluster category $\c_{\Gamma}$.
  \end{definition}
  
  Denote by $C$ the set of all $(d-2)$-cluster tilting objects obtained by finite mutations of $\pi(\Gamma)$, and by $C'$ the set of all $d$-angulations  obtained by finite flips of $(S,M,D)$. 
  \begin{conjecture}
    There is a bijection between $C$ and $C'$ which sends $\pi(\Gamma)$ to $(S,M,D)$.
  \end{conjecture}
  The conjecture is true for some special $(S,M,D)$ corresponding to (classical) $m$-cluster categories of type $A_n$ and type $\tilde{A_n}$ (see \cite{BM02,BT20,LJM}). 
  \begin{example}
    	\begin{enumerate}
    		\item $d=4$, $\c_{A_4}$ has a  geometric model as the left picture in  Figure~\ref{fig11}.
    		\item $d=4$, $\c_{\tilde{A_4}}$ has a  geometric model as  the right picture in  Figure~\ref{fig11}.
    	\end{enumerate}	
    	\begin{figure}[htpb]
    		\centering
    		\begin{tikzpicture}[scale=0.4]
    			\draw[ultra thick](0,0) circle (3);
    			\draw  (2.6,1.5)\rn (-2.6,1.5)\rn(2.6,-1.5)\rn(-2.6,-1.5)\rn(2.6,1.5)\rn (3,0)\rn (0,3)\rn (-3,0)\rn (0,-3)\rn (1.5,2.6)\rn  (-1.5,2.6)\rn(1.5,-2.6)\rn(-1.5,-2.6)\rn;
    			\draw[blue, thick](-2.6,-1.5)to(-1.5,2.6); 
    			\draw[blue, thick](-2.6,-1.5)to(1.5,2.6);
    			\draw[blue, thick](-2.6,-1.5)to(3,0); 
    			\draw[blue, thick](-2.6,-1.5)to(1.5,-2.6);  
    			\begin{scope}[shift={(9,0)}]
    				\draw[ultra thick](0,0) circle (3);
    				\draw[ultra thick](0,0) circle (1);
    				\draw   (3,0)\rn (0,3)\rn (-3,0)\rn (0,-3)\rn (1,0)\rn (0,1)\rn (-1,0)\rn (0,-1)\rn;
    				\draw[blue, thick](0,3)to[out=-40,in=80](1,0); 
    				\draw[blue, thick](0,3)to[out=220,in=100](-1,0); 
    				\draw[blue, thick](0,-3)to[out=40,in=-800](1,0); 
    				\draw[blue, thick](0,-3)to[out=140,in=-100](-1,0);		
    			\end{scope}	
    		\end{tikzpicture}
    		\caption{$d=4$, $\c_{A_4}$ and $\c_{\tilde{A_4}}$}\label{fig11}
    	\end{figure}
     \end{example}
  
  For  $T=\mu_{s_k}\cdots\mu_{s_2}\mu_{s_1}(\pi(\Gamma)) \in C$, we denote by  $(S,M,\mu_{s_k}\cdots\mu_{s_2}\mu_{s_1}(D))$  the corresponding $d$-angulation of $T$. Then, we have the following theorem regarding the number of complements for
    $T$. 
    \begin{lemma}\label{lemma4.9}
      We have the following commutative diagram: 
      \begin{align*}
    		\xymatrix@C=4pc{(S,M,D) \ar[r] \ar[d]^{\text{flips}}&(Q,W)\ar[d]^{ \text{mutations}}\ar[r]^{\text{Ginzburg}}&\Gamma \ar[d]^{ \text{mutations}}\ar[r]& \pi(\Gamma)\ar[d]^{ \text{mutations}}\\
    			(S,M,\mu_{s_k}\cdots\mu_{s_2}\mu_{s_1}(D)) \ar[r] & \mu^{'}_{s_k}\cdots\mu^{'}_{s_2}\mu^{'}_{s_1}(Q,W) \ar[r]^-{\text{Ginzburg}} & \Gamma^{'} \ar[r] & T.}
    	\end{align*} 
    Therefore, the indecomposable direct summand $T_j$ of $T$  correspond to an arc denoted by $j$ on $(S,M,\mu_{s_k}\cdots\mu_{s_2}\mu_{s_1}(D))$.
    \end{lemma} 
    \begin{proof}
       Let $\Gamma=\oplus(\Gamma)_j$ as a direct sum of indecomposable objects in $\per (\Gamma)$, and  $\mu_{i}(\Gamma)=\oplus_{j\neq i}(\Gamma)_j\oplus R((\Gamma)_i)$  is the mutation of $\Gamma$ at $i$ in $\per (\Gamma)$. Denote by $\Gamma_2$ the derived endomorphism dg algebra of $\mu_i(\Gamma)$ in $\per (\Gamma)$. Since $\per (\Gamma) \simeq \per (\Gamma_2)$, which gives an equivalence between $\add \mu_i(\Gamma)$ and $\add \Gamma_2$.
    	By Theorem~\ref{mainthm}, Theorem~\ref{thm2.16} and Theorem~\ref{thm:1.9}, we have the following commutative diagram:
    	\begin{align*}
    		\xymatrix@C=4pc{(S,M,D) \ar[r] \ar[d]^{\text{flip}}&(Q,W)\ar[d]^{ \text{mutation}}\ar[r]^{\text{Ginzburg}}&\Gamma \ar[d]^{ \text{silting mutation}}\\
    			(S,M,\mu_i(D)) \ar[r] &\mu^{'}_{i}(Q,W) \ar[r]^{\text{Ginzburg}}&\Gamma_2 .}
    	\end{align*}
    	Applying the projection functor $\pi$ to the mutation triangle in $\per (\Gamma)$
    \begin{equation*}
              \begin{aligned}
    			(\Gamma)_i\rga &S\rga R((\Gamma)_i) \rga (\Gamma)_i[1],
    		\end{aligned}
    \end{equation*}
     we have the  mutation triangle in $\c_{\Gamma}$ by Theorem~\ref{thm4.1} and Theorem~\ref{thm4.3}   	
    	\begin{equation*}
    		\begin{aligned}
    			\pi((\Gamma)_i)\rga &\pi(S)\rga \pi(R((\Gamma)_i)) \rga \pi((\Gamma)_i)[1].
    		\end{aligned}
    	\end{equation*} 
    It follows that $R(\pi((\Gamma)_i))=\pi(R((\Gamma)_i))$.    	
    	So the assertion follows by repeating this process.
    \end{proof}
    
    \begin{theorem}\label{thm4.10}
    	Let $T$ be an cluster tilting object in $C$ and $T_j$ be an  indecomposable direct summand of $T$. The almost complete $(d-2)$-cluster tilting object $T/T_j$ has
    	exactly $d-1$ complements if the corresponding arc  $j$ on $(S,M,\mu_{s_k}\cdots\mu_{s_2}\mu_{s_1}(D))$ is not  self-folded.
    \end{theorem}
    \begin{proof}
    	 If the corresponding arc  $j$ on $(S,M,\mu_{s_k}\cdots\mu_{s_2}\mu_{s_1}(D))$ is not  self-folded, then there are no loops at $j$ in  $\mu^{'}_{s_k}\cdots\mu^{'}_{s_2}\mu^{'}_{s_1}(Q,W)$ by  Definition~\ref{def:QW} and Lemma~\ref{lemma4.9}. It follows from Theorem~\ref{thm4.2} that the almost complete $(d-2)$-cluster tilting object $T/T_j$ has
    	exactly $d-1$ complements. 
    \end{proof}  
    \begin{remark}
       When $j$ is  not a self-folded arc, we flip $j$ to itself after $(d-1)$ times. However, when $j$ is a self-folded arc, the situation is different (see Figure~\ref{fig10}). The result in Theorem~\ref{thm4.10} for self-folded arc $j$ is not right (see Example~\ref{ex.4.12}). 
    \end{remark}
    \begin{figure}[htpb]
    	\centering
    	\begin{tikzpicture}[scale=0.5]
    		\draw[ultra thick](0,0) circle (3);
    		\draw[ultra thick](0,0) circle (1.5);
    		\draw[blue,thick](0,-1.5)to(0,-3.0); 
    		\draw(-1.5,0)\rn (1.5,0)\rn (0,-1.5)\rn (-3,0)\rn (3,0)\rn (0,-3)[red]\rn  (0,1.5)\rn;
    		\draw[purple,thick,dashed,->] (-0.1,-1.7)to[out=180,in=-90](-1.7,-0.1);
    		\draw[purple,thick,dashed,->] (0.1,-2.8)to[out=0,in=-90](2.8,-0.1);
    		\draw[] (4.5,0) node {\Large{$\Longrightarrow$}};
    		\begin{scope}[shift={(9,0)}]
    			\draw[ultra thick](0,0) circle (3);
    			\draw[ultra thick](0,0) circle (1.5);
    			\draw[blue,thick](-1.5,0)to[out=260,in=180](0,-2.4)to[out=0,in=250](3,0); 
    			\draw(-1.5,0)\rn (1.5,0)\rn (0,-1.5)\rn (-3,0)\rn (3,0)\rn (0,-3)[red]\rn (0,1.5)\rn;
    			\draw[purple,thick,dashed,->] (-1.7,0.1)to[out=90,in=180](-0.1,1.7);
    			\draw[purple,thick,dashed,->] (2.8,0.1)to[out=90,in=0](0,2.8)to[out=180,in=90](-2.8,0.1);
    			\draw[] (4.5,0) node {\Large{$\Longrightarrow$}};
    			\begin{scope}[shift={(9,0)}]
    				\draw[ultra thick](0,0) circle (3);
    				\draw[ultra thick](0,0) circle (1.5);
    				\draw[blue,thick](0,1.5)to[out=170,in=90](-2,0)to[out=-90,in=180](0,-2)to[out=0,in=-90](2,0)to[out=90,in=0](0,2.2)to[out=180,in=70](-3,0); 
    				\draw(-1.5,0)\rn (1.5,0)\rn (0,-1.5)\rn (-3,0)\rn (3,0)\rn (0,-3)[red]\rn (0,1.5)\rn;		
    			\end{scope}	
    		\end{scope}	
    	\end{tikzpicture}
    	\caption{}\label{fig10}
    \end{figure}
\begin{example}\label{ex.4.12}
	Suppose that $d=5$. Let $(S,M,D)$ be a $d$-angulation as shown by Figure~\ref{fig12}. Then the Ginzburg algebra $\Gamma$ associated to $(S,M,D)$ has the underlying graded quiver as follows
	
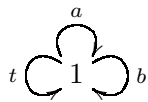
\begin{align*}
		\xymatrix@C=6pc{\ar@(lu,ru)^{a}
			1 \ar@(ru,rd)^{b} \ar@(lu,ld)_{t}  }
	\end{align*}
	with $|a|=-1$, $|b|(=|a^{op}|)=-2$ and $|t|=-4$.  The differential takes the following values
	\[
	 d(a)=0,\quad d(b)=0,\quad d(t)=ab-ba.
	\]
	The algebra $\Gamma$ is an indecomposable object in  $\per (\Gamma)$. Let $\Gamma$ be the silting object $T$. Then $L^{k}T$ is isomorphic to  $T[k]$ and $R^{k}T$  is isomorphic to $T[-k]$ for all $k\geq 0$.
	It is clear that $H^{-r}\Gamma \neq 0$ for all $r \geq 0$, because $d(a^r)=0$ and $a^r \notin \im d$. We claim that the image of $T$ in $\c_{\Gamma}$ is not isomorphic to the image of $T[4]$ in $\c_{\Gamma}$.  Otherwise, assume that $\pi (T)$ is isomorphic to $\pi (T[4])$. Then the following
	isomorphisms hold
	\begin{equation*}
		\begin{aligned}
		\Hom_{\c_{\Gamma}}(\pi (T),\pi (T)[1]) \simeq	\Hom_{\c_{\Gamma}}(\pi (T[4]),\pi (T)[1])\simeq\Hom_{\c_{\Gamma}}(T[3],T).
		\end{aligned}
	\end{equation*}
	By Theorem~\ref{thm4.3}, we have  the following
	isomorphisms
	\begin{equation*}
		\begin{aligned}
			\Hom_{\c_{\Gamma}}(\pi (T),\pi (T)[1]) \simeq\Hom_{\c_{\Gamma}}(T[3],T)
			\simeq \Hom_{\per (\Gamma)}(T[3],T)\simeq H^{-3}\Gamma.
		\end{aligned}
	\end{equation*}
	The left  term of these isomorphisms vanishes since $\pi (T)$ is a $3$-cluster tilting object in $\c_{\Gamma}$. Therefore, we 
	obtain a contradiction.
	 \begin{figure}[htpb]
		\centering
		\begin{tikzpicture}[scale=0.5]
			\draw[ultra thick](0,0) circle (3);
			\draw[ultra thick](0,0) circle (1.5);
			\draw[red,thick](0,-1.5)to(0,-3.0); 
			\draw   (0,-3)\rn  (0,1.5)\rn (0,-1.5)\rn;
			\draw[blue] (0.2,-2)node{$1$};
		\end{tikzpicture}
		\caption{}\label{fig12}
	\end{figure}
\end{example}  
%\appendix
\section{Proofs of Theorem~\ref{thm2.16} and  Theorem~\ref{mainthm}}
\subsection{Proof of Theorem~\ref{thm2.16}}

	Since there are many arrow pairs, which might seem confusing, we view each arrow pair as an edge and only display those edges associated to the two $d$-gons that contain $i$. Thus we get two regular graphs corresponding to the two 
	$d$-gons containing arc $i$. Assume $i=1$, then we have the following cases.
	\begin{enumerate}[\text{Case}~(1)]
		\item $|A|=0$. We just change the degree of arrows, and the assertion follows immediately.
		\item $|A|=1$. Suppose without loss of generality that $A=\{a_{21^+} \}$.  First, consider the case when the arc  $1$ is not  self-folded.  Denote by $S_1$ and $S_2$ the two regular graphs, and then glue $1^{+}$ and $1^{-}$ into $1$. Here, $S_1$ corresponds to the regular graph of the point set $M_1=\{1^+,2, \cdots , k-1,k\}$  and $S_2$ corresponds to regular graph of the point set $M_2=\{1^-,2', \cdots , l'-1,l'\}$ (see the left and right part of the following graph).
		
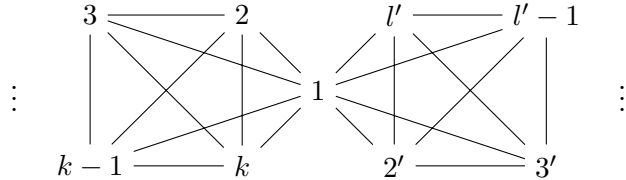
\begin{align*}
			\begin{tikzpicture}[baseline=-3.5pt][scale=0.8]
				\node (1) at (0,0) {$1$};
				\node (2) at (-1,1) {$2$};
				\node (3) at (-3,1) {$3$};
				\node (s1) at (-4,0) {$\vdots$};
				\node (k-1) at  (-3,-1) {$k-1$};
				\node (k) at  (-1,-1) {$k$};
				\node (2') at (1,-1) {$2'$};
				\node (3') at (3,-1) {$3'$};
				\node (s2) at (4,0) {$\vdots$};
				\node (l'-1) at  (3,1) {$l'-1$};
				\node (l') at  (1,1) {$l'$};
				\draw [-] (1) to   (2) to (3) to  (k-1) to (k) to (1);
				\draw [-] (2) to   (k-1) to (1) to  (3) to (k) to (2);
				\draw [-] (1) to   (2') to (3') to  (l'-1) to (l') to (1);
				\draw [-] (l'-1) to (1)to  (3') to (l') to (2')to (l'-1);
			\end{tikzpicture}
		\end{align*}
		The new arrow pairs are $\{b_{2t},b_{t2} \}_{t\in P}$, where $P=(M_1\cup M_2)\verb|\|\{1^+,1^-,2\}$. To be more precisely,
		\begin{equation*}
			\begin{aligned}
				b_{2j}=[a_{21^+}a_{1^+j}]=a_{21^+}a_{1^+j} ,\ b_{j2}=[a_{j1^+}a_{21^+}^{-1}]=a_{j1^+}a_{21^+}^{-1},\ \forall j \in 	\{ 3,\cdots , k \},\\
				b_{2j}=[a_{21^+}a_{1^-j}]=a_{21^+}a_{1^-j} ,\ b_{j2}=[a_{j1^-}a_{21^+}^{-1}]=a_{j1^-}a_{21^+}^{-1},\ \forall j \in 	\{ 2',\cdots , l' \}.
			\end{aligned}	
		\end{equation*}
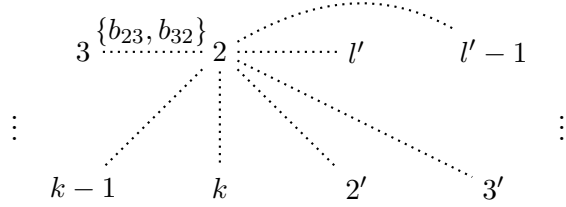
		\begin{align*}
			\begin{tikzpicture}[scale=0.9][baseline=-3.5pt]
				\node (2) at (-1,1) {$2$};
				\node (3) at (-3,1) {$3$};
				\node (s1) at (-4,0) {$\vdots$};
				\node (k-1) at  (-3,-1) {$k-1$};
				\node (k) at  (-1,-1) {$k$};
				\node (2') at (1,-1) {$2'$};
				\node (3') at (3,-1) {$3'$};
				\node (s2) at (4,0) {$\vdots$};
				\node (l'-1) at  (3,1) {$l'-1$};
				\node (l') at  (1,1) {$l'$};
				\draw [-, dotted, thick]  (2) to (3) ;
				\draw [-, dotted, thick]  (2) to (k) ;
				\draw [-, dotted, thick]  (2) to (k-1) ;
				\draw [-, dotted, thick]  (2) to (2') ;
				\draw [-, dotted, thick]    (2) to [out=30,in=150] (l'-1) ;
				\draw [-, dotted, thick]  (2) to (l') ;
				\draw [-, dotted, thick]  (2) to (3') ;
				\node (s) at (-2,1.3) {$\{b_{23},b_{32} \}$};
			\end{tikzpicture}
		\end{align*}
		
		The arrow pair $\{a_{21^+},a_{1^+2}\}$ is replaced by the arrow pair $\{a_{21^+}^*,a_{1^+2}^*=(a_{21^+}^*)^{op}\}$  in the opposite direction. 
		
		By the Definition~\ref{mutation}, we need to remove the following superfluous arrow pairs:
		\[
		\{a_{2j},a_{j2},b_{2,j},b_{j2}\}_{j\in \{3,\cdots , k \}}.
		\]
		Then we get the local graph (arrow pairs) in $\mu^{'}_{1}(Q,W)$ as below.
		
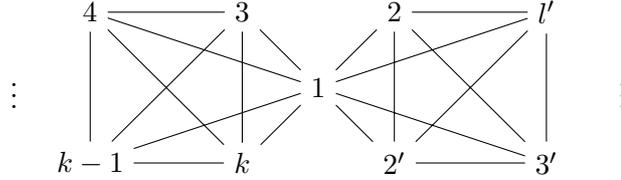
\begin{align*}
			\begin{tikzpicture}[baseline=-3.5pt][scale=0.8]
				\node (1) at (0,0) {$1$};
				\node (2) at (-1,1) {$3$};
				\node (3) at (-3,1) {$4$};
				\node (s1) at (-4,0) {$\vdots$};
				\node (k-1) at  (-3,-1) {$k-1$};
				\node (k) at  (-1,-1) {$k$};
				\node (2') at (1,-1) {$2'$};
				\node (3') at (3,-1) {$3'$};
				\node (s2) at (4,0) {$\vdots$};
				\node (l'-1) at  (3,1) {$l'$};
				\node (l') at  (1,1) {$2$};
				\draw [-] (1) to   (2) to (3) to  (k-1) to (k) to (1);
				\draw [-] (2) to   (k-1) to (1) to  (3) to (k) to (2);
				\draw [-] (1) to   (2') to (3') to  (l'-1) to (l') to (1);
				\draw [-] (l'-1) to (1)to  (3') to (l') to (2')to (l'-1);
			\end{tikzpicture}
		\end{align*}
		By Step~4 of Definition~\ref{mutation}, we know $\mu^{'}_{1}(Q,W)$ is a quiver with superpotential associated to 
		$(S,M,\mu_1 (D))$: $3$-cycles in a $d$-gon in $(S,M,\mu_1 (D))$ are all provided by those in $\mu_{1}(W)$, while all the $2$-cycles in $\mu_{1}(W)$ are removed by the Step~4 of Definition~\ref{mutation}.
				
		If the arc $1$ is self-folded then the local graph is as follows. Denote by $S_1=S_2$ the regular graph determined by the point set $M=\{1^+,2, \cdots , k,1^-,2',\cdots ,l'\}$.  
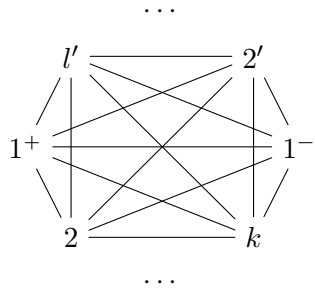
		\begin{align*}
			\begin{tikzpicture}[baseline=-2.5pt,scale=0.6]
				\node (1^+) at (-3,0) {$1^+$};
				\node (2) at (-2,-2) {$2$};
				\node (k) at (2,-2) {$k$};
				\node (s1) at (0,3) {$\cdots$};
				\node (1^-) at  (3,0) {$1^-$};
				\node (2') at (2,2) {$2'$};
				\node (s2) at (0,-3) {$\cdots$};
				\node (l') at  (-2,2) {$l'$};
				\draw [-] (1^+) to   (l') to (2') to  (1^-) to (k) to (2)to (1^+);
				\draw [-] (2) to   (l') to (1^-) to  (2) ;
				\draw [-] (2') to   (k) to (1^+) to  (2') ;
				\draw [-] (l')to(k) ;
				\draw [-] (2)to(2') ;
				\draw [-] (1^+)to(1^-) ;
			\end{tikzpicture}
		\end{align*}
		For convenience, we denote $\alpha=a_{21^+},\varphi=a_{1^+ 1^-},\varphi^{op}=a_{1^+ 1^-}^{op}=a_{1^- 1^+}$. The new arrow pairs are  
		\[\{b_{2t},b_{t2} \}_{t\in P}\cup \{c_{2t},c_{t2} \}_{t\in P}\cup \{a_{21^+}a_{1^-2} ,a_{21^-}a_{21^+}^{-1}\}, \]  
		\[ \{\alpha \varphi, \varphi^{op} \alpha^{-1}\}\cup \{\alpha \varphi^{op},\varphi \alpha^{-1}\} \cup \{\alpha \varphi \alpha^{-1},\alpha \varphi^{op} \alpha^{-1} \},
		\] where $P=M\verb|\|\{1^+,1^-,2\}$. To be more precisely,		
		\begin{equation*}
			\begin{aligned}
				b_{2j}=[a_{21^+}a_{1^+ j}]=a_{21^+}a_{1^+j} ,\ \ b_{j2}=[a_{j1^+}a_{21^+}^{-1}]=a_{j1^+}a_{21^+}^{-1},\ \forall j \in 	P,\\
				c_{2j}=[a_{21^+}a_{1^- j}]=a_{21^+}a_{1^-j} ,\ \ c_{j2}=[a_{j1^-}a_{21^+}^{-1}]=a_{j1^-}a_{21^+}^{-1},\ \forall j \in 	P.
			\end{aligned}	
		\end{equation*}
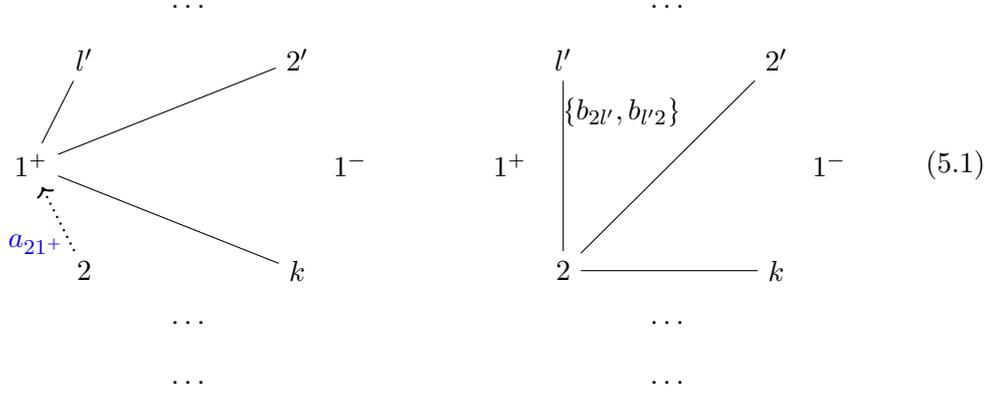
		\begin{align}\label{5.1}
			\begin{tikzpicture}[baseline=-2.5pt,scale=0.7]
				\node (1^+) at (-3,0) {$1^+$};
				\node (2) at (-2,-2) {$2$};
				\node (k) at (2,-2) {$k$};
				\node (s1) at (0,3) {$\cdots$};
				\node (1^-) at  (3,0) {$1^-$};
				\node (2') at (2,2) {$2'$};
				\node (s2) at (0,-3) {$\cdots$};
				\node (l') at  (-2,2) {$l'$};
				\draw [->, dotted, thick] (2) to (1^+);
				\draw (1^+) to (l'); 
				\draw (1^+) to (2');
				\draw (1^+) to (k);
				\draw[blue] (-2.9,-1.5)node{$a_{21^+}$};
				\begin{scope}[shift={(9,0)}]
					\node (1^+) at (-3,0) {$1^+$};
					\node (2) at (-2,-2) {$2$};
					\node (k) at (2,-2) {$k$};
					\node (s1) at (0,3) {$\cdots$};
					\node (1^-) at  (3,0) {$1^-$};
					\node (2') at (2,2) {$2'$};
					\node (s2) at (0,-3) {$\cdots$};
					\node (l') at  (-2,2) {$l'$};
					\draw (2) to (l'); 
					\draw (2) to (2');
					\draw (2) to (k);
					\node (s) at  (-.9,1) {$\{b_{2l'},b_{l'2} \}$};
				\end{scope}	
			\end{tikzpicture}
		\end{align}		
		
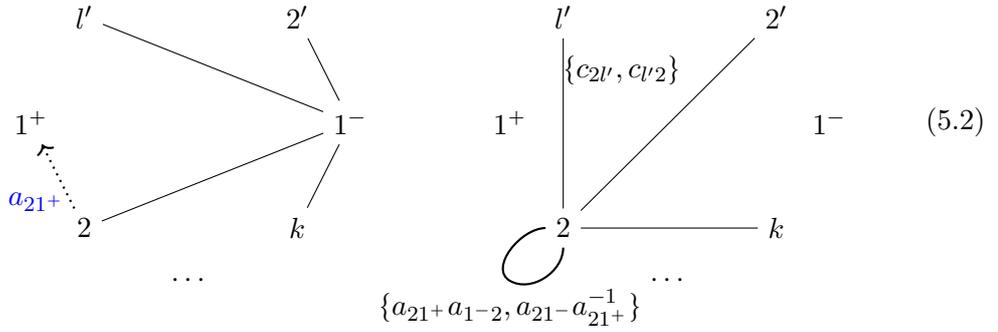
\begin{align}\label{5.2}
			\begin{tikzpicture}[baseline=-2.5pt,scale=0.7]
				\node (1^+) at (-3,0) {$1^+$};
				\node (2) at (-2,-2) {$2$};
				\node (k) at (2,-2) {$k$};
				\node (s1) at (0,3) {$\cdots$};
				\node (1^-) at  (3,0) {$1^-$};
				\node (2') at (2,2) {$2'$};
				\node (s2) at (0,-3) {$\cdots$};
				\node (l') at  (-2,2) {$l'$};
				\draw [->, dotted, thick] (2) to (1^+);
				\draw (1^-) to (l'); 
				\draw (1^-) to (2');
				\draw (1^-) to (k);
				\draw (1^-) to (2);
				\draw[blue] (-2.9,-1.5)node{$a_{21^+}$};
				\begin{scope}[shift={(9,0)}]
					\node (1^+) at (-3,0) {$1^+$};
					\node (2) at (-2,-2) {$2$};
					\node (k) at (2,-2) {$k$};
					\node (s1) at (0,3) {$\cdots$};
					\node (1^-) at  (3,0) {$1^-$};
					\node (2') at (2,2) {$2'$};
					\node (s2) at (0,-3) {$\cdots$};
					\node (l') at  (-2,2) {$l'$};
					\draw (2) to (l'); 
					\draw (2) to (2');
					\draw (2) to (k);
					\draw [-,  thick]  (2) to[out=-90,in=-30] (-3,-3) to[out=150,in=180](2) ;
					\node (s) at  (-.9,1) {$\{c_{2l'},c_{l'2} \}$};
					\node (t) at  (-3,-3.5) {$\{a_{21^+}a_{1^-2} ,a_{21^-}a_{21^+}^{-1}\}$};  
				\end{scope}	
			\end{tikzpicture}
		\end{align}		
			
		(\ref{5.1}) and (\ref{5.2}) are illustrations of the generation process of the arrow pairs  $\{b_{2t},b_{t2} \}_{t\in P} \cup \{c_{2t},c_{t2} \}_{t\in P}$,\ $ \{a_{21^+}a_{1^-2} ,a_{21^-}a_{21^+}^{-1}\}$. Notice that $\{a_{21^+}a_{1^-2} ,a_{21^-}a_{21^+}^{-1}\}$ are loops attached to $2$.
		
		Next, we give an illustration (see (\ref{5.3})) of the generation process of the arrow pairs 	$\{\alpha \varphi, \varphi^{op} \alpha^{-1}\}\cup \{\alpha \varphi^{op},\varphi \alpha^{-1}\}\cup \{\alpha \varphi \alpha^{-1},\alpha \varphi^{op} \alpha^{-1} \}$, where
		\begin{itemize}
			\item $\{\alpha \varphi, \varphi^{op} \alpha^{-1}\}\cup \{\alpha \varphi^{op},\varphi \alpha^{-1}\}$  are given by $\varphi, \varphi^{op}$ in Definition~\ref{def:mut QsP}~(3); and 
			\item $\{\alpha \varphi \alpha^{-1},\alpha \varphi^{op} \alpha^{-1} \}$ are given by $\varphi, \varphi^{op}$ in Definition~\ref{def:mut QsP}~(4).
		\end{itemize}
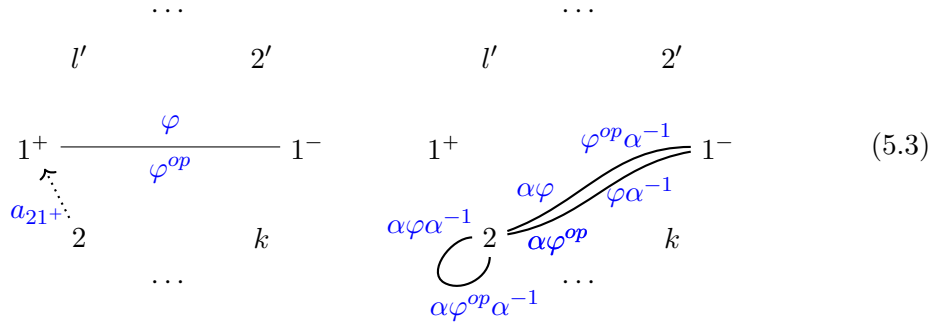
		\begin{align}\label{5.3}
			\begin{tikzpicture}[baseline=-2.5pt,scale=0.6]
				\node (1^+) at (-3,0) {$1^+$};
				\node (2) at (-2,-2) {$2$};
				\node (k) at (2,-2) {$k$};
				\node (s1) at (0,3) {$\cdots$};
				\node (1^-) at  (3,0) {$1^-$};
				\node (2') at (2,2) {$2'$};
				\node (s2) at (0,-3) {$\cdots$};
				\node (l') at  (-2,2) {$l'$};
				\draw [->, dotted, thick] (2) to (1^+);
				\draw (1^-) to (1^+); 
				\draw[blue] (-2.9,-1.5)node{$a_{21^+}$};
				\draw[blue] (0,-0.5)node{$\varphi ^{op}$};
				\draw[blue](0,0.5)node{$\varphi $};
				\begin{scope}[shift={(9,0)}]
					\node (1^+) at (-3,0) {$1^+$};
					\node (2) at (-2,-2) {$2$};
					\node (k) at (2,-2) {$k$};
					\node (s1) at (0,3) {$\cdots$};
					\node (1^-) at  (3,0) {$1^-$};
					\node (2') at (2,2) {$2'$};
					\node (s2) at (0,-3) {$\cdots$};
					\node (l') at  (-2,2) {$l'$};
					\draw [ thick]  (2) to[out=10,in=-170] (1^-) ;
					\draw [ thick]  (2) to[out=20,in=-180] (1^-) ;
					\draw [-,  thick]  (2) to[out=-90,in=-30] (-3,-3) to[out=150,in=180](2) ;
					\draw[blue] (-3.3,-1.7)node{$\alpha \varphi \alpha^{-1}$};
					\draw[blue] (-2.1,-3.5)node{$\alpha \varphi^{op} \alpha^{-1}$};
					\draw[blue] (-1,-1)node{$\alpha \varphi $};
					\draw[blue] (-.5,-2.1)node{$\alpha \varphi^{op} $};
					\draw[blue] (1,0.2)node{$ \varphi^{op}\alpha^{-1} $};
					\draw[blue] (-.5,-2.1)node{$\alpha \varphi^{op} $};
					\draw[blue] (1.3,-1)node{$ \varphi \alpha^{-1} $};
				\end{scope}	
			\end{tikzpicture}
		\end{align}	
		
		The arrow pair $\{a_{21^+},a_{1^+2}\}$ is replaced by the arrow pair $\{a_{21^+}^*,a_{1^+2}^*=(a_{21^+}^*)^{op}\}$ in the opposite direction.
		
		By the Definition~\ref{mutation}, we have to remove the following superfluous arrow pairs:
		\[
		\{a_{2j},a_{j2},b_{2j},b_{j2},a_{21^-},a_{1^-2},\alpha \varphi,\varphi^{op}\alpha^{-1}\}_{j\in P }\cup \{\alpha \varphi \alpha^{-1},\alpha \varphi^{op} \alpha^{-1}, a_{21^+}a_{1^-2} ,a_{21^-}a_{21^+}^{-1} \}. 
		\]
		We can view $\{a_{21^+}^*,a_{1^+2}^*\}$ as the arrow pair between $2$ and $1^-$, and $\{\alpha \varphi^{op},\varphi \alpha^{-1}\}$ as the arrow pair between $2$ and $1^+$. Then, we have the following local graph of $\mu^{'}_{1}(Q,W)$.		
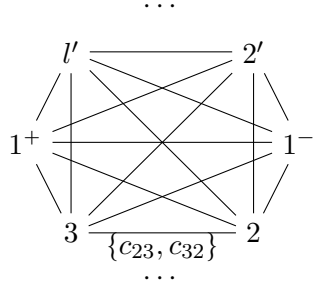
		\begin{align*}
			\begin{tikzpicture}[baseline=-2.5pt,scale=0.6]
				\node (1^+) at (-3,0) {$1^+$};
				\node (2) at (-2,-2) {$3$};
				\node (k) at (2,-2) {$2$};
				\node (s1) at (0,3) {$\cdots$};
				\node (1^-) at  (3,0) {$1^-$};
				\node (2') at (2,2) {$2'$};
				\node (s2) at (0,-3) {$\cdots$};
				\node (l') at  (-2,2) {$l'$};
				\draw [-] (1^+) to   (l') to (2') to  (1^-) to (k) to (2)to (1^+);
				\draw [-] (2) to   (l') to (1^-) to  (2) ;
				\draw [-] (2') to   (k) to (1^+) to  (2') ;
				\draw [-] (l')to(k) ;
				\draw [-] (2)to(2') ;
				\draw [-] (1^+)to(1^-) ;
				\node (t) at (0,-2.3) {$\{c_{23},c_{32} \}$};
			\end{tikzpicture}
		\end{align*}		   
		where vertices between $1^+$ and $1^-$ are $\{3,4,\cdots ,k,2 \}$ and $\{2',3',\cdots ,l' \}$. By the Step~4 in Definition~\ref{mutation}, we know that $\mu^{'}_{1}(Q,W)$ is a quiver with superpotential associated to 
		$(S,M,\mu_1 (D))$.	
		
		\item $|A|=2.$ Suppose $A=\{a_{21^+},a_{2'1^-} \}$. Then the discussion is similar to Case (2). For convenience, we just give the new arrow pairs and the superfluous arrow pairs.
		
		If the arc $1$ is not self-folded then the local graph is as follows. Let $S_1$ and $S_2$ be the regular graphs determined by point sets $M_1=\{1^+,2, \cdots , k-1,k\}$ and $M_2=\{1^-,2', \cdots , l'-1,l'\}$ respectively, see the left and right part of the following graph. 
		
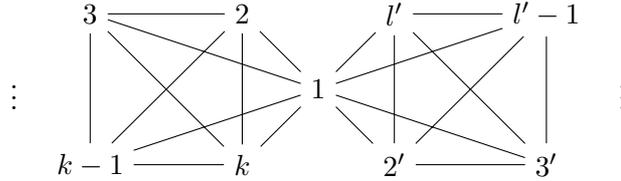
\begin{align*}
			\begin{tikzpicture}[baseline=-3.5pt][scale=0.8]
				\node (1) at (0,0) {$1$};
				\node (2) at (-1,1) {$2$};
				\node (3) at (-3,1) {$3$};
				\node (s1) at (-4,0) {$\vdots$};
				\node (k-1) at  (-3,-1) {$k-1$};
				\node (k) at  (-1,-1) {$k$};
				\node (2') at (1,-1) {$2'$};
				\node (3') at (3,-1) {$3'$};
				\node (s2) at (4,0) {$\vdots$};
				\node (l'-1) at  (3,1) {$l'-1$};
				\node (l') at  (1,1) {$l'$};
				\draw [-] (1) to   (2) to (3) to  (k-1) to (k) to (1);
				\draw [-] (2) to   (k-1) to (1) to  (3) to (k) to (2);
				\draw [-] (1) to   (2') to (3') to  (l'-1) to (l') to (1);
				\draw [-] (l'-1) to (1)to  (3') to (l') to (2')to (l'-1);
			\end{tikzpicture}
		\end{align*}		
		The new arrow pairs are $\{b_{2t},b_{t2} ,b_{2't},b_{t2'}\}_{t\in P}$, where $P=(M_1\cup M_2)\verb|\|\{1^+,1^-,2,2'\}$. To be more precisely,
		\begin{equation*}
			\begin{aligned}
				b_{2j}=[a_{21^+}a_{1^+j}]=a_{21^+}a_{1^+j} ,\ b_{j2}=[a_{j1^+}a_{21^+}^{-1}]=a_{j1^+}a_{21^+}^{-1},\ \forall j \in 	\{ 3,\cdots , k \},\\
				b_{2j}=[a_{21^+}a_{1^-j}]=a_{21^+}a_{1^-j} ,\ b_{j2}=[a_{j1^-}a_{21^+}^{-1}]=a_{j1^-}a_{21^+}^{-1},\ \forall j \in 	\{ 3',\cdots , l' \},\\
				b_{2'j}=[a_{2'1^-}a_{1^-j}]=a_{2'1^-}a_{1^-j} ,\ b_{j2'}=[a_{j1^-}a_{2'1^-}^{-1}]=a_{j1^-}a_{2'1^-}^{-1},\ \forall j \in 	\{ 3',\cdots , l' \},\\
				b_{2'j}=[a_{2'1^-}a_{1^+j}]=a_{2'1^-}a_{1^+j} ,\ b_{j2'}=[a_{j1^+}a_{2'1^-}^{-1}]=a_{j1^+}a_{2'1^-}^{-1},\ \forall j \in 	\{ 3,\cdots , k \}.
			\end{aligned}	
		\end{equation*}
		
		The arrow pair $\{a_{21^+},a_{1^+2}\}$ and $\{a_{2'1^-},a_{1^-2'}\}$ are replaced by arrow pairs $\{a_{21^+}^*,a_{1^+2}^*=(a_{21^+}^*)^{op}\}$ and $\{a_{2'1^-}^*,a_{1^-2'}^*=(a_{2'1^-}^*)^{op}\}$ in opposite directions, respectively. We can view $\{a_{21^+}^*,a_{1^+2}^*\}$ as the arrow pair between $2$ and $1^-$ in $\mu_{1}^{'}(Q)$, and $\{a_{2'1^-}^*,a_{1^-2'}^*\}$ as the arrow pair between $2'$ and $1^+$ in $\mu_{1}^{'}(Q)$.
		
		The superfluous arrow pairs:
		\[
		\{b_{2j},b_{j2},a_{2j},a_{j2} \}_{j\in \{3,\cdots,k\}} \text{ and } \{b_{2'j},b_{j2'},a_{2'j},a_{j2'} \}_{j\in \{3',\cdots,l'\}}.
		\]	
		The remaining discussion is similar to Case~(2) and we omit it.

		If the arc $1$ is self-folded then the local graph is as follows. Denote by $S_1=S_2$ the regular graph determined by the point set $M=\{1^+,1^-,2, \cdots , k,2',\cdots ,l'\}$.  For convenience, let $P=M\verb|\| \{1^+,1^-,2,2' \}$, $\alpha=a_{21^+},\beta=a_{2'1^-},\varphi=a_{1^+ 1^-},\varphi^{op}=a_{1^+ 1^-}^{op}=a_{1^- 1^+}$.
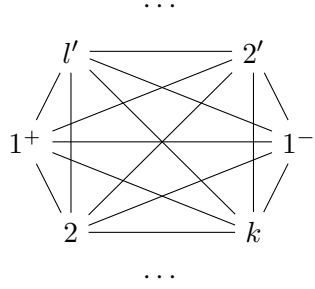
		\begin{align*}
			\begin{tikzpicture}[baseline=-2.5pt,scale=0.6]
				\node (1^+) at (-3,0) {$1^+$};
				\node (2) at (-2,-2) {$2$};
				\node (k) at (2,-2) {$k$};
				\node (s1) at (0,3) {$\cdots$};
				\node (1^-) at  (3,0) {$1^-$};
				\node (2') at (2,2) {$2'$};
				\node (s2) at (0,-3) {$\cdots$};
				\node (l') at  (-2,2) {$l'$};
				\draw [-] (1^+) to   (l') to (2') to  (1^-) to (k) to (2)to (1^+);
				\draw [-] (2) to   (l') to (1^-) to  (2) ;
				\draw [-] (2') to   (k) to (1^+) to  (2') ;
				\draw [-] (l')to(k) ;
				\draw [-] (2)to(2') ;
				\draw [-] (1^+)to(1^-) ;
			\end{tikzpicture}
		\end{align*}
		
		The new arrow pairs are 
		\[\{b_{2t},b_{t2} ,c_{2t},c_{t2},b_{2't},b_{t2'} ,c_{2't},c_{t2'}\}_{t\in P},\]
		\[ \{ a_{21^+}a_{1^+2'},a_{2'1^+}a_{21^+}^{-1},a_{2'1^-}a_{1^-2},a_{21^-}a_{2'1^-}^{-1}\},\]
		\[
		\{a_{21^+}a_{1^-2},a_{21^-}a_{21^+}^{-1},a_{2'1^-}a_{1^+2'},a_{2'1^+}a_{2'1^-}^{-1}  \},
		\] 
		\[
		\{ \alpha\varphi , \varphi^{op}\alpha^{-1},\alpha \varphi^{op},\varphi \alpha^{-1}\}, \{ \beta\varphi^{op} , \varphi\beta^{-1},\beta \varphi,\varphi^{op} \beta^{-1}\},	 
		\]
		\[
		\{\alpha \varphi \alpha^{-1},\alpha \varphi^{op} \alpha^{-1},\beta \varphi \beta^{-1},\beta \varphi^{op} \beta^{-1}	   
		\}
		\]
		\[
		\{\alpha \varphi \beta^{-1},\beta \varphi^{op} \alpha^{-1},\beta \varphi \alpha^{-1},\alpha \varphi^{op} \beta^{-1}	
		\},	 
		\]
		where 
		\begin{equation*}
			\begin{aligned}
				b_{2j}=[a_{21^+}a_{1^+j}]=a_{21^+}a_{1^+j} ,\ b_{j2}=[a_{j1^+}a_{21^+}^{-1}]=a_{j1^+}a_{21^+}^{-1},\ \forall j \in 	P,\\
				b_{2'j}=[a_{2'1^-}a_{1^-j}]=a_{2'1^-}a_{1^-j} ,\ b_{j2'}=[a_{j1^-}a_{2'1^+}^{-1}]=a_{j1^-}a_{2'1^-}^{-1},\ \forall j \in 	P,\\
				c_{2j}=[a_{21^+}a_{1^-j}]=a_{21^+}a_{1^-j} ,\ c_{j2}=[a_{j1^-}a_{21^+}^{-1}]=a_{j1^-}a_{21^+}^{-1},\ \forall j \in 	P,\\
				c_{2'j}=[a_{2'1^-}a_{1^+j}]=a_{2'1^-}a_{1^+j} ,\ c_{j2'}=[a_{j1^+}a_{2'1^-}^{-1}]=a_{j1^+}a_{2'1^-}^{-1},\ \forall j \in 	P.
			\end{aligned}	
		\end{equation*}
		
		The arrow pair $\{a_{21^+},a_{1^+2}\}$ and $\{a_{2'1^-},a_{1^-2'}\}$ are replaced by arrow pairs $\{a_{21^+}^*,a_{1^+2}^*=(a_{21^+}^*)^{op}\}$ and $\{a_{2'1^-}^*,a_{1^-2'}^*=(a_{2'1^-}^*)^{op}\}$ in opposite directions, respectively. 
		%$\{a_{2'1^-},a_{1^-2'}\}$ is replaced by the arrow pair $\{a_{2'1^-}^*,a_{1^-2'}^*=(a_{2'1^-}^*)^{op}\}$  in the opposite direction.  
		We can view $\{a_{21^+}^*,a_{1^+2}^*\}$ as the arrow pair between $2$ and $1^-$ in $\mu_{1}^{'}(Q)$, and $\{a_{2'1^-}^*,a_{1^-2'}^*\}$ the arrow pair between $2'$ and $1^+$ in $\mu_{1}^{'}(Q)$.
		
		The superfluous arrow pairs (we omit recurring ones):
		\[
		\{b_{2j},b_{j2},a_{2j},a_{j2}, b_{2'j},b_{j2'},a_{2'j},a_{j2'}\}_{j\in P},
		\]
		\[
		\{\alpha\varphi , \varphi^{op}\alpha^{-1},a_{21^-},a_{1^-2} \}, \{\beta\varphi^{op} , \varphi\beta^{-1},a_{2'1^+},a_{1^+2'} \}.
		\]	
		\[
		\{a_{21^+}a_{1^+2'},a_{22'}\},
		\{a_{2'1^-}a_{1^-2},a_{2'2}\}
		\]
		\[
		\{\alpha\varphi \beta^{-1}, \beta\varphi^{op} \alpha^{-1},
		a_{21^-}a_{2'1^-}^{-1},a_{2'1^+}a_{21^+}^{-1}
		\}
		\]
		\[
		\{
		\alpha \varphi \alpha^{-1},\alpha \varphi^{op} \alpha^{-1},a_{21^+}a_{1^-2},a_{21^-}a_{21^+}^{-1}
		\},\]
		\[\{\beta \varphi \beta^{-1},\beta \varphi^{op} \beta^{-1},a_{2'1^-}a_{1^+2'},a_{2'1^+}a_{2'1^-}^{-1}  \}.
		\]
		
		We can view $\{\alpha \varphi^{op},\varphi\alpha^{-1} \}$ as the arrow pair between $2,1^+$ in $\mu_{1}^{'}(Q)$, $\{\beta \varphi, \varphi^{op}\beta^{-1} \}$ as the arrow pair between $2',1^-$ in $\mu_{1}^{'}(Q)$, and $\{\alpha \varphi^{op} \beta^{-1}, \beta \varphi \alpha^{-1} \}$ as the arrow pair between $2,2'$ in $\mu_{1}^{'}(Q)$.

      The remaining discussion is similar to Case~(2) and we omit it.
	\end{enumerate}
\begin{remark}
    To complete the proof of Theorem~\ref{thm2.16}, we still have to check that $\{\mu^{'}_{i}(W),\mu^{'}_{i}(W)\}=0$ ($d^2=0$). However, such a condition is implied during the proof of Theorem~\ref{mainthm} below.
\end{remark}
\subsection{Proof of Theorem~\ref{mainthm}}
	Assume $i=1$ without loss generality.
	Since there are many arrow pairs, which might seem confusing, we view each arrow pair as an edge and only display those edges  associated to the two $d$-gons that contain $1$. We get two regular graphs corresponding to the two 
	$d$-gons containing arc $1$.
\subsubsection{Non-self-folded case}\

	First, consider the case when arc  $1$ is not  self-folded.  We define two regular graphs $S_1$ and $S_2$, and then glue $1^{+}$ and $1^{-}$ into $1$.
	Denote
	\begin{align*}
		A=\{\alpha \in Q_1 |\alpha:j \rga 1 ,  |\alpha|=0 \}.
	\end{align*}
	Thus $|A| \in \{0, 1 ,2\}$ and we have the following cases.
	\begin{enumerate}[\text{Case}~(1)]
		\item $|A|=0$. Then the assertion follows directly by Definition~\ref{def:mut QsP} and Definition~\ref{mutation}. 

         \item $|A|=1$.	Suppose $A=\{a_{21} \}$ and local graph is as follows. Let $S_1$ be the regular graph determined by the  point set $M_1=\{1,2, \cdots , k-1,k\}$, and $S_2$ the regular graph determined by $M_2=\{1,2', \cdots , l'-1,l'\}$.  
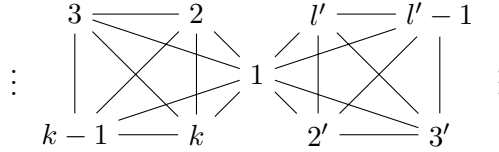
	\begin{align*}
		\begin{tikzpicture}[baseline=-3.5pt,scale=0.8]
			\node (1) at (0,0) {$1$};
			\node (2) at (-1,1) {$2$};
			\node (3) at (-3,1) {$3$};
			\node (s1) at (-4,0) {$\vdots$};
			\node (k-1) at  (-3,-1) {$k-1$};
			\node (k) at  (-1,-1) {$k$};
			\node (2') at (1,-1) {$2'$};
			\node (3') at (3,-1) {$3'$};
			\node (s2) at (4,0) {$\vdots$};
			\node (l'-1) at  (3,1) {$l'-1$};
			\node (l') at  (1,1) {$l'$};
			\draw [-] (1) to   (2) to (3) to  (k-1) to (k) to (1);
			\draw [-] (2) to   (k-1) to (1) to  (3) to (k) to (2);
			\draw [-] (1) to   (2') to (3') to  (l'-1) to (l') to (1);
			\draw [-] (l'-1) to (1)to  (3') to (l') to (2')to (l'-1);
		\end{tikzpicture}
	\end{align*}
	By the definition of $\Gamma_{1}$, the new arrow pairs are $\{b_{2t},b_{t2} \}_{t\in P}$, where $P=(M_1\cup M_2)\verb|\|\{1,2\}$. For convenience, we keep the notations as in Definition~\ref{def:mut QsP}.
	\begin{equation*}
		b_{2j}=[a_{21}a_{1j}]=a_{21}a_{1j} ,\ b_{j2}=[a_{j1}a_{21}^{-1}]=a_{j1}a_{21}^{-1},\ \forall j \in 	P.
	\end{equation*}
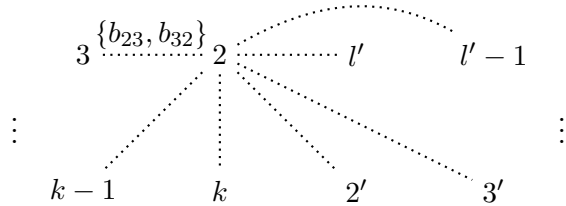
	\begin{align*}
		\begin{tikzpicture}[baseline=-3.5pt,scale=0.9]
			\node (2) at (-1,1) {$2$};
			\node (3) at (-3,1) {$3$};
			\node (s1) at (-4,0) {$\vdots$};
			\node (k-1) at  (-3,-1) {$k-1$};
			\node (k) at  (-1,-1) {$k$};
			\node (2') at (1,-1) {$2'$};
			\node (3') at (3,-1) {$3'$};
			\node (s2) at (4,0) {$\vdots$};
			\node (l'-1) at  (3,1) {$l'-1$};
			\node (l') at  (1,1) {$l'$};
			\draw [-, dotted, thick]  (2) to (3) ;
			\draw [-, dotted, thick]  (2) to (k) ;
			\draw [-, dotted, thick]  (2) to (k-1) ;
			\draw [-, dotted, thick]  (2) to (2') ;
			\draw [-, dotted, thick]    (2) to [out=30,in=150] (l'-1) ;
			\draw [-, dotted, thick]  (2) to (l') ;
			\draw [-, dotted, thick]  (2) to (3') ;
			\node (s) at (-2,1.3) {$\{b_{23},b_{32} \}$};
		\end{tikzpicture}
	\end{align*}	
	
	The arrow pair $\{a_{21},a_{12}\}$ is replaced by the arrow pair $\{a_{21}^*,a_{12}^*=(a_{21}^*)^{op}\}$ in the opposite direction.
	
	Since our aim is to obtain $\Gamma_2$ from $\Gamma_1$ by canceling certain arrow pairs via Lemma~\ref{lemma3.1}, we simplify the cancellation process by treating the signs of the path components in the differential in a fuzzy manner. Specifically, we denote  $d(\alpha)=\pm p_1  \cdots \pm p_m$ as  $d(\alpha)= p_1  +\cdots + p_m$, where the signs are omitted during computation.
	After completing the cancellation process, we can restore the actual signs (positive or negative) as needed.  
	All discussions henceforth will assume this simplified approach.
	
	Now we can compute the differential in $\Gamma_{1}$. Since
	\begin{equation*}
		d_1(a_{2k})=\dec (\red (d(a_{2k})))=\dec (\red (a_{21}a_{1k}))=a_{21}a_{1k}=b_{2k},
	\end{equation*}	
	we can cancel $\{a_{2k},b_{2k}\}$.
	
	Since
	\begin{equation*}
		\begin{aligned}
			d_1(a_{2,k-1})&=\dec (\red (d(a_{2,k-1})))\\ &=a_{21}a_{1,k-1}+a_{2k}a_{k,k-1}\\
			&=a_{21}a_{1,k-1}\\  &=b_{2,k-1},
		\end{aligned}	
	\end{equation*}	
	we cancel $\{a_{2,k-1},b_{2,k-1}\}$. 
	
	Similarly, we have
	\begin{equation*}
		\begin{aligned}
			d_1(a_{2,k-2})&=b_{2,k-2}\\ & \vdots \\
			d_1(a_{23})&=b_{23}.
		\end{aligned}	
	\end{equation*}
	So, we cancel $\{a_{2,k-2},b_{2,k-2}\},\cdots ,\{a_{23},b_{23}\}$. 
	
	For $\{b_{32},a_{32}\}$, we have
	\begin{equation*}
		\begin{aligned}
			d_1(b_{32})&=d_1(a_{31}a_{21}^{-1}) \\
			&=d_1(a_{31})a_{21}^{-1}+a_{31}a_{21}^*+\dec (d(a_{31})/ a_{21})\\
			&=\dec (\red(d(a_{31})) )a_{21}^{-1}+a_{31}a_{21}^*+a_{32} \\
			&=a_{31}a_{21}^*+a_{32}.
		\end{aligned}	
	\end{equation*}	
	So, $\{b_{32},a_{32}\}$ is canceled. Note that $a_{32}$ may also appears in the differential of other arrows, and one has to replace them by $a_{31}a_{21}^*$ after the cancellation.
	
	For $\{b_{42},a_{42}\}$, we have
	\begin{equation*}
		\begin{aligned}
			d_1(b_{42})&=d_1(a_{41}a_{21}^{-1}) \\
			&=d_1(a_{41})a_{21}^{-1}+a_{41}a_{21}^*+\dec (d(a_{41})/ a_{21})\\
			&=\dec (\red(d(a_{41})) )a_{21}^{-1}+a_{41}a_{21}^*+a_{42} \\
			&= a_{43}a_{31}a_{21}^{-1}  +a_{41}a_{21}^*+a_{42} \\
			&= a_{43}b_{32}+a_{41}a_{21}^*+a_{42} \\
			&=a_{41}a_{21}^*+a_{42}.
		\end{aligned}	
	\end{equation*}	
	So we cancel $\{b_{42},a_{42}\}$. Similar as before, all $a_{42}$ appearing in the differential of other arrows are replaced by $a_{41}a_{41}^*$ (after cancellation). We orderly cancel $\{b_{32},a_{32}\},\{b_{42},a_{42}\},\cdots ,\{b_{k-1,2},a_{k-1,2} \},\{b_{k,2},a_{k,2} \}$ and get relations $a_{j2}=\pm a_{j1}a_{21}^*, \ j\in\{3,4,\cdots , k-1,k \}$. 
	
	Above procedure cancels all the arrows that not in $\Gamma_{2}$. After which, we have the following local graph.
	
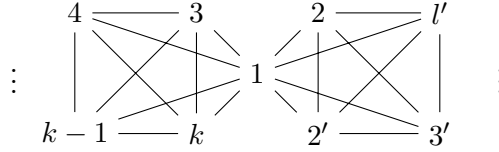
\begin{align*}
		\begin{tikzpicture}[baseline=-3.5pt,scale=0.8]
			\node (1) at (0,0) {$1$};
			\node (2) at (-1,1) {$3$};
			\node (3) at (-3,1) {$4$};
			\node (s1) at (-4,0) {$\vdots$};
			\node (k-1) at  (-3,-1) {$k-1$};
			\node (k) at  (-1,-1) {$k$};
			\node (2') at (1,-1) {$2'$};
			\node (3') at (3,-1) {$3'$};
			\node (s2) at (4,0) {$\vdots$};
			\node (l'-1) at  (3,1) {$l'$};
			\node (l') at  (1,1) {$2$};
			\draw [-] (1) to   (2) to (3) to  (k-1) to (k) to (1);
			\draw [-] (2) to   (k-1) to (1) to  (3) to (k) to (2);
			\draw [-] (1) to   (2') to (3') to  (l'-1) to (l') to (1);
			\draw [-] (l'-1) to (1)to  (3') to (l') to (2')to (l'-1);
		\end{tikzpicture}
	\end{align*}
	
	Now, the arrow pair between $1$ and $2$ is $a_{12}^*:2\rga 1$, $a_{21}^*:1\rga 2$, and the arrow pair between $2$ and $s$ is $\{b_{2s}$, $b_{s2}\}$,  for all $s\in \{2',3',\cdots ,l' \}$.  
	The arrow pairs between other points are unchanged. As for differential, we have
	\begin{equation*}
		\begin{aligned}
			d_1(b_{2'2})&=d_1(a_{2'1}a_{21}^{-1})\\ &=d_1(a_{2'1})a_{21}^{-1}+a_{2'1}a_{21}^*+\dec (d(a_{2'1})/ a_{21}) \\ &=a_{2'1}a_{21}^* \ ,
		\end{aligned}
	\end{equation*}
	and 
	\begin{equation*}
		\begin{aligned}
			d_1(b_{3'2})&=d_1(a_{3'1})a_{21}^{-1}+a_{a_{3'1}}a_{21}^*+\dec (d(a_{3'1})/ a_{21})\\&=a_{3'2'}a_{2'1}a_{21}^{-1}+a_{a_{3'1}}a_{21}^*\\&=a_{3'2'}b_{2'2}+a_{3'1}a_{21}^*.
		\end{aligned}
	\end{equation*}
	A similar discussion gives the following differentials
	\begin{equation*}
		\begin{aligned}
			d_1(b_{2'2})&=a_{2'1}a_{21}^*,\\
			d_1(b_{3'2})&=a_{3'2'}b_{2'2}+a_{3'1}a_{21}^*,\\
			&\vdots \\
			d_1(b_{l'2})&=a_{l',l'-1}b_{l'-1,2}+a_{l',l'-2}b_{l'-2,2}+\cdots +a_{l'2'}b_{2'2}+a_{3'1}a_{21}^*.
		\end{aligned}
	\end{equation*}	
	Now, let us calculate $d_1(b_{2,l'}),d_1(b_{2,l'-1}),\cdots ,d_1(b_{22'})$ :
	
	\begin{equation*}
		\begin{aligned}
			d_1(b_{2,l'})&=d_1(a_{21}a_{1,l'})=a_{21} \dec (\red (d(a_{1,l'})))=0 \\
			d_1(b_{2,l'-1})&=d_1(a_{21}a_{1,l'-1})=a_{21} \dec (\red (d(a_{1,l'-1}))) \\
			&=a_{21}d(a_{1,l'-1}) \\
			&=b_{2,l'}a_{l',l'-1}\\
			&\vdots \\
			d_1(b_{2,2'})&=  b_{2,l'}a_{l',2}+b_{2,l'-1}a_{l'-1,2}+\cdots +b_{2,3'}a_{3',2}.
		\end{aligned}
	\end{equation*}	
	None of the calculations above involve the relations that we derived earlier. As for $a_{12}^*$ and $a_{21}^*$, we have
	\begin{equation*}
		\begin{aligned}
			d_1(a_{21}^*)&=0 ,\\ 
			d_1(a_{12}^*)&= a_{21}(a_{13}a_{31}+\cdots +a_{1k}a_{k1}+a_{12'}a_{2'1}+\cdots +a_{1,l'}a_{l',1})\\
			&=b_{23}a_{31}+\cdots +b_{2k}a_{k1}+b_{22'}a_{2'1}+\cdots +b_{2,l'}a_{l',1}\\
			&=b_{2l'}a_{l'1}+\cdots +b_{2,2'}a_{2',1}.
		\end{aligned}
	\end{equation*}
	
	Similarly, differential of the other arrows in $\Gamma_{1}$ do not involve relations $a_{j2}=\pm a_{j1}a_{21}^*, \ j\in\{3,4,\cdots , k-1,k \}$. Thus, we get a dg quiver algebra $\Gamma_{1}'$.
	
    Since differentials involve no relations, we have the following quasi-isomorphism:
	$$\Gamma_1 \rga \Gamma_{1}'=\Gamma_{2}.$$
		
    \item $|A|=2$. Suppose $A=\{a_{21},a_{2'1} \}$ and the local graph is as follows. Let $S_1$ and $S_2$ be the regular graphs determined by point sets $M_1=\{1,2, \cdots , k-1,k\}$ and $M_2=\{1,2', \cdots , l'-1,l'\}$, respectively. 
	
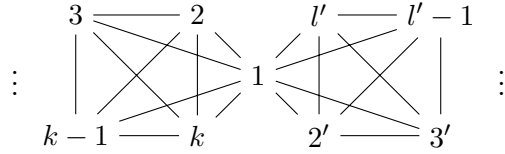
\begin{align*}
		\begin{tikzpicture}[baseline=-3.5pt,scale=0.8]
			\node (1) at (0,0) {$1$};
			\node (2) at (-1,1) {$2$};
			\node (3) at (-3,1) {$3$};
			\node (s1) at (-4,0) {$\vdots$};
			\node (k-1) at  (-3,-1) {$k-1$};
			\node (k) at  (-1,-1) {$k$};
			\node (2') at (1,-1) {$2'$};
			\node (3') at (3,-1) {$3'$};
			\node (s2) at (4,0) {$\vdots$};
			\node (l'-1) at  (3,1) {$l'-1$};
			\node (l') at  (1,1) {$l'$};
			\draw [-] (1) to   (2) to (3) to  (k-1) to (k) to (1);
			\draw [-] (2) to   (k-1) to (1) to  (3) to (k) to (2);
			\draw [-] (1) to   (2') to (3') to  (l'-1) to (l') to (1);
			\draw [-] (l'-1) to (1)to  (3') to (l') to (2')to (l'-1);
		\end{tikzpicture}
	\end{align*}
	By the definition of $\Gamma_{1}$, the new arrow pairs are $\{b_{2t},b_{t2} ,b_{2't},b_{t2'}\}_{t\in P}$, where $P=(M_1\cup M_2)\verb|\|\{1,2,2'\}$. Similar as Case~(2), we have% we maintain the notations as defined in Definition~\ref{def:mut QsP}:
	\begin{equation*}
		\begin{aligned}
			b_{2j}=[a_{21}a_{1j}]=a_{21}a_{1j} ,\ b_{j2}=[a_{j1}a_{21}^{-1}]=a_{j1}a_{21}^{-1},  \\
			b_{2'j}=[a_{2'1}a_{1j}]=a_{2'1}a_{1j} ,\ b_{j2'}=[a_{j1}a_{2'1}^{-1}]=a_{j1}a_{2'1}^{-1}, \ \forall j \in 	P.
		\end{aligned}
	\end{equation*}	
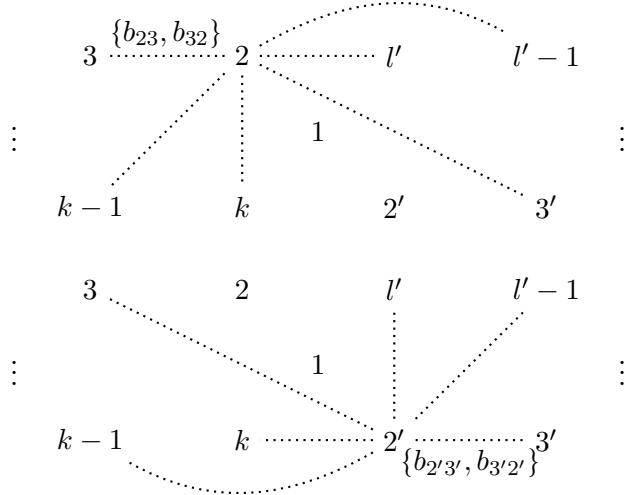
	\begin{align*}
		\begin{tikzpicture}[baseline=-3.5pt,scale=1]
			\node (1) at (0,0) {$1$};
			\node (2) at (-1,1) {$2$};
			\node (3) at (-3,1) {$3$};
			\node (s1) at (-4,0) {$\vdots$};
			\node (k-1) at  (-3,-1) {$k-1$};
			\node (k) at  (-1,-1) {$k$};
			\node (2') at (1,-1) {$2'$};
			\node (3') at (3,-1) {$3'$};
			\node (s2) at (4,0) {$\vdots$};
			\node (l'-1) at  (3,1) {$l'-1$};
			\node (l') at  (1,1) {$l'$};
			\draw [-, dotted, thick]  (2) to (3) ;
			\draw [-, dotted, thick]  (2) to (k) ;
			\draw [-, dotted, thick]  (2) to (k-1) ;
			\draw [-, dotted, thick]    (2) to [out=30,in=150] (l'-1) ;
			\draw [-, dotted, thick]  (2) to (l') ;
			\draw [-, dotted, thick]  (2) to (3') ;
			\node (s) at (-2,1.3) {$\{b_{23},b_{32} \}$};
		\end{tikzpicture}
	\end{align*}	
	\begin{align*}
		\begin{tikzpicture}[baseline=-3.5pt,scale=1]
			\node (1) at (0,0) {$1$};
			\node (2') at (-1,1) {$2$};
			\node (3) at (-3,1) {$3$};
			\node (s1) at (-4,0) {$\vdots$};
			\node (k-1) at  (-3,-1) {$k-1$};
			\node (k) at  (-1,-1) {$k$};
			\node (2) at (1,-1) {$2'$};
			\node (3') at (3,-1) {$3'$};
			\node (s2) at (4,0) {$\vdots$};
			\node (l'-1) at  (3,1) {$l'-1$};
			\node (l') at  (1,1) {$l'$};
			\draw [-, dotted, thick]  (2) to (3) ;
			\draw [-, dotted, thick]  (2) to (k) ;
			\draw [-, dotted, thick]  (2) to[out=-150,in=-30] (k-1) ;
			\draw [-, dotted, thick]    (2) to  (l'-1) ;
			\draw [-, dotted, thick]  (2) to (l') ;
			\draw [-, dotted, thick]  (2) to (3') ;
			\node (s) at (2,-1.3) {$\{b_{2'3'},b_{3'2'} \}$};
		\end{tikzpicture}
	\end{align*}	
	
	The arrow pair $\{a_{21},a_{12}\}$ (resp. $\{a_{2'1},a_{12'}\}$) is replaced by the arrow pair $\{a_{21}^*,a_{12}^*=(a_{21}^*)^{op}\}$ (resp. $\{a_{2'1}^*,a_{12}^*=(a_{2'1}^*)^{op}\}$) in the opposite direction.
	
	Similar to Case (2), we can cancel the arrows $\{a_{23},b_{23},\cdots ,a_{2,k} ,b_{2,k}\}$, $\{a_{32},b_{32},\cdots ,a_{k,2} ,b_{k,2}\}$,  $\{a_{2'3'},b_{2'3'},\cdots ,a_{2',l'} ,b_{2',l'}\}$ and $\{a_{3'2'},b_{3'2'},\cdots ,a_{l',2'} ,b_{l',2'}\}$. Thus, we get the following local graph, and hence the quasi-isomorphism $\Gamma_1 \rga \Gamma_{1}'=\Gamma_{2}$.
	
	
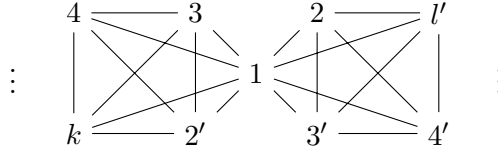
\begin{align*}
		\begin{tikzpicture}[baseline=-3.5pt,scale=0.8]
			\node (1) at (0,0) {$1$};
			\node (2) at (-1,1) {$3$};
			\node (3) at (-3,1) {$4$};
			\node (s1) at (-4,0) {$\vdots$};
			\node (k-1) at  (-3,-1) {$k$};
			\node (k) at  (-1,-1) {$2'$};
			\node (2') at (1,-1) {$3'$};
			\node (3') at (3,-1) {$4'$};
			\node (s2) at (4,0) {$\vdots$};
			\node (l'-1) at  (3,1) {$l'$};
			\node (l') at  (1,1) {$2$};
			\draw [-] (1) to   (2) to (3) to  (k-1) to (k) to (1);
			\draw [-] (2) to   (k-1) to (1) to  (3) to (k) to (2);
			\draw [-] (1) to   (2') to (3') to  (l'-1) to (l') to (1);
			\draw [-] (l'-1) to (1)to  (3') to (l') to (2')to (l'-1);
		\end{tikzpicture}
	\end{align*}
\end{enumerate}
	In summary, if $1$ is not a self-folded arc, the theorem holds.	
\subsubsection{Self-folded case}\

	We consider the case when $1$ is a self-folded arc. We have the regular graph containing $1^+$ and $1^-$ as before. Again denote by $A=\{\alpha \in Q_1 |\alpha:j \rga 1 ,  |\alpha|=0 \}$.	So $|A| \in \{0, 1 ,2\}$ and we have the following cases.
	\begin{enumerate}[\text{Case}~(1)]
	\item $|A|=0$. Then the assertion follows by Definition~\ref{def:mut QsP} and Definition~\ref{mutation}.
	
	\item $|A|=1$. Suppose $A=\{a_{21^+} \}$ and the local graph is as follows. Let $S$ be the regular graph determined by the point set $M=\{1^+,2, \cdots , k,1^-,2',\cdots ,l'\}$.  
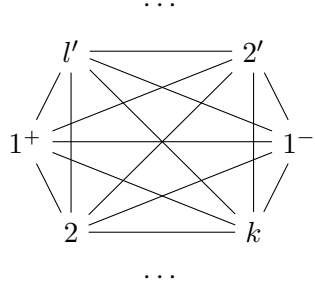
	\begin{align*}
		\begin{tikzpicture}[baseline=-2.5pt,scale=0.6]
			\node (1^+) at (-3,0) {$1^+$};
			\node (2) at (-2,-2) {$2$};
			\node (k) at (2,-2) {$k$};
			\node (s1) at (0,3) {$\cdots$};
			\node (1^-) at  (3,0) {$1^-$};
			\node (2') at (2,2) {$2'$};
			\node (s2) at (0,-3) {$\cdots$};
			\node (l') at  (-2,2) {$l'$};
			\draw [-] (1^+) to   (l') to (2') to  (1^-) to (k) to (2)to (1^+);
			\draw [-] (2) to   (l') to (1^-) to  (2) ;
			\draw [-] (2') to   (k) to (1^+) to  (2') ;
			\draw [-] (l')to(k) ;
			\draw [-] (2)to(2') ;
			\draw [-] (1^+)to(1^-) ;
		\end{tikzpicture}
	\end{align*}	
	 For convenience, we denote $\alpha=a_{21^+},\varphi=a_{1^+ 1^-},\varphi^{op}=a_{1^+ 1^-}^{op}=a_{1^- 1^+}$. By the definition of $\Gamma_{1}$, the new arrow pairs are  
	\[
	\{b_{2t},b_{t2} \}_{t\in P}\cup \{c_{2t},c_{t2} \}_{t\in P}\cup \{a_{21^+}a_{1^-2} ,a_{21^-}a_{21^+}^{-1}\},
	\]  
	\[ 
	\{\alpha \varphi \alpha^{-1},\alpha \varphi^{op} \alpha^{-1} \}\cup \{\alpha \varphi, \varphi^{op} \alpha^{-1}\}\cup \{\alpha \varphi^{op},\varphi \alpha^{-1}\},
	\] 
	where $P=M\verb|\|\{1^+,1^-,2\}$. We maintain the notations as in Definition~\ref{def:mut QsP}.
	\begin{equation*}
		\begin{aligned}
			b_{2j}=[a_{21^+}a_{1^+ j}]=a_{21^+}a_{1^+j} ,\ \ b_{j2}=[a_{j1^+}a_{21^+}^{-1}]=a_{j1^+}a_{21^+}^{-1},\ \forall j \in 	P,\\
			c_{2j}=[a_{21^+}a_{1^- j}]=a_{21^+}a_{1^-j} ,\ \ c_{j2}=[a_{j1^-}a_{21^+}^{-1}]=a_{j1^-}a_{21^+}^{-1},\ \forall j \in 	P.
		\end{aligned}	
	\end{equation*}
	\begin{align}\label{5.4}
		\begin{tikzpicture}[baseline=-2.5pt,scale=0.7]
			\node (1^+) at (-3,0) {$1^+$};
			\node (2) at (-2,-2) {$2$};
			\node (k) at (2,-2) {$k$};
			\node (s1) at (0,3) {$\cdots$};
			\node (1^-) at  (3,0) {$1^-$};
			\node (2') at (2,2) {$2'$};
			\node (s2) at (0,-3) {$\cdots$};
			\node (l') at  (-2,2) {$l'$};
			\draw [->, dotted, thick] (2) to (1^+);
			\draw (1^+) to (l'); 
			\draw (1^+) to (2');
			\draw (1^+) to (k);
			\draw[blue] (-2.9,-1.5)node{$a_{21^+}$};
			\begin{scope}[shift={(9,0)}]
				\node (1^+) at (-3,0) {$1^+$};
				\node (2) at (-2,-2) {$2$};
				\node (k) at (2,-2) {$k$};
				\node (s1) at (0,3) {$\cdots$};
				\node (1^-) at  (3,0) {$1^-$};
				\node (2') at (2,2) {$2'$};
				\node (s2) at (0,-3) {$\cdots$};
				\node (l') at  (-2,2) {$l'$};
				\draw (2) to (l'); 
				\draw (2) to (2');
				\draw (2) to (k);
				\node (s) at  (-.9,1) {$\{b_{2l'},b_{l'2} \}$};
			\end{scope}	
		\end{tikzpicture}
	\end{align}	
	\begin{align}\label{5.5}
		\begin{tikzpicture}[baseline=-2.5pt,scale=0.7]
			\node (1^+) at (-3,0) {$1^+$};
			\node (2) at (-2,-2) {$2$};
			\node (k) at (2,-2) {$k$};
			\node (s1) at (0,3) {$\cdots$};
			\node (1^-) at  (3,0) {$1^-$};
			\node (2') at (2,2) {$2'$};
			\node (s2) at (0,-3) {$\cdots$};
			\node (l') at  (-2,2) {$l'$};
			\draw [->, dotted, thick] (2) to (1^+);
			\draw (1^-) to (l'); 
			\draw (1^-) to (2');
			\draw (1^-) to (k);
			\draw (1^-) to (2);
			\draw[blue] (-2.9,-1.5)node{$a_{21^+}$};
			\begin{scope}[shift={(9,0)}]
				\node (1^+) at (-3,0) {$1^+$};
				\node (2) at (-2,-2) {$2$};
				\node (k) at (2,-2) {$k$};
				\node (s1) at (0,3) {$\cdots$};
				\node (1^-) at  (3,0) {$1^-$};
				\node (2') at (2,2) {$2'$};
				\node (s2) at (0,-3) {$\cdots$};
				\node (l') at  (-2,2) {$l'$};
				\draw (2) to (l'); 
				\draw (2) to (2');
				\draw (2) to (k);
				\draw [-,  thick]  (2) to[out=-90,in=-30] (-3,-3) to[out=150,in=180](2) ;
				\node (s) at  (-.9,1) {$\{c_{2l'},c_{l'2} \}$};
				\node (t) at  (-3,-3.5) {$\{a_{21^+}a_{1^-2} ,a_{21^-}a_{21^+}^{-1}\}$};  
			\end{scope}	
		\end{tikzpicture}
	\end{align}
			
	(\ref{5.4}) and (\ref{5.5}) are illustrations of the generation process of the arrow pairs  $\{b_{2t},b_{t2} \}_{t\in P}\cup\{c_{2t},c_{t2} \}_{t\in P}$,\ $ \{a_{21^+}a_{1^-2} ,a_{21^-}a_{21^+}^{-1}\}$. Notice that $a_{21^+}a_{1^-2} ,a_{21^-}a_{21^+}^{-1}$ are loops attached to $2$.
	
	Next, we give an illustration (see (\ref{5.6})) of the generation process of the arrow pairs 	$\{\alpha \varphi, \varphi^{op} \alpha^{-1}\}\cup \{\alpha \varphi^{op},\varphi \alpha^{-1}\} \cup \{\alpha \varphi \alpha^{-1},\alpha \varphi^{op} \alpha^{-1} \}$, where
	\begin{itemize}
		\item $\{\alpha \varphi, \varphi^{op} \alpha^{-1}\}\cup \{\alpha \varphi^{op},\varphi \alpha^{-1}\}$ is given by $\varphi, \varphi^{op}$ in Definition~\ref{def:mut QsP}~(3); and
		\item $\{\alpha \varphi \alpha^{-1},\alpha \varphi^{op} \alpha^{-1} \}$ is given by $\varphi, \varphi^{op}$ in Definition~\ref{def:mut QsP}~(4).
	\end{itemize}
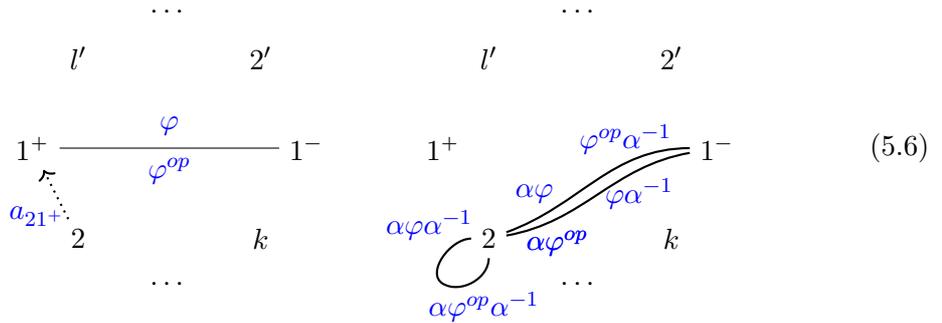
	\begin{align}\label{5.6}
		\begin{tikzpicture}[baseline=-2.5pt,scale=0.6]
			\node (1^+) at (-3,0) {$1^+$};
			\node (2) at (-2,-2) {$2$};
			\node (k) at (2,-2) {$k$};
			\node (s1) at (0,3) {$\cdots$};
			\node (1^-) at  (3,0) {$1^-$};
			\node (2') at (2,2) {$2'$};
			\node (s2) at (0,-3) {$\cdots$};
			\node (l') at  (-2,2) {$l'$};
			\draw [->, dotted, thick] (2) to (1^+);
			\draw (1^-) to (1^+); 
			\draw[blue] (-2.9,-1.5)node{$a_{21^+}$};
			\draw[blue] (0,-0.5)node{$\varphi ^{op}$};
			\draw[blue](0,0.5)node{$\varphi $};
			\begin{scope}[shift={(9,0)}]
				\node (1^+) at (-3,0) {$1^+$};
				\node (2) at (-2,-2) {$2$};
				\node (k) at (2,-2) {$k$};
				\node (s1) at (0,3) {$\cdots$};
				\node (1^-) at  (3,0) {$1^-$};
				\node (2') at (2,2) {$2'$};
				\node (s2) at (0,-3) {$\cdots$};
				\node (l') at  (-2,2) {$l'$};
				\draw [ thick]  (2) to[out=10,in=-170] (1^-) ;
				\draw [ thick]  (2) to[out=20,in=-180] (1^-) ;
				\draw [-,  thick]  (2) to[out=-90,in=-30] (-3,-3) to[out=150,in=180](2) ;
				\draw[blue] (-3.3,-1.7)node{$\alpha \varphi \alpha^{-1}$};
				\draw[blue] (-2.1,-3.5)node{$\alpha \varphi^{op} \alpha^{-1}$};
				\draw[blue] (-1,-1)node{$\alpha \varphi $};
				\draw[blue] (-.5,-2.1)node{$\alpha \varphi^{op} $};
				\draw[blue] (1,0.2)node{$ \varphi^{op}\alpha^{-1} $};
				\draw[blue] (-.5,-2.1)node{$\alpha \varphi^{op} $};
				\draw[blue] (1.3,-1)node{$ \varphi \alpha^{-1} $};
			\end{scope}	
		\end{tikzpicture}
	\end{align}	
	The arrow pair $\{a_{21^+},a_{1^+2}\}$ is replaced by the arrow pair $\{a_{21^+}^*,a_{1^+2}^*=(a_{21^+}^*)^{op}\}$  in the opposite direction. 
			
	Next, we compute the differential in $\Gamma_{1}$ to cancel arrow pairs.	Since
	\[
		\begin{aligned}
			d_1(a_{2,l'})&=\dec (\red (d(a_{2,l'})))=a_{2,1^+}a_{1^+,l'}=b_{2,l'},
		\end{aligned}	
	\]
	we cancel $\{a_{2,l'},b_{2,l'} \}$.
	
	Since
	\begin{equation*}
		\begin{aligned}
			d_1(a_{2,l'-1})&=\dec (\red (d(a_{2,l'-1})))\\ &=a_{21^+}a_{1^+,l'-1}+a_{2l'}a_{l',l'-1}\\
			&=b_{2,l'-1},
		\end{aligned}	
	\end{equation*}	
	we cancel $\{a_{2,l'-1},b_{2,l'-1} \}$. 
	
	By similar computations, we have
	\begin{equation*}
		\begin{aligned}
			d_1(a_{2,l'-2})&=b_{2,l'-2} ,\\ & \vdots \\
			d_1(a_{22'})&=b_{22'}.
		\end{aligned}	
	\end{equation*}
	So $\{a_{2,l'-2},b_{2,l'-2}\},\cdots ,\{a_{22'},b_{22'}\}$ are canceled. 
	
	Since
	\begin{equation*}
		\begin{aligned}
			d_1(a_{21^-})&=\dec (\red (d(a_{21^-})))\\ &=a_{21^+}a_{1^+ 1^-}+0\\
			&=\alpha \varphi ,
		\end{aligned}	
	\end{equation*}	
	we also cancel $\{a_{21^-},\alpha \varphi \}$.	
	
	Consider arrows $a_{21^-}a_{21^+}^{-1},\alpha \varphi \alpha^{-1}$, we have 
	\begin{equation*}
		\begin{aligned}
			d_1(a_{21^-}a_{21^+}^{-1})&=d_1(a_{21^-})a_{21^+}^{-1}+a_{21^-}a_{21^+}^*+\dec (d(a_{21^-})/a_{21^+})\\
			&=\alpha \varphi \alpha^{-1}+0+0\\
			&=\alpha \varphi \alpha^{-1}.
		\end{aligned}	
	\end{equation*}	
	So we cancel $\{a_{21^-}a_{21^+}^{-1},\alpha \varphi \alpha^{-1}  \}$. Since
	\begin{equation*}
		\begin{aligned}
			d_1(a_{2,k})&=\dec (\red (d(a_{2,k})))\\ &=a_{21^-}\Delta a_{1^-,k}+a_{21^+}a_{1^+k}\\
			&=a_{21^-}(1-a_{21^+}^{-1}a_{21^+}) a_{1^-,k}+b_{2,k} \\
			&=a_{21^-}a_{21^+}^{-1}a_{21^+}a_{1^-,k}+b_{2,k} \\
			&= (a_{21^-}a_{21^+}^{-1})c_{2,k}+b_{2,k} \\
			&=b_{2,k},
		\end{aligned}	
	\end{equation*}	
	we cancel $\{a_{2,k},b_{2,k}\}$.  By similar computations, we cancel $\{a_{2,k-1},b_{2,k-1}\},\cdots,\{a_{23},b_{23} \}$. 
	
	Since
	\begin{equation*}
		\begin{aligned}
			d_1(b_{32})&=d_1(a_{31^+}a_{21^+}^{-1}) \\
			&=d_1(a_{31^+})a_{21^+}^{-1}+a_{31^+}a_{21^+}^*+\dec (d(a_{31^+})/ a_{21^+})\\
			&=\dec (\red(d(a_{31^+})) )a_{21^+}^{-1}+a_{31^+}a_{21^+}^*+a_{32} \\
			&=a_{31^+}a_{21^+}^*+a_{32},
		\end{aligned}	
	\end{equation*}		
	we cancel $\{b_{32},a_{32}\}$. Similar as before, all $a_{32}$ that appear in the differential of other arrows are replaced by the relation $a_{32}=\pm a_{31^+}a_{21^+}^*$ after cancellation. 

	By similar computations, $\{b_{s2},a_{s2}\}, s\in\{3,\cdots ,k\}$ are canceled. Similarly, we replace all $a_{s2}$ that appear in the differential of other arrows by the relation $a_{s2}=\pm a_{s1^+}a_{21^+}^*$ after cancellation.	
	
	Since $\varphi^{op}=a_{1^-1^+}$, we have
	\begin{equation*}
		\begin{aligned}
			d_1(\alpha \varphi^{op} \alpha^{-1})&=d_1(\alpha \varphi^{op} )\alpha^{-1}+ \alpha \varphi^{op} \alpha ^*+\alpha \dec (d(\varphi^{op})/\alpha)\\
			&=\alpha \dec (\red (d (\varphi^{op})))\alpha^{-1}+\alpha \varphi^{op} \alpha ^*+\alpha a_{1^-2}\\
			&=\sum_{s=3}^{k}c_{2s}b_{s2}+\alpha \varphi^{op} a_{21^+}^*+ a_{21^+}a_{1^-2}\\
			&=\alpha \varphi^{op} a_{21^+}^*+ a_{21^+}a_{1^-2}.
		\end{aligned}	
	\end{equation*}		
	So we cancel $\{\alpha \varphi^{op} \alpha^{-1},a_{21^+}a_{1^-2}\}$. Similarly we replace $a_{21^+}a_{1^-2}$ by $a_{21^+}a_{1^-2}=\pm \alpha \varphi^{op} a_{21^+}^*$. 
	
	Since
	\begin{equation*}
		\begin{aligned}
			d_1( \varphi^{op} \alpha^{-1})&=d_1(\varphi^{op} )\alpha^{-1}+ \varphi^{op} \alpha ^*+\dec (d(\varphi^{op})/\alpha)\\
			&=\dec (d(\varphi^{op}))\alpha^{-1} + \alpha ^*\alpha \varphi^{op} \alpha^{-1}  + \varphi^{op} \alpha ^*  + a_{1^-2}\\
			&=\sum_{s=3}^{k}a_{1^-s}b_{s2}  +\varphi^{op} \alpha ^*  + a_{1^-2} \\
			&=\varphi^{op} \alpha ^*  + a_{1^-2},
		\end{aligned}	
	\end{equation*}		
	we cancel $\{\varphi^{op} \alpha^{-1},a_{1^-2}\}$. Similarly we replace $a_{1^-2}$ by the relation $a_{1^-2}=\pm \varphi^{op} \alpha ^*=\pm a_{1^- 1^+}a_{21^+}^*$ in the differential of other arrows. 	
	
	Since
	\begin{equation*}
		\begin{aligned}
			d_1(b_{2'2})&=d_1(a_{2'1^+}a_{21^+}^{-1}) \\
			&=d_1(a_{2'1^+})a_{21^+}^{-1}+a_{2'1^+}a_{21^+}^*+\dec (d(a_{2'1^+})/ a_{21^+})\\
			&=\dec (\red(d(a_{2'1^+})) )a_{21^+}^{-1}+a_{2'1^+}a_{21^+}^*+a_{2'2} \\
			&=\dec ( a_{2'1^-}\varphi^{op} +\sum_{s=3}^{k}a_{2's}a_{s1^+})a_{21^+}^{-1}+      a_{2'1^+}a_{21^+}^*+a_{2'2}\\
			&=a_{2'1^-}\Delta  \varphi^{op}a_{21^+}^{-1}+\sum_{s=3}^{k}a_{2's}b_{s2}+a_{2'1^+}a_{21^+}^*+a_{2'2}\\
			&= a_{2'1^-}\varphi^{op} \alpha^{-1}+c_{2'2} \alpha\varphi^{op} \alpha^{-1}+a_{2'1^+}a_{21^+}^*+a_{2'2}\\
			&=a_{2'1^+}a_{21^+}^*+a_{2'2},
		\end{aligned}	
	\end{equation*}		
	we cancel $\{b_{2'2} ,a_{2'2}\}$ and replace $a_{2'2}$ by the relation $a_{2'2}=\pm a_{2'1^+}a_{21^+}^*$. 
	
	Similar computations as before, we cancel all $\{b_{s2} ,a_{s2}\}, s \in \{3',\cdots ,l' \}$, and replace $a_{s2}$ (that appear in the differential of other arrows) by the relation $a_{s2}=\pm a_{s1^+}a_{21^+}^*$.
	
	Finally, we have canceled 
	$\{a_{2k},a_{k2},b_{2k},b_{k2}\}_{k\in P}$, and got relations $a_{k2}=\pm a_{k1^+}a_{21^+}^*$, for all $k\in P$. In particular, 
	\begin{itemize}
		\item for arrows between $2$ and $1^-$: $\{a_{21^-},a_{1^-2},\alpha \varphi,\varphi^{op}\alpha^{-1}, \alpha \varphi^{op}, \varphi\alpha^{-1}\}$, we have canceled $\{  a_{21^-},a_{1^-2},\alpha \varphi,\varphi^{op}\alpha^{-1}\}$, and got relations $a_{1^-2}=\pm a_{1^- 1^+}a_{21^+}^*$;
		
		\item for loops attached on $2$: $\{\alpha \varphi \alpha^{-1},\alpha \varphi^{op} \alpha^{-1}, a_{21^+}a_{1^-2} ,a_{21^-}a_{21^+}^{-1} \}$, we have canceled them all and got relations $a_{21^+}a_{1^-2}=\pm \alpha \varphi^{op} a_{21^+}^*$.
	\end{itemize}
	In fact, the differential of arrows does not involve above relations by definition of superpotential.
	
	We can view $\{a_{21^+}^*,a_{1^+2}^*\}$ as the arrow pair between $2$ and $1^-$, and $\{\alpha \varphi^{op},\varphi \alpha^{-1}\}$ the arrow pair between $2$ and $1^+$. We have successfully canceled arrows that do not appear in $\Gamma_{2}$. Thus, we have the following local graph.
	
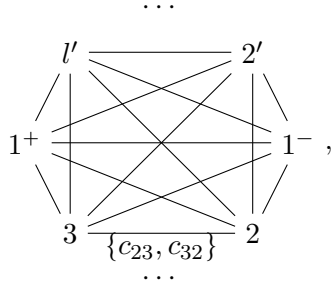
\begin{align*}
		\begin{tikzpicture}[baseline=-2.5pt,scale=0.6]
			\node (1^+) at (-3,0) {$1^+$};
			\node (2) at (-2,-2) {$3$};
			\node (k) at (2,-2) {$2$};
			\node (s1) at (0,3) {$\cdots$};
			\node (1^-) at  (3,0) {$1^-$};
			\node (2') at (2,2) {$2'$};
			\node (s2) at (0,-3) {$\cdots$};
			\node (l') at  (-2,2) {$l'$};
			\draw [-] (1^+) to   (l') to (2') to  (1^-) to (k) to (2)to (1^+);
			\draw [-] (2) to   (l') to (1^-) to  (2) ;
			\draw [-] (2') to   (k) to (1^+) to  (2') ;
			\draw [-] (l')to(k) ;
			\draw [-] (2)to(2') ;
			\draw [-] (1^+)to(1^-) ;
			\node (t) at (0,-2.3) {$\{c_{23},c_{32} \}$};
		\end{tikzpicture},
	\end{align*}	
	where vertices between $1^+$ and $1^-$ are $\{3,4,\cdots ,k,2 \}$ and $\{2',3',\cdots ,l' \}$.	
	
	 Since differentials involves no relations, it follows that we have the following quasi-isomorphism:
	$$\Gamma_1 \rga \Gamma_{1}'=\Gamma_{2}.$$	
	\item $|A|=2$. The assertion follows by a similar discussion as in Case~(2).
\end{enumerate}
	In summary, if $1$ is a self-folded arc, the theorem holds as well.
\section{Remarks}
The main results of the paper are contained in Bo Le's Master thesis \cite{le} written in December 2023 at Tsinghua University. We have recently noted that Lucie Jacquet-Malo \cite{jacquet-malo_construction_2024} 
has done a similar work. This section is dedicated to highlighting  the differences between the two papers, particularly regarding the key definition: the quiver with superpotential $(Q,W)$ associated to a $d$-angulation $(S,M,D)$.

For a given $(S,M,D)$, we associate a different quiver with superpotential $(Q,W)$  associated to it.  To illustrate this, we use the example from Jacquet-Malo's article (see  \cite[Figure~8]{jacquet-malo_construction_2024})  to provide a visual comparison.
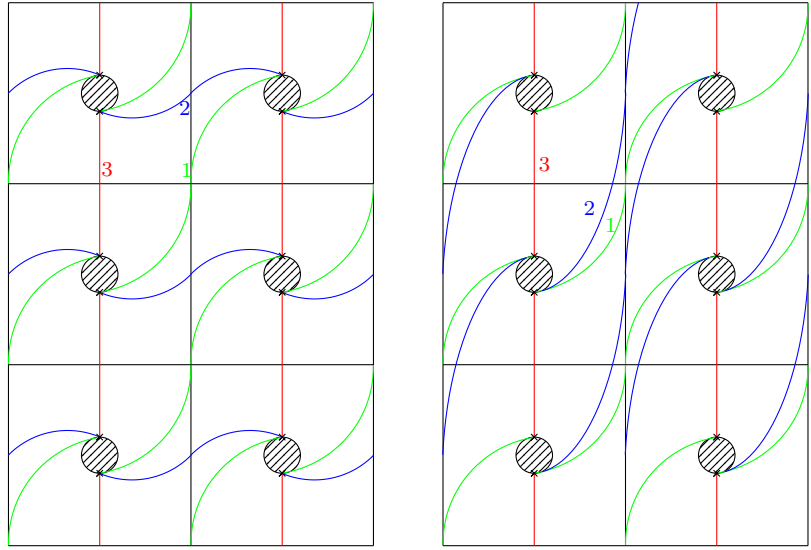
\begin{center}
	\begin{tikzpicture}[scale=0.8]
		\draw (-3,4)-- (-3,-5);
		\draw (0,4)-- (0,-5);
		\draw (3,-5)-- (3,4);
		\draw (-3,4)-- (3,4);
		\draw (-3,-5)-- (3,-5);
		\draw (-3,1)-- (3,1);
		\draw (-3,-2)-- (3,-2);
		\draw [fill=black,pattern=north east lines] (-1.5,-0.5) circle (0.3cm);
		\draw [fill=black,pattern=north east lines] (-1.5,-3.5) circle (0.3cm);
		\draw [fill=black,pattern=north east lines] (-1.5,2.5) circle (0.3cm);
		\draw [fill=black,pattern=north east lines] (1.5,2.5) circle (0.3cm);
		\draw [fill=black,pattern=north east lines] (1.5,-0.5) circle (0.3cm);
		\draw [fill=black,pattern=north east lines] (1.5,-3.5) circle (0.3cm);
		\draw [shift={(-0.97,0.45)},color=blue]  plot[domain=4.31:5.51,variable=\t]({1*1.36*cos(\t r)+0*1.36*sin(\t r)},{0*1.36*cos(\t r)+1*1.36*sin(\t r)});
		\draw [shift={(0.97,-1.45)},color=blue]  plot[domain=1.17:2.37,variable=\t]({1*1.36*cos(\t r)+0*1.36*sin(\t r)},{0*1.36*cos(\t r)+1*1.36*sin(\t r)});
		\draw [shift={(-0.97,3.45)},color=blue]  plot[domain=4.31:5.51,variable=\t]({1*1.36*cos(\t r)+0*1.36*sin(\t r)},{0*1.36*cos(\t r)+1*1.36*sin(\t r)});
		\draw [shift={(-0.97,-2.55)},color=blue]  plot[domain=4.31:5.51,variable=\t]({1*1.36*cos(\t r)+0*1.36*sin(\t r)},{0*1.36*cos(\t r)+1*1.36*sin(\t r)});
		\draw [shift={(0.97,1.55)},color=blue]  plot[domain=1.17:2.37,variable=\t]({1*1.36*cos(\t r)+0*1.36*sin(\t r)},{0*1.36*cos(\t r)+1*1.36*sin(\t r)});
		\draw [shift={(0.97,-4.45)},color=blue]  plot[domain=1.17:2.37,variable=\t]({1*1.36*cos(\t r)+0*1.36*sin(\t r)},{0*1.36*cos(\t r)+1*1.36*sin(\t r)});
		\draw [shift={(2.03,3.45)},color=blue]  plot[domain=4.31:5.51,variable=\t]({1*1.36*cos(\t r)+0*1.36*sin(\t r)},{0*1.36*cos(\t r)+1*1.36*sin(\t r)});
		\draw [shift={(2.03,0.45)},color=blue]  plot[domain=4.31:5.51,variable=\t]({1*1.36*cos(\t r)+0*1.36*sin(\t r)},{0*1.36*cos(\t r)+1*1.36*sin(\t r)});
		\draw [shift={(2.03,-2.55)},color=blue]  plot[domain=4.31:5.51,variable=\t]({1*1.36*cos(\t r)+0*1.36*sin(\t r)},{0*1.36*cos(\t r)+1*1.36*sin(\t r)});
		\draw [shift={(-2.03,-4.45)},color=blue]  plot[domain=1.17:2.37,variable=\t]({1*1.36*cos(\t r)+0*1.36*sin(\t r)},{0*1.36*cos(\t r)+1*1.36*sin(\t r)});
		\draw [shift={(-2.03,-1.45)},color=blue]  plot[domain=1.17:2.37,variable=\t]({1*1.36*cos(\t r)+0*1.36*sin(\t r)},{0*1.36*cos(\t r)+1*1.36*sin(\t r)});
		\draw [shift={(-2.03,1.55)},color=blue]  plot[domain=1.17:2.37,variable=\t]({1*1.36*cos(\t r)+0*1.36*sin(\t r)},{0*1.36*cos(\t r)+1*1.36*sin(\t r)});
		\draw [color=red] (-1.5,2.8)-- (-1.5,4);
		\draw [color=red] (-1.5,2.2)-- (-1.5,-0.2);
		\draw [color=red] (-1.5,-0.8)-- (-1.5,-3.2);
		\draw [color=red] (-1.5,-3.8)-- (-1.5,-5);
		\draw [color=red] (1.5,2.8)-- (1.5,4);
		\draw [color=red] (1.5,2.2)-- (1.5,-0.2);
		\draw [color=red] (1.5,-0.8)-- (1.5,-3.2);
		\draw [color=red] (1.5,-3.8)-- (1.5,-5);
		\draw [shift={(-1.78,0.96)},color=green]  plot[domain=-1.41:0.02,variable=\t]({1*1.78*cos(\t r)+0*1.78*sin(\t r)},{0*1.78*cos(\t r)+1*1.78*sin(\t r)});
		\draw [shift={(1.78,1.04)},color=green]  plot[domain=1.73:3.16,variable=\t]({1*1.78*cos(\t r)+0*1.78*sin(\t r)},{0*1.78*cos(\t r)+1*1.78*sin(\t r)});
		\draw [shift={(-1.78,3.96)},color=green]  plot[domain=-1.41:0.02,variable=\t]({1*1.78*cos(\t r)+0*1.78*sin(\t r)},{0*1.78*cos(\t r)+1*1.78*sin(\t r)});
		\draw [shift={(-1.78,-2.04)},color=green]  plot[domain=-1.41:0.02,variable=\t]({1*1.78*cos(\t r)+0*1.78*sin(\t r)},{0*1.78*cos(\t r)+1*1.78*sin(\t r)});
		\draw [shift={(1.78,-1.96)},color=green]  plot[domain=1.73:3.16,variable=\t]({1*1.78*cos(\t r)+0*1.78*sin(\t r)},{0*1.78*cos(\t r)+1*1.78*sin(\t r)});
		\draw [shift={(1.78,-4.96)},color=green]  plot[domain=1.73:3.16,variable=\t]({1*1.78*cos(\t r)+0*1.78*sin(\t r)},{0*1.78*cos(\t r)+1*1.78*sin(\t r)});
		\draw [shift={(-1.22,1.04)},color=green]  plot[domain=1.73:3.16,variable=\t]({1*1.78*cos(\t r)+0*1.78*sin(\t r)},{0*1.78*cos(\t r)+1*1.78*sin(\t r)});
		\draw [shift={(1.22,3.96)},color=green]  plot[domain=-1.41:0.02,variable=\t]({1*1.78*cos(\t r)+0*1.78*sin(\t r)},{0*1.78*cos(\t r)+1*1.78*sin(\t r)});
		\draw [shift={(1.22,0.96)},color=green]  plot[domain=-1.41:0.02,variable=\t]({1*1.78*cos(\t r)+0*1.78*sin(\t r)},{0*1.78*cos(\t r)+1*1.78*sin(\t r)});
		\draw [shift={(-1.22,-1.96)},color=green]  plot[domain=1.73:3.16,variable=\t]({1*1.78*cos(\t r)+0*1.78*sin(\t r)},{0*1.78*cos(\t r)+1*1.78*sin(\t r)});
		\draw [shift={(1.22,-2.04)},color=green]  plot[domain=-1.41:0.02,variable=\t]({1*1.78*cos(\t r)+0*1.78*sin(\t r)},{0*1.78*cos(\t r)+1*1.78*sin(\t r)});
		\draw [shift={(-1.22,-4.96)},color=green]  plot[domain=1.73:3.16,variable=\t]({1*1.78*cos(\t r)+0*1.78*sin(\t r)},{0*1.78*cos(\t r)+1*1.78*sin(\t r)});
		\begin{scriptsize}
			\draw [color=black] (-1.5,2.8)-- ++(-1.5pt,-1.5pt) -- ++(3.0pt,3.0pt) ++(-3.0pt,0) -- ++(3.0pt,-3.0pt);
			\draw [color=black] (-1.5,2.2)-- ++(-1.5pt,-1.5pt) -- ++(3.0pt,3.0pt) ++(-3.0pt,0) -- ++(3.0pt,-3.0pt);
			\draw [color=black] (-1.5,-0.2)-- ++(-1.5pt,-1.5pt) -- ++(3.0pt,3.0pt) ++(-3.0pt,0) -- ++(3.0pt,-3.0pt);
			\draw [color=black] (-1.5,-0.8)-- ++(-1.5pt,-1.5pt) -- ++(3.0pt,3.0pt) ++(-3.0pt,0) -- ++(3.0pt,-3.0pt);
			\draw [color=black] (-1.5,-3.2)-- ++(-1.5pt,-1.5pt) -- ++(3.0pt,3.0pt) ++(-3.0pt,0) -- ++(3.0pt,-3.0pt);
			\draw [color=black] (-1.5,-3.8)-- ++(-1.5pt,-1.5pt) -- ++(3.0pt,3.0pt) ++(-3.0pt,0) -- ++(3.0pt,-3.0pt);
			\draw [color=black] (1.5,2.8)-- ++(-1.5pt,-1.5pt) -- ++(3.0pt,3.0pt) ++(-3.0pt,0) -- ++(3.0pt,-3.0pt);
			\draw [color=black] (1.5,2.2)-- ++(-1.5pt,-1.5pt) -- ++(3.0pt,3.0pt) ++(-3.0pt,0) -- ++(3.0pt,-3.0pt);
			\draw [color=black] (1.5,-0.8)-- ++(-1.5pt,-1.5pt) -- ++(3.0pt,3.0pt) ++(-3.0pt,0) -- ++(3.0pt,-3.0pt);
			\draw [color=black] (1.5,-3.8)-- ++(-1.5pt,-1.5pt) -- ++(3.0pt,3.0pt) ++(-3.0pt,0) -- ++(3.0pt,-3.0pt);
			\draw [color=black] (1.5,-0.2)-- ++(-1.5pt,-1.5pt) -- ++(3.0pt,3.0pt) ++(-3.0pt,0) -- ++(3.0pt,-3.0pt);
			\draw [color=black] (1.5,-3.2)-- ++(-1.5pt,-1.5pt) -- ++(3.0pt,3.0pt) ++(-3.0pt,0) -- ++(3.0pt,-3.0pt);
			\draw[color=blue] (-0.1,2.26) node {$2$};
			\draw[color=red] (-1.38,1.24) node {$3$};
			\draw[color=green] (-0.07,1.22) node {$1$};
		\end{scriptsize}
	\end{tikzpicture}
	{\ \ \ \ \ }
	\begin{tikzpicture}[scale=0.8]
		\draw (-3,4)-- (-3,-5);
		\draw (0,4)-- (0,-5);
		\draw (3,-5)-- (3,4);
		\draw (-3,4)-- (3,4);
		\draw (-3,-5)-- (3,-5);
		\draw (-3,1)-- (3,1);
		\draw (-3,-2)-- (3,-2);
		\draw [fill=black,pattern=north east lines] (-1.5,-0.5) circle (0.3cm);
		\draw [fill=black,pattern=north east lines] (-1.5,-3.5) circle (0.3cm);
		\draw [fill=black,pattern=north east lines] (-1.5,2.5) circle (0.3cm);
		\draw [fill=black,pattern=north east lines] (1.5,2.5) circle (0.3cm);
		\draw [fill=black,pattern=north east lines] (1.5,-0.5) circle (0.3cm);
		\draw [fill=black,pattern=north east lines] (1.5,-3.5) circle (0.3cm);
		\draw [color=red] (-1.5,2.8)-- (-1.5,4);
		\draw [color=red] (-1.5,2.2)-- (-1.5,-0.2);
		\draw [color=red] (-1.5,-0.8)-- (-1.5,-3.2);
		\draw [color=red] (-1.5,-3.8)-- (-1.5,-5);
		\draw [color=red] (1.5,2.8)-- (1.5,4);
		\draw [color=red] (1.5,2.2)-- (1.5,-0.2);
		\draw [color=red] (1.5,-0.8)-- (1.5,-3.2);
		\draw [color=red] (1.5,-3.8)-- (1.5,-5);
		\draw [shift={(-1.78,0.96)},color=green]  plot[domain=-1.41:0.02,variable=\t]({1*1.78*cos(\t r)+0*1.78*sin(\t r)},{0*1.78*cos(\t r)+1*1.78*sin(\t r)});
		\draw [shift={(1.78,1.04)},color=green]  plot[domain=1.73:3.16,variable=\t]({1*1.78*cos(\t r)+0*1.78*sin(\t r)},{0*1.78*cos(\t r)+1*1.78*sin(\t r)});
		\draw [shift={(-1.78,3.96)},color=green]  plot[domain=-1.41:0.02,variable=\t]({1*1.78*cos(\t r)+0*1.78*sin(\t r)},{0*1.78*cos(\t r)+1*1.78*sin(\t r)});
		\draw [shift={(-1.78,-2.04)},color=green]  plot[domain=-1.41:0.02,variable=\t]({1*1.78*cos(\t r)+0*1.78*sin(\t r)},{0*1.78*cos(\t r)+1*1.78*sin(\t r)});
		\draw [shift={(1.78,-1.96)},color=green]  plot[domain=1.73:3.16,variable=\t]({1*1.78*cos(\t r)+0*1.78*sin(\t r)},{0*1.78*cos(\t r)+1*1.78*sin(\t r)});
		\draw [shift={(1.78,-4.96)},color=green]  plot[domain=1.73:3.16,variable=\t]({1*1.78*cos(\t r)+0*1.78*sin(\t r)},{0*1.78*cos(\t r)+1*1.78*sin(\t r)});
		\draw [shift={(-1.22,1.04)},color=green]  plot[domain=1.73:3.16,variable=\t]({1*1.78*cos(\t r)+0*1.78*sin(\t r)},{0*1.78*cos(\t r)+1*1.78*sin(\t r)});
		\draw [shift={(1.22,3.96)},color=green]  plot[domain=-1.41:0.02,variable=\t]({1*1.78*cos(\t r)+0*1.78*sin(\t r)},{0*1.78*cos(\t r)+1*1.78*sin(\t r)});
		\draw [shift={(1.22,0.96)},color=green]  plot[domain=-1.41:0.02,variable=\t]({1*1.78*cos(\t r)+0*1.78*sin(\t r)},{0*1.78*cos(\t r)+1*1.78*sin(\t r)});
		\draw [shift={(-1.22,-1.96)},color=green]  plot[domain=1.73:3.16,variable=\t]({1*1.78*cos(\t r)+0*1.78*sin(\t r)},{0*1.78*cos(\t r)+1*1.78*sin(\t r)});
		\draw [shift={(1.22,-2.04)},color=green]  plot[domain=-1.41:0.02,variable=\t]({1*1.78*cos(\t r)+0*1.78*sin(\t r)},{0*1.78*cos(\t r)+1*1.78*sin(\t r)});
		\draw [shift={(-1.22,-4.96)},color=green]  plot[domain=1.73:3.16,variable=\t]({1*1.78*cos(\t r)+0*1.78*sin(\t r)},{0*1.78*cos(\t r)+1*1.78*sin(\t r)});
		\draw [shift={(-1.46,-0.32)},color=blue]  plot[domain=0.13:1.53,variable=\t]({-0.03*3.48*cos(\t r)+1*1.47*sin(\t r)},{-1*3.48*cos(\t r)+-0.03*1.47*sin(\t r)});
		\draw [shift={(1.46,-0.68)},color=blue]  plot[domain=0.13:1.53,variable=\t]({0.03*3.48*cos(\t r)+-1*1.47*sin(\t r)},{1*3.48*cos(\t r)+0.03*1.47*sin(\t r)});
		\draw [shift={(-1.46,2.68)},color=blue]  plot[domain=0.13:1.53,variable=\t]({-0.03*3.48*cos(\t r)+1*1.47*sin(\t r)},{-1*3.48*cos(\t r)+-0.03*1.47*sin(\t r)});
		\draw [shift={(1.46,-3.68)},color=blue]  plot[domain=0.13:1.53,variable=\t]({0.03*3.48*cos(\t r)+-1*1.47*sin(\t r)},{1*3.48*cos(\t r)+0.03*1.47*sin(\t r)});
		\draw [shift={(1.54,2.68)},color=blue]  plot[domain=0.13:1.53,variable=\t]({-0.03*3.48*cos(\t r)+1*1.47*sin(\t r)},{-1*3.48*cos(\t r)+-0.03*1.47*sin(\t r)});
		\draw [shift={(1.54,-0.32)},color=blue]  plot[domain=0.13:1.53,variable=\t]({-0.03*3.48*cos(\t r)+1*1.47*sin(\t r)},{-1*3.48*cos(\t r)+-0.03*1.47*sin(\t r)});
		\draw [shift={(-1.54,-3.68)},color=blue]  plot[domain=0.13:1.53,variable=\t]({0.03*3.48*cos(\t r)+-1*1.47*sin(\t r)},{1*3.48*cos(\t r)+0.03*1.47*sin(\t r)});
		\draw [shift={(-1.54,-0.68)},color=blue]  plot[domain=0.13:1.53,variable=\t]({0.03*3.48*cos(\t r)+-1*1.47*sin(\t r)},{1*3.48*cos(\t r)+0.03*1.47*sin(\t r)});
		\draw [shift={(1.46,2.32)},color=blue]  plot[domain=4.22:4.67,variable=\t]({-0.03*3.48*cos(\t r)+1*1.47*sin(\t r)},{-1*3.48*cos(\t r)+-0.03*1.47*sin(\t r)});
		\begin{scriptsize}
			\draw [color=black] (-1.5,2.8)-- ++(-1.5pt,-1.5pt) -- ++(3.0pt,3.0pt) ++(-3.0pt,0) -- ++(3.0pt,-3.0pt);
			\draw [color=black] (-1.5,2.2)-- ++(-1.5pt,-1.5pt) -- ++(3.0pt,3.0pt) ++(-3.0pt,0) -- ++(3.0pt,-3.0pt);
			\draw [color=black] (-1.5,-0.2)-- ++(-1.5pt,-1.5pt) -- ++(3.0pt,3.0pt) ++(-3.0pt,0) -- ++(3.0pt,-3.0pt);
			\draw [color=black] (-1.5,-0.8)-- ++(-1.5pt,-1.5pt) -- ++(3.0pt,3.0pt) ++(-3.0pt,0) -- ++(3.0pt,-3.0pt);
			\draw [color=black] (-1.5,-3.2)-- ++(-1.5pt,-1.5pt) -- ++(3.0pt,3.0pt) ++(-3.0pt,0) -- ++(3.0pt,-3.0pt);
			\draw [color=black] (-1.5,-3.8)-- ++(-1.5pt,-1.5pt) -- ++(3.0pt,3.0pt) ++(-3.0pt,0) -- ++(3.0pt,-3.0pt);
			\draw [color=black] (1.5,2.8)-- ++(-1.5pt,-1.5pt) -- ++(3.0pt,3.0pt) ++(-3.0pt,0) -- ++(3.0pt,-3.0pt);
			\draw [color=black] (1.5,2.2)-- ++(-1.5pt,-1.5pt) -- ++(3.0pt,3.0pt) ++(-3.0pt,0) -- ++(3.0pt,-3.0pt);
			\draw [color=black] (1.5,-0.8)-- ++(-1.5pt,-1.5pt) -- ++(3.0pt,3.0pt) ++(-3.0pt,0) -- ++(3.0pt,-3.0pt);
			\draw [color=black] (1.5,-3.8)-- ++(-1.5pt,-1.5pt) -- ++(3.0pt,3.0pt) ++(-3.0pt,0) -- ++(3.0pt,-3.0pt);
			\draw[color=red] (-1.33,1.33) node {$3$};
			\draw [color=black] (1.5,-0.2)-- ++(-1.5pt,-1.5pt) -- ++(3.0pt,3.0pt) ++(-3.0pt,0) -- ++(3.0pt,-3.0pt);
			\draw [color=black] (1.5,-3.2)-- ++(-1.5pt,-1.5pt) -- ++(3.0pt,3.0pt) ++(-3.0pt,0) -- ++(3.0pt,-3.0pt);
			\draw[color=green] (-0.24,0.32) node {$1$};
			\draw[color=blue] (-0.59,0.59) node {$2$};
		\end{scriptsize}
	\end{tikzpicture}
\end{center}

By the Definition~\ref{def:QW}, we have $Q$ as follows:
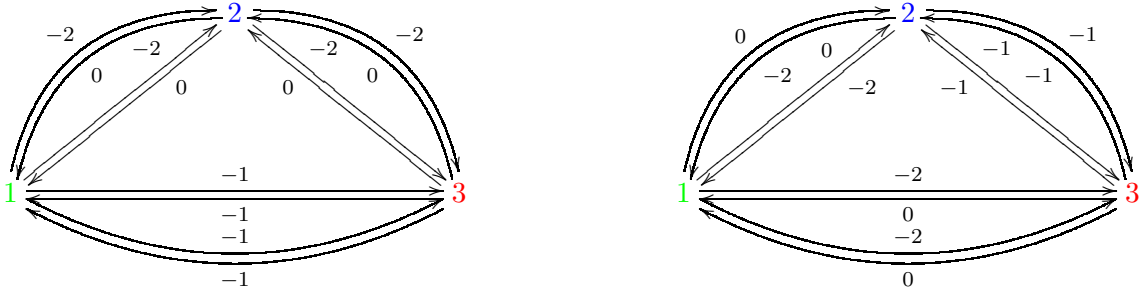
\begin{align*}
	\xymatrix@C=6pc{ & \ar@<0.3ex>^(0.3){0}[ldd]\ar@<0.3ex>^(0.3){-2}[rdd] \textcolor{blue}{2} \ar@<0.4ex>^{-2}@/^2pc/[rdd] \ar@<0.4ex>^{0}@/_2pc/[ldd] &  &    &  \ar@<0.3ex>^(0.3){-2}[ldd]\ar@<0.3ex>^(0.3){-1}[rdd] \textcolor{blue}{2} \ar@<0.4ex>^{-1}@/^2pc/[rdd] \ar@<0.4ex>^{-2}@/_2pc/[ldd] & 
		\\
		&&&&&
		\\
		\textcolor{green}{1} \ar@<0.3ex>^{-1}[rr] \ar@<0.3ex>^(0.7){-2}[ruu] \ar@<0.4ex>^{-1}@/_2pc/[rr] \ar@<0.4ex>^{-2}@/^2pc/[ruu] &  &\textcolor{red}{3}  \ar@<0.3ex>^{-1}[ll]\ar@<0.3ex>^(0.7){0}[luu] \ar@<0.4ex>^{-1}@/^2pc/[ll] \ar@<0.4ex>^{0}@/_2pc/[luu]  &\textcolor{green}{1} \ar@<0.3ex>^{-2}[rr] \ar@<0.3ex>^(0.7){0}[ruu] \ar@<0.4ex>^{-2}@/_2pc/[rr] \ar@<0.4ex>^{0}@/^2pc/[ruu] &  &\textcolor{red}{3}  \ar@<0.3ex>^{0}[ll]\ar@<0.3ex>^(0.7){-1}[luu] \ar@<0.4ex>^{0}@/^2pc/[ll] \ar@<0.4ex>^{-1}@/_2pc/[luu] }.
\end{align*}
The  graded quiver $Q$ is given in \cite{jacquet-malo_construction_2024} as follows:
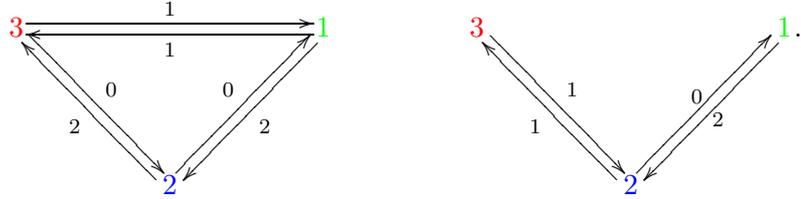
\begin{align*}
		\xymatrix@1 {
			\textcolor{red}{3} \ar@<2pt>^1[rrrr] \ar@<2pt>^0[ddrr] & & & & \textcolor{green}{1} \ar@<2pt>^1[llll] \ar@<2pt>^2[ddll] & &
			\textcolor{red}{3}  \ar@<2pt>^1[ddrr] & & & & \textcolor{green}{1} \ar@<2pt>_0[ddll]
			 \\
			\\
			& & \textcolor{blue}{2} \ar@<2pt>^2[uull] \ar@<2pt>^0[uurr] & & & & & & \textcolor{blue}{2} \ar@<2pt>^1[uull] \ar@<2pt>_2[uurr] & & \\
	        }.
\end{align*}

It is obvious that the two definitions of graded quivers are  different. One important point to note is that there are at most four arrows from vertex $i$ to vertex $j$ when both $i$ and $j$ correspond to self-folded arcs within the same $d$-gon. This differs from  Proposition 3.3 in \cite{jacquet-malo_construction_2024}. Finally, we provide an example to illustrate the difference between the two definitions of the superpotential.
\begin{example}
	Consider $(S,M,D)$ given in Figure~\ref{fig13}, there are $4$ arcs $\{1,2,3,4\}$ and $4$ pentagons while only $G_1,G_2$ contribute arrow pairs. We denote by $\{a_{12},a_{21},a_{32},a_{23},a_{13},a_{31}\}$ the arrow  pairs given in $G_1$, $\{b_{12},b_{21},b_{42},b_{24},b_{14},b_{41}\}$ the arrow  pairs given in $G_2$. Then all $3$-cycles of degree $-2$ are $a_{13}a_{32}a_{21},\  b_{12}b_{24}b_{41},\ a_{13}a_{32}b_{21}$ and $a_{12}b_{24}b_{41}$. By Definition~\ref{def:QW}, $W= a_{13}a_{32}a_{21}+ b_{12}b_{24}b_{41}$  is not the sum of all the $3$-cycles of degree $-2$, which is different from the superpotential given in  \cite[Definition~3.1]{jacquet-malo_construction_2024}.
		\begin{figure}[htpb]
		\centering
		\begin{tikzpicture}[scale=0.7]
			\draw[ultra thick](0,0) circle (3);
			\draw[ultra thick](0,0) circle (1);
			\draw   (3,0)\rn (0,3)\rn (-3,0)\rn (0,-3)\rn (1,0)\rn (0,1)\rn (-1,0)\rn  (2,2.2)\rn (-2,2.2)\rn (-2,-2.2)\rn (2,-2.2)\rn (-2.5,-1.6)\rn ; 
			\draw[red, thick](3,0)to(1,0); 
			\draw[red, thick](-3,0)to(-1,0); 
			\draw[red, thick](-3,0)to[out=40,in=180](0,2)to[out=0,in=140](3,0); 
			\draw[red, thick](-3,0)to[out=-40,in=180](2,-2.2); 
			\draw[blue] (-2,0.2)node{$1$} (2,0.2)node{$2$} (0,1.8)node{$3$} (0,-2.2)node{$4$} (-1,1)node{$G_1$}  (1.5,-1)node{$G_2$};
		  \end{tikzpicture}
		\caption{}\label{fig13}
	   \end{figure}
\end{example}
%\bibliographystyle{unsrt}
%\bibliographystyle{alpha}
%\bibliography{refs}

\textbf{Bo Le}\\
E-mail: \textsf{bole\_math@163.com}\\[0.3cm]
\textbf{Bin Zhu}\\
Department of Mathematical Sciences, Tsinghua University, Beijing, 100084, People's Republic of
China.\\
E-mail: \textsf{zhu-b@mail.tsinghua.edu.cn}
\end{document}